\documentclass[a4paper,12pt]{article}
\usepackage[utf8]{inputenc}
\usepackage[T1]{fontenc}
\usepackage[english]{babel}
\usepackage{amsmath, amsthm, amssymb}
\usepackage{mathrsfs}%fourier muunnoksen fontti
\usepackage{enumerate}
\usepackage{tikz-cd} %commutative diagrams
\usepackage{tikz}
\usetikzlibrary{patterns}
\usepackage[nottoc]{tocbibind} %add bibliography to table of contents, use option: numbib for number
\usepackage{imakeidx} %multiple indexes, eg: subjects, notation
      \theoremstyle{plain}
      \newtheorem{theorem}{Theorem}[subsection]
      \newtheorem*{theorem*}{Theorem}
      \newtheorem{lemma}[theorem]{Lemma}
      \newtheorem*{lemma*}{Lemma}
      \newtheorem{corollary}[theorem]{Corollary}
      \newtheorem*{corollary*}{Corollary}
      \newtheorem{theoremSect}{Theorem}[section]
      \newtheorem{lemmaSect}[theoremSect]{Lemma}
      \newtheorem{corollarySect}[theoremSect]{Corollary}
      
      \theoremstyle{definition}
      \newtheorem{definition}[theorem]{Definition}
      \newtheorem*{definition*}{Definition}

      \theoremstyle{remark}
      \newtheorem{remark}[theorem]{Remark}
      \newtheorem*{remark*}{Remark}
      \newtheorem{example}[theorem]{Example}
      \newtheorem*{example*}{Example}
      \newtheorem{remarkSect}[theoremSect]{Remark}
      
\DeclareMathOperator{\supp}{supp}

\newcommand{\wbar}[1]{\overline{#1}}
\newcommand{\what}[1]{\widehat{#1}}
\newcommand{\F}{\mathscr{F}}
\newcommand{\Ca}{\mathscr{C}}
\newcommand{\Cab}{\wbar{\mathscr{C}}}
\newcommand{\C}{\mathbb{C}}
\newcommand{\R}{\mathbb{R}}
\newcommand{\N}{\mathbb{N}}
\newcommand{\Tr}{\operatorname{Tr}}
\renewcommand{\Re}{\operatorname{Re}}%{\mathfrak{Re}}}
\renewcommand{\Im}{\operatorname{Im}}%{\mathfrak{Im}}}

\renewcommand{\O}{\Omega}
\renewcommand{\d}{\partial}
\newcommand{\db}{\wbar{\partial}}
\newcommand{\abs}[1]{\left\lvert #1 \right\rvert}
\newcommand{\norm}[1]{\left\lVert #1 \right\rVert}
\def\clap#1{\hbox to 0pt{\hss#1\hss}}

\def\mathrlap{\mathpalette\mathrlapinternal}

\def\mathrlapinternal#1#2{\rlap{$\mathsurround=0pt#1{#2}$}}

\makeindex[title=Index of Subjects, columnsep=15pt, intoc]
\makeindex[name=notation, title=Index of Notation, columns=1, intoc]

\numberwithin{equation}{section}

\title{On the Gel'fand-Calder\'on inverse problem in two dimensions }
\author{Eemeli Bl{\aa}sten}
\date{}

\begin{document}
\maketitle

\vspace{\stretch{1}}
\begin{center}
\textit{Academic dissertation\\ \medskip To be presented, with the permission of the \\Faculty of Science of the University of Helsinki, for public examination \\in Auditorium CK112, Exactum, on April 30, 2013, at 10 o'clock.}\\
\vspace{1em}
Helsinki 2013
\end{center}
\vspace{\stretch{2}}

\vfill
\begin{flushright}
\smallskip University of Helsinki\\ Faculty of Science \\ Department of Mathematics and Statistics \\
\medskip Supervisor Lassi P\"aiv\"arinta
\end{flushright}
\thispagestyle{empty}

\newpage
\thispagestyle{empty}
\begin{tabular}{ll}
Supervisor: & Professor Lassi P\"aiv\"arinta\\
& Department of Mathematics and Statistics\\
& University of Helsinki\\
& Helsinki, Finland\\
&\\
&\\
Pre-examiners: &Professor Giovanni Alessandrini\\
& Dipartimento di Matematica e Informatica\\
& Universit\`a di Trieste\\
& Trieste, Italia\\
&\\
& Professor Alberto Ruiz\\
& Departamento de Matem\'aticas\\
& Universidad Autonoma de Madrid\\
& Madrid, Espa\~na\\
&\\
&\\
Opponent: & Professor Yaroslav Kurylev\\
& Department of Mathematics\\
& University College London\\
& London, United Kingdom
\end{tabular}

\vfill
\begin{flushleft}
ISBN 978-952-10-8698-4 (paperback)\\
ISBN 978-952-10-8699-1 (PDF)\\
\smallskip
Unigrafia\\
E-thesis (http://www.e-thesis.fi)\\
\smallskip
Helsinki 2013
\end{flushleft}

\newpage
\thispagestyle{empty}
\begin{center}
{\large Acknowledgments}
\end{center}

\bigskip
First and foremost I would like to thank my PhD advisor Lassi P\"aiv\"arinta for introducing me to Bukhgeim's original article and having a good enough intuition to see that I was going the right way.  I wish to thank Alberto Ruiz and Giovanni Alessandrini for their precious time spent reading the manuscript very carefully and giving excellent suggestions.

I owe my officemate Esa Vesalainen and numerous other colleagues a lot for the many invaluable discussions during the past years. Without them, I would be stuck in my own way of thinking. A special thanks to Pedro Caro for suggesting me to use boundary data instead of assuming well-posedness.

I also want to thank my wife Wang Ruiling for her patience and our wonderful time together. She said that no matter what problems I will encounter, I would find a way around them.

My special thanks to the Finnish Inverse Problems Society for arranging the yearly Inverse Problems Days. They are a wonderful workshop to meet other Finnish researchers and appreciate how close the community is. It made me very happy about starting graduate studies in this field.

Lastly, I thank the Academy of Finland for its financial support through the Finnish Center of Excellence in Inverse Problems Research.

\newpage
\thispagestyle{empty}
\tableofcontents

\vfill
\section{Introduction}
\setcounter{page}{0}
\subsection{Abstract}
We prove uniqueness and stability for the inverse boundary value problem of the 2D Schr\"odinger equation. We assume only that the potentials are in $H^{s,(2,1)}(\O)$, $s>0$, which is slightly smaller than the Sobolev space $H^{s,2}(\O)$. The thesis consists of two parts.

In the first part, we define the spaces $H^{s,(p,q)}$ of distributions whose fractional derivatives are in the Lorentz space $L^{(p,q)}$. We prove the embedding $H^{1,(n,1)} \hookrightarrow C^0$ and an interpolation identity.

The inverse problem is considered in the second part of the thesis. We prove a new Carleman estimate for $\db$. This estimate has a decay rate of $\tau^{-1} \ln \tau$. After that we use Bukhgeim's oscillating exponential solutions, Alessandrini's identity and stationary phase to get information about the difference of the potentials from the difference of the Cauchy data.

\subsection{History and related work}
This short survey of results concerning inverse boundary value problems for the conductivity and Schr\"odinger equations is based mostly on introductions in \cite{astalaPaivarinta} and \cite{nachman2}. We mention also a few papers from recent years that we have personally heard of. The majority of the results cited below were proven for the conductivity equation or the Schr\"odinger equation having a potential coming from a related conductivity equation.

The inverse problem of the Schr\"odinger equation, also known as the Gel'fand or Gel'fand-Calder\'on inverse problem \index{inverse problem!Gel'fand-Calder\'on} (see \cite{gelfand}), is the following one:
\begin{equation}
\text{Given } C_q = \{ (u_{|\d\O}, \d_\nu u_{|\d\O}) \mid \Delta u + q u = 0 \} \text{ deduce } q.
\end{equation}
In other words, given measurements of the solutions $u$ only on the boundary \index{boundary data} $\d\O$ of an object or area $\O$, what can we say about the potential $q$ inside of $\O$? The Schr\"odinger equation can model acoustic, electromagnetic and quantum waves. Hence this inverse problem models inverse scattering of time harmonic waves in these situations.

One of the important early papers on inverse boundary value problems is by Calder\'on \cite{calderonProb}. He considered an isotropic body $\Omega$ from which one would like to deduce the electrical conductivity $\gamma$ by doing electrical measurements on the boundary. If we keep the voltage $u$ fixed as $f$ on the boundary, then the stationary state of $u$ can be modeled by the boundary value problem
\begin{equation}
\index{equation!conductivity}
\begin{split}
\nabla \cdot (\gamma \nabla u) &= 0, \quad \Omega\\
u &= f, \quad \partial\Omega .
\end{split}
\end{equation}
The weighted normal derivative $\gamma \partial_\nu u$ is the current flux going out of $\Omega$. Calder\'on asked whether knowing the boundary measurements, or Dirichlet-Neumann map $\Lambda_\gamma : f \mapsto \gamma \partial_\nu u_{|\partial\Omega}$, is enough to determine the conductivity $\gamma$ inside the whole domain $\Omega$. This is called the \emph{Calder\'on problem}. \index{inverse problem!Calder\'on} He showed the injectivity of a linearized problem near $\gamma \equiv 1$.

The inverse problem for the conductivity equation can be reduced to that of the Schr\"odinger equation. To transform the conductivity equation into the equation $\Delta v + q v = 0$, it is enough to do the change of variables $u = \gamma^{-\frac{1}{2}} v$, $q = -\gamma^{-\frac{1}{2}} \Delta \gamma^{\frac{1}{2}}$. The Dirichlet-Neumann map for the new equation can be recovered from the boundary data of the old one: $\Lambda_q = \gamma^{-\frac{1}{2}}\big( \Lambda_\gamma + \frac{1}{2}\frac{\partial \gamma}{\partial \nu} \big) \gamma^{-\frac{1}{2}}$.

Sylvester and Uhlmann solved the problem in dimensions $n$ at least three 
for smooth conductivities bounded away from zero \cite{sylvesterUhlmann}. They constructed \emph{complex geometric optics solutions}\index{solutions!complex geometric optics|textbf}, that is, solutions of the form
\begin{equation}
u_j = e^{x\cdot \,\zeta_j}\big( 1 + O(\frac{1}{|\zeta_j|}) \big),
\end{equation}
where the complex vectors $\zeta_j$ satisfy
\begin{equation}
\label{history_G_U_vectors}
\begin{split}
\zeta_1 &= i(k+m) + l,\\
\zeta_2 &= i(k - m) - l,
\end{split}
\end{equation}
where $l,k,m \in \R^n$ are perpendicular vectors satisfying $|l|^2 = |k|^2 + |m|^2$. Using a well-known orthogonality relation for the potentials $q_1$ and $q_2$, called the Alessandrini identity \cite{alessandrini}\index{Alessandrini's identity}, they got
\begin{equation}
0 = \int (q_1 - q_2) u_1 u_2 dx= \int (q_1 - q_2) e^{2i x \cdot k} \big( 1 + O(\frac{1}{|m|})\big) dx,
\end{equation}
and after taking $|m| \longrightarrow \infty$ they saw that the Fourier transforms of $q_1$ and $q_2$ are the same, so the potentials are so too. Note that the only part that requires $n\geq 3$ in this solution is the existence of the three vectors $l,k,m$.

Some papers solve the Calder\'on problem in dimension two with various assumptions. Namely Kohn and Vogelius \cite{KVI} \cite{KVII}, Alessandrini \cite{alessandrini}, Nachman \cite{nachman2} and finally Astala and P\"aiv\"arinta \cite{astalaPaivarinta}. The first three of these require the conductivity to be piecewise analytic. Nachman required two derivatives to convert the conductivity equation into the Schr\"odringer equation. The paper of Astala and P\"aiv\"arinta solved Calder\'on's problem most generally: there were no requirements on the smoothness of the conductivity. It just had to be bounded away from zero and infinity, which is physically realistic.

There are also some results for the inverse boundary value problem of the Schr\"odinger equation whose potential is not assumed to be of the conductivity type. Jerison and Kenig proved, according to \cite{jerisonKenig}, that if $q\in L^p(\Omega)$ with $p>\frac{n}{2}$, $n\geq 3$, then the Dirichlet-Neumann map $\Lambda_q$ determines the potential $q$ uniquely. The case $n=2$ was open until the paper of Bukhgeim. In \cite{bukhgeim}, he introduced new kinds of solutions to the Schr\"odinger equation, which allow the use of stationary phase. This led to an elegant solution of this long standing open problem. There is a point in the argument that requires differentiability of the potentials. Imanuvilov and Yamamoto published the paper \cite{imanuvilovYamamotoLp} in arXiv after the writing of this thesis. They seem to have fixed that problem and hence proven uniqueness for $q\in L^p(\O)$, $p>2$.\index{non-smooth potential} \index{L@$L^p$ potential|see{non-smooth potential}}

Some more recent results in two dimensions have concerned partial data, stability and reducing smoothness requirements for the conductivities and potentials. Notable results of partial data include Imanuvilov, Uhlmann, Yamamoto \cite{imanuvilovUhlmannYamamoto} and Guillarmou and Tzou \cite{guillarmouTzou}. In the first paper the authors consider the Schr\"odinger equation in a plane domain and in the second one on a Riemann surface with boundary. The results of both papers state that knowing the Cauchy data on any open subset on the boundary determines the potential uniquely if it is smooth enough.

Stability seemed to be proven first for the inverse problem of the conductivity equation. Liu \cite{liu1997stability} showed it for potentials of the conductivity type. Barcel\'o, Faraco and Ru\'iz \cite{barceloFaracoRuiz} showed stability for H\"older continuous conductivities. Clop, Faraco and Ru\'iz generalized it to $W^{\alpha,p}$, $\alpha > 0$, in \cite{clopFaracoRuiz}. For the Schr\"odinger equation, there's the result of Novikov and Santacesaria for $C^2$ potentials in \cite{novikovSantacesaria}.

Lastly, we cite very briefly some reconstruction methods\index{reconstruction}. This paragraph is certainly very incomplete as reconstruction was not the focus of the thesis. Nachman gave the first result for the conductivity equation for $n\geq 3$ in \cite{nachman1} and later for $n=2$ in \cite{nachman2}. In the recent paper \cite{astalaMuellerPaivarintaPeramakiSiltanen}, the authors show a numerical reconstruction method for piecewise smooth conductivities in 2D. For a more in-depth survey, see the introduction in that same paper. The case of the Schr\"odinger equation in the plane seems to be more elusive. Bukhgeim mentioned a reconstruction formula at the end of \cite{bukhgeim}, but as far as we know, there are no published numerical methods for reconstructing the potential in 2D. There is a reconstruction formula using only the boundary data explicitly in \cite{novikovSantacesariaRec} though.

\subsection{The main result and sketch of the proof}
\label{sketchSection}
We will give a top-down sketch for proving uniqueness and stability. Before that, we will describe the inverse problem. Let $q_1$ and $q_2$ be two potentials for the Schr\"odinger equations\index{equation!Schr\"odinger} $(\Delta + q_j) u = 0$. We define the \emph{boundary data} $C_{q_j}$ as the collection of pairs $(u_{|\O}, \nu \cdot \d_\nu u_{|\O})$ of boundary values and boundary derivates of all solutions u. If we assume that the operators $\Delta + q_j$ are well posed in Hadamard's sense, then the two sets of boundary data\index{boundary data} become the \emph{Dirichlet-Neumann maps}\index{Dirichlet-Neumann map} $\Lambda_{q_j} : u_{|\O} \mapsto \nabla u_{|\O}$, where $\Delta u + q_j u = 0$. The problem is, what can we tell about $q_1 - q_2$ if we know $C_{q_1}$, $C_{q_2}$? We will show the following:
\begin{theorem*}
Let $\O \subset \C$ be a bounded Lipschitz domain, $M>0$ and $0<s<\frac{1}{2}$. Then there is a positive real number $C$ such that if
$\norm{q_j}_{s,(2,1)} \leq M$ then
\begin{equation}
\norm{q_1-q_2}_{L^{(2,\infty)}(\O)} \leq C \left( \ln d(C_{q_1},C_{q_2})^{-1} \right)^{-s/4}.
\end{equation}
\end{theorem*}
Here $q_j \in H^{s,(2,1)}(\O)$, which can be considered as a slightly smaller space than $H^{s,2}(\O)$, and $d(C_{q_1},C_{q_2})$ is the distance between $C_{q_1}$ and $C_{q_2}$ in a certain sense. It is basically
\begin{equation}
\sup \Big\{ \abs{\int_\O u_1(q_1-q_2)u_2} \,\Big|\, \Delta u_j + q_j u_j = 0, u_j \in W^{1,2}(\O), \norm{u_j} = 1\Big\},
\end{equation}
but, using Green's formula\index{integration by parts}, the integral over $\O$ can be transformed to
\begin{equation}
\cdots = \int_\O u_2 \Delta u_1 - u_1 \Delta u_2 dm = \int_{\d\O} u_2 \nu\cdot\nabla u_1 - u_1 \nu\cdot\nabla u_2 d\sigma,
\end{equation}
which are measurements done on the boundary. Hence, our goal is to estimate $\norm{q_1 - q_2}$ by expressions involving $\int_\O u_1(q_1-q_2)u_2$. This is achieved by choosing special solutions $u_1$, $u_2$, which allow the use of a stationary phase method\index{stationary phase}. Another powerful tool we will use is Carleman estimates. They will take care of the error term, which comes from the fact that the solutions $u_1$ and $u_2$ are not analytic.

The top-down idea starts as follows. Stationary phase arguments show that
\begin{equation}
\norm{q_1-q_2} \longleftarrow \norm{ \frac{2\tau}{\pi} e^{i\tau (z^2 + \wbar{z}^2)} \ast (q_1-q_2)}
\end{equation}
as $\tau \to \infty$. We will show that there are solutions such that $u_1 u_2 \to e^{i\tau(z^2+\wbar{z}^2)}$. This construction was first shown by Bukhgeim \cite{bukhgeim}\index{solutions!Bukhgeim's}. Those solutions will in fact look like $u_1 = e^{i\tau (z-z_0)^2} f_1$, $u_2 = e^{i\tau (\wbar{z}-\wbar{z_0})^2} f_2$, where $z_0$ is the variable outside the convolution, and $f_j \to 1$. Hence we get
\begin{multline}
\norm{q_1 - q_2} \leq  \norm{q_1 - q_2 - \int_\O \frac{2\tau}{\pi} e^{i\tau R} (q_1-q_2) dm} + \norm{\int_\O \frac{2\tau}{\pi} e^{i\tau R} (q_1-q_2) dm} \\
\leq \norm{q_1 - q_2 - \int_\O \frac{2\tau}{\pi} e^{i\tau R} (q_1-q_2) dm} + \norm{\frac{2\tau}{\pi} \int_\O u^{(1)}(q_1 - q_2) u^{(2)} dm} \\
+ \norm{\frac{2\tau}{\pi} \int_\O e^{i\tau R}(q_1 - q_2) (1 - f_1 f_2) dm},
\end{multline}
where $R = (z-z_0)^2 + (\wbar{z}-\wbar{z_0})^2$.

The first term in the equation of the above paragraph can be estimated by $\tau^{-s/2} \norm{q_1-q_2}_{H^{s,2}}$ because of stationary phase. The second one is easy because of the definition of $d(C_{q_1},C_{q_2})$. It has the upper bound
\begin{equation}
d(C_{q_1},C_{q_2}) \norm{u_1}\norm{u_2} \sim e^{c\tau} d(C_{q_1},C_{q_2})
\end{equation}
because of the form of the solutions. The last term is the hardest. By using a suitable cut-off function, we can estimate it above by
\begin{equation}
\label{introErrorTerm}
\tau^{1-s/3}\norm{q_1-q_2}_{H^{s,(2,1)}(\O)} \norm{1-f_1f_2}_{H^{s,(2,\infty)}(\O)}.
\end{equation}

We need to show that $\norm{1-f_1f_2}_s = o( \tau^{s/3-1} )$ as $\tau \to \infty$ to get uniqueness. This is the part that requires new results. It all boils down to Carleman estimates\index{Carleman estimate}. Section \ref{carlemanSubsection} with theorem \ref{BIGTHM} and corollaries \ref{corollary1} and \ref{corollary2} are all about proving them. The new estimates are
\begin{equation}
\begin{split}
&\norm{r}_{H^{(2,\infty)}} \leq C_\O \tau^{-1}(1+\ln \tau) \norm{e^{i\tau (\wbar{z}-\wbar{z_0})^2} \db e^{-i\tau (\wbar{z}-\wbar{z_0})^2} r}_{H^{1,(2,1)}}\\
&\norm{r}_{C^0} \leq C_\O \tau^{-1/3} \norm{e^{i\tau (\wbar{z}-\wbar{z_0})^2} \db e^{-i\tau (\wbar{z}-\wbar{z_0})^2} r}_{H^{1,(2,1)}}\\
&\norm{r}_{H^{s,(2,\infty)}} \leq C_\O \tau^{-1}(1+\ln \tau) \lVert e^{-i\tau(z-z_0)^2} \Delta e^{i\tau(z-z_0)^2} r\rVert_{H^{s,(2,1)}}\\
&\norm{r}_{M^s} \leq C_\O \tau^{-1/3} \lVert e^{-i\tau(z-z_0)^2} \Delta e^{i\tau(z-z_0)^2} r \rVert_{H^{s,(2,1)}}
\end{split}
\end{equation}
where $H^{s,(p,q)}$ is a slight generalization of $H^{s,p}$, and $M^s$ is a space whose functions have smoothness $s$ and can be embedded into $C^0$. We will prove the estimates in the integral form, that is, having the Cauchy operator on the left-hand side. Choosing $r = f_j - 1$ implies that $\norm{1-f_1f_2}_s = O( \tau^{-1}\ln \tau )$. Hence, whenever $s>0$, the error term \eqref{introErrorTerm} tends to zero as $\tau$ grows.

Combining all the upper bounds, we have
\begin{equation}
\norm{q_1-q_2} \leq \tau^{-\beta s} + e^{c\tau} d(C_{q_1},C_{q_2})
\end{equation}
with some $\beta, c > 0$. A suitable choice of $\tau$ implies the claim.

%\vfill
%\section{Notation for function space}
%\begin{itemize}
%\item We are going to use the following spaces. 
%\subitem $L^p$: the standard Lebesgue space of index $p \in [1,\infty]$.
%\subitem $W^{k,p}$, $k$ integer: the space of $L^p$ functions whose distribution derivatives of order up to $k$ are also in $L^p$
%\subitem $L^{(p,q)}$, $p>1$, $0<q\leq \infty$: the Lorentz space (with norm), as defined in \cite{ONeil}
%\subitem $C^k(\wbar{\O})$, $k$ integer: the space of uniformly continuous functions on $\wbar{\O}$ whose derivatives of order up to $k$ are also uniformly continuous on $\wbar{\O}$
%\subitem $H^{s,(p,q)}$, $W^{k,(p,q)}$: The Lorentz-Bessel potential and Lorentz-Sobolev spaces. For basic properties see \ref{WkpqWelldefined} and \ref{HspqWelldefined}.
%\item We don't always write the whole symbol for the space when taking the norm: 
%\subitem $\norm{\cdot}_p$ denotes the $L^p$ norm
%\subitem $\norm{\cdot}_{k,p}$ denotes the $W^{k,p}$ norm
%\subitem $\norm{\cdot}_{(p,q)}$ denotes the $L^{(p,q)}$ norm
%\subitem $\norm{\cdot}_{s,(p,q)}$ denotes either the $H^{s,(p,q)}$ or $W^{s,(p,q)}$ depending on context.
%\end{itemize}

\vfill
\section{Function spaces}
\label{fSpacesSection}
\subsection{Banach-valued Lorentz spaces}
\begin{definition}
\index{functions!simple|textbf}
Let $A$ be a vector space and $X \subset \R^n$ measurable. Then the mapping $f: X \to A$ is a \emph{simple function} if
\begin{equation}
f(x) = \sum_{k=0}^N a_k \chi_{E_k}(x)
\end{equation}
for all $x\in X$ and some $N \in \N$, $a_k \in A$ and disjoint measurable $E_k \subset \R^n$. We use the Lebesgue measure in $\R^n$ where not specified explicitly.
\end{definition}

\begin{definition}
\index{functions!strongly measurable|textbf}
Let $A$ be a Banach space and $X \subset \R^n$ measurable. A function $X \to A$ is \emph{strongly measurable} if there is a sequence of simple functions $f_m:X \to A$ such that
\begin{equation}
f(x) = \lim_{m\to\infty} f_m(x)
\end{equation}
for almost all $x \in X$.
\end{definition}

\begin{definition}
\index{distribution function|textbf}
\index{non-increasing rearrangement|textbf}
\index[notation]{mf@$m(f,\lambda)$; distribution function}
Let $A$ be a Banach space, $\O \subset \R^n$ open and $f:\O \to A$ strongly measurable. Then the \emph{distribution function of $f$}, $\lambda \mapsto m(f,\lambda)$, defined on the non-negative reals, is
\begin{equation}
m(f,\lambda) = m\{ x\in \O \mid \abs{f(x)}_A > \lambda \}.
\end{equation}
The \emph{non-increasing rearrangement of $f$} is the map $f^*:\R_+\cup\{0\} \to \R_+\cup\{0\}$ given by
\begin{equation}
f^*(s) = \inf\{ \lambda \geq 0 \mid m(f,\lambda) \leq s \}.\index[notation]{fstar@$f^*(s)$; non-increasing rearrangement}
\end{equation}
\end{definition}

\begin{definition}
\index{spaces!Lorentz|textbf}
\index[notation]{Lpq@$L^{p,q}, L^{(p,q)}$; Lorentz spaces}
Let $A$ be a Banach space, $\O \subset \R^n$ open, $1<p<\infty$ and $1\leq q\leq \infty$. Then the \emph{seminormed Lorentz space $L^{p,q}(\O,A)$} is the following set
\begin{equation}
\begin{split}
&\{ f:\O\to A \mid f \text{ strongly measurable}, \norm{f}_{L^{p,q}(\O,A)} < \infty \} \\
& \norm{f}_{L^{p,q}(\O,A)} = \left( \int_0^\infty \left( s^{1/p} f^*(s) \right)^q \frac{ds}{s} \right)^{1/q} \quad \text{if } q < \infty, \\
& \norm{f}_{L^{p,q}(\O,A)} = \sup_{s\geq 0} s^{1/p} f^*(s) \quad \text{if } q = \infty,
\end{split}
\end{equation}
equipped with the equivalence $f = g$ if $f(x) = g(x)$ for almost all $x \in \O$.

The \emph{(normed) Lorentz space $L^{(p,q)}(\O,A)$} is defined as
\begin{equation}
\begin{split}
&\{ f:\O\to A \mid f \text{ strongly measurable}, \norm{f}_{L^{(p,q)}(\O,A)} < \infty \} \\
& \norm{f}_{L^{(p,q)}(\O,A)} = \left( \int_0^\infty \left( t^{1/p} f^{**}(t) \right)^q \frac{dt}{t} \right)^{1/q} \quad \text{if } q < \infty, \\
& \norm{f}_{L^{(p,q)}(\O,A)} = \sup_{s\geq 0} t^{1/p} f^{**}(t) \quad \text{if } q = \infty,
\end{split}
\end{equation}
where $f^{**}(t) = \frac{1}{t} \int_0^t f^*(s) ds$.\index[notation]{fstarstar@$f^{**}(t)$} Again, we set $f=g$ if they are equal almost everywhere.
\end{definition}

\begin{remark}
\index{spaces!weak-$L^p$|textbf}
\index[notation]{Lpstar@$L^{p*}$; weak $L^p$-spaces}
The spaces $L^{p,\infty}(\O,A)$ and $L^{(p,\infty)}(\O,A)$ are sometimes written $L^{p*}(\O,A)$ and are called \emph{weak $L^p$-spaces}.
\end{remark}

\begin{remark}
We often leave the domain $\O$ out of the notation, so write $L^{p,q}(A)$ and $L^{(p,q)}(A)$ for these spaces. On the other hand, sometimes we leave the range out. Whether the set is the domain or range should be clear from the context.
\end{remark}

\begin{theorem}
\label{lorentzProperties}
Let $A$ be a Banach space, $\O \subset \R^n$ open, $1<p<\infty$ and $1\leq q\leq \infty$. Then $L^{p,q}(\O,A)$ is a complete semi-normed space and $L^{(p,q)}(\O,A)$ is a Banach space. Moreover $L^{p,q} \equiv L^{(p,q)}$ with
\begin{equation}
\norm{f}_{p,q} \leq \norm{f}_{(p,q)} \leq \frac{p}{p-1} \norm{f}_{p,q}.
\end{equation}
The spaces have the following properties:
\begin{itemize}
\item If $1\leq q \leq Q \leq \infty$ then $L^{(p,q)} \hookrightarrow L^{(p,Q)}$ and $L^{(p,p)} = L^p$
\item $\norm{ \abs{f}^r}_{p,q} = \norm{f}^r_{pr,qr}$ for $r \geq 1$.
\item Simple functions are dense in $L^{(p,q)}$ if $q < \infty$
\item Countably valued $L^{(p,\infty)}$ functions are dense in $L^{(p,\infty)}$
\end{itemize}
\end{theorem}
\begin{proof}
Note that if $f:\O\to A$ is strongly measurable, then $\abs{f}_A : \O \to \R$ is measurable. Hence most of the proofs follow exactly like in the complex-valued case, for example in chapter 1.4. of Grafakos \cite{grafakos}. The following all refer to that book. Completeness and equivalence follow from 1.4.11, 1.4.12. The inclusions follow from 1.4.10 and the $L^p$ equality from 1.4.5(12). The proof of the exponential scaling of the norm is given by 1.4.7.

Densities will be proven using a different source. The spaces $L^{(p,q)}(\O,A)$ of this theorem can be gotten using real interpolation on the Banach couple $(L^{p_0}(\O,A), L^{p_1}(\O,A))$ with some $1<p_0<p<p_1<\infty$ according to theorem 5.2.1 in \cite{BL}. Simple functions are dense in the spaces $L^p(\O,A)$ for $1\leq p<\infty$ by corollary III.3.8 in \cite{dunfordSchwartz1}, hence they are so in the intersection $L^{p_0} \cap L^{p_1}$ too. The latter is dense in $L^{(p,q)}(\O,A)$ when $q<\infty$ by theorem 3.4.2 of \cite{BL}. This inclusion is a bounded linear operator, so simple functions are dense in $L^{(p,q)}(\O,A)$.

Let $f \in L^{(p,\infty)}(\O,A)$. Split $\O$ into a countable number of disjoint bounded and measurable sets $\O_j$. According to corollary 3 of section II.1 in \cite{diestelUhl}, there are countably valued measurable functions\index{functions!countably valued} $s_j:\O_j \to A$ such that
\begin{equation}
\abs{f(x)-s_j(x)}_A < \epsilon 2^{-j} \min(1, \norm{\chi_j}_{(p,\infty)}^{-1})
\end{equation}
for all $x \in \O_j$. We write $\chi_j = \chi_{\O_j}$. Note that $s_j \in L^{(p,\infty)}(\O_j,A)$. Extend $s_j$ by zero to the whole domain $\O$ and let $s(x) = \sum_j s_j(x)$. Now
\begin{multline}
\norm{f-s_j}_{L^{(p,\infty)}(\O)} \leq \sum_{j=1}^\infty \norm{(f-s_j)\chi_j}_{L^{(p,\infty)}(\O)} \\
\leq \sum_{j=1}^\infty \norm{\chi_j}_{L^{(p,\infty)}(\O)} \sup_{x\in \O_j} \abs{f(x) - s_j(x)}_A < \epsilon \sum_{j=1}^\infty 2^{-j} = \epsilon.
\end{multline}
Moreover, $s$ is a countable sum of countably valued measurable functions, so it satisfies our claim.
\end{proof}

\begin{lemma}[Minkowski's integral inequality]
\index{Minkowski's integral inequality|textbf}
\label{minkowskiInt}
Let $A$ be Banach, $\O \subset \R^n$ and $S \subset \R^m$ both open. Moreover let $1<p<\infty$ and $1 \leq q \leq \infty$. Let $f: \O \times S \to A$ be strongly measurable. If $f(\cdot, y) \in L^{(p,q)}(\O,A)$ for almost all $y \in S$ and $y \mapsto \norm{f(\cdot,y)}_{(p,q)}$ is in $L^1(S,\R)$, then
\begin{equation}
x \mapsto \int_S f(x,y) dm(y)
\end{equation}
is in $L^{(p,q)}(\O,A)$ and
\begin{equation}
\norm{\int_S f(\cdot, y) dm(y)}_{L^{(p,q)}(\O,A)} \leq C_p \int_S \norm{f(\cdot,y)}_{L^{(p,q)}(\O,A)} dm(y)
\end{equation}
where $C_p<\infty$ depends only on $p$.
\end{lemma}
\begin{proof}
Denote $g(x) = \int_S \abs{f(x,y)}_A dm(y)$, so $g:\O \to \R \cup \{\infty\}$ is measurable by Fubini's theorem, for example 8.8.a in \cite{rudin}. We will first show that the real valued $g\in \big(L^{(p',q')}(\O)\big)^*$, where $a^{-1}+a'^{-1}=1$ for $a =p,q$. This will imply that $g \in L^{(p,q)}(\O)$ by theorem 1.4.17 in \cite{grafakos} and lemma 2 in \cite{cwikel} because they show that
\begin{equation}
\big(L^{(p',q')}(\O)\big)^* \cap \{\text{measurable functions}\} \subset L^{(p,q)}(\O)
\end{equation}
assuming that the measure is non-atomic, which $m$ is. The right-hand side of the next estimate will be finite, hence we may use Fubini's theorem. It implies, with the generalized H\"older's inequality of O'Neil \cite{ONeil}, that
\begin{multline}
\norm{w\mapsto \int_\O gw dm}_{\left(L^{(p',q')}(\O)\right)^*} = \sup_{\norm{w}_{(p',q')} = 1} \abs{\int_\O g(x) w(x) dm(x) } \\
\leq \sup_{\norm{w}_{(p',q')}=1} \int_\O g(x) \abs{w(x)} dm(x) \\
= \sup_{\norm{w}_{(p',q')}=1} \int_S \int_\O \abs{f(x,y)}_A \abs{w(x)} dm(x) dm(y) \\
\leq \sup_{\norm{w}_{(p',q')}=1} \int_S \norm{f(\cdot,y)}_{(p,q)} \norm{w}_{(p',q')} dm(y) = RHS < \infty
\end{multline}
by the assumptions on $f$. Hence $q \in L^{(p,q)}(\O,\R)$ and so $y\mapsto f(x,y)$ is integrable for almost all $x$. It remains to show that $x \mapsto \int_S f(x,y)dm(y)$ is strongly measurable, since then
\begin{equation}
\norm{\int_S f(\cdot,y)dm(y)}_{L^{(p,q)}(\O,A)} \leq \norm{g}_{L^{(p,q)}(\O,\R)} \leq C_p \norm{g}_{\left(L^{(p',q')}(\O,\R)\right)^*},
\end{equation}
and so it is in $L^{(p,q)}(\O,A)$.

Let  $S_m: \O\times S \to A$ be simple functions such that $S_m(x,y) \to f(x,y)$ almost everywhere. We may assume that $\abs{S_m(x,y)}_A \leq \abs{f(x,y)}_A$ by considering $t_m S_m \abs{S_m}_A^{-1}$ instead of $S_m$, where $t_m$ are simple real-valued functions rising to $\abs{f}_A$. We may also assume that $S_m$ has bounded support. Define $s_{x,m}(y) = S_m(x,y)$. Now $s_{x,m}$ is a simple function on $S$, $s_{x,m}(y) \to f(x,y)$ for almost all $y$ for almost all $x$, and $\abs{s_{x,m}(y)}_A \leq \abs{f(x,y)}_A \in L^1(S)$ for almost all $x$. Hence, for almost all $x$, we get
\begin{equation}
\int_S f(x,y) dm(y) = \lim_{m\to\infty} \int_S s_{x,m}(y) dm(y)
\end{equation}
by dominated convergence. The latter integrals are strongly measurable, so the claim follows.
\end{proof}

\subsection{Interpolation of Lorentz spaces}
\label{newSpacesSection}
We use definitions like in \cite{BL} when intepolating. In particular $(\cdot,\cdot)_{[\theta]}$ represents complex interpolation. We give a short definition and a few examples. After them, we interpolate Banach-valued Lorentz spaces. The proof is an almost exact replica of theorem 5.1.2 in \cite{BL}, where Bergh and L\"ofstr\"om interpolate Banach valued $L^p$ spaces.

\begin{definition}
Let $A_0, A_1$ be topological vector spaces and assume that there is a Hausdorff topological vector space $\mathscr{H}$ such that $A_0, A_1 \hookrightarrow \mathscr{H}$. Then $A_0$ and $A_1$ are \emph{compatible}.\index{compatible|textbf}
\end{definition}

\begin{definition}
Let $A_0$ and $A_1$ be Banach spaces which are subspaces of a Hausdorff topological vector space $\mathscr{H}$. Then $(A_0,A_1)$ is said to be \emph{a compatible Banach couple}, or \emph{a Banach couple} for short.\index{Banach couple|textbf}
\end{definition}
\begin{remark}
Compatible couples are normally defined like this: If $\mathscr{C}$ is a subcategory of all normed vector spaces, then $(A_0,A_1)$ is \emph{a compatible couple in $\mathscr{C}$} if these conditions hold: i) $A_0$ and $A_1$ are compatible, ii) $A_0 \cap A_1 \in \mathscr{C}$ and iii) $A_0 + A_1 \in \mathscr{C}$. Our definition satisfies this in the category of Banach spaces by lemma 2.3.1 in \cite{BL}.
\end{remark}

\begin{definition}
\index[notation]{FA@$\mathscr{F}(A)$}
Let $S = \{ z \in \C \mid 0 < \Re z < 1\}$ and $\wbar{A}=(A_0,A_1)$ be a compatible Banach couple. Then $\mathscr{F}(\wbar{A})$ consist of the all the functions $ f: \wbar{S} \to A_0+A_1$ satisfying
\begin{itemize}
\item $f$ is bounded and continuous when $A_0+A_1$ is equipped with the norm $\norm{a}_{A_0+A_1} = \inf_{a=a_0+a_1} \norm{a_0}_{A_0} + \norm{a_1}_{A_1}$
\item $f$ is analytic on $S$
\item the maps $t \mapsto f(it)$, $t\mapsto f(1+it)$ are continuous $\R \to A_0$, $\R \to A_1$, respectively, and they tend to zero as $\abs{t} \to \infty$
\end{itemize}
We equip $\mathscr{F}(\wbar{A})$ with the norm
\begin{equation}
\norm{f}_{\mathscr{F}(A_0,A_1)} = \max\big( \sup_{t\in\R} \norm{f(it)}_{A_0} , \sup_{t\in\R} \norm{f(1+it)}_{A_1} \big).
\end{equation}
\end{definition}

\begin{remark}
$\mathscr{F}(\wbar{A})$ is a Banach space by theorem 4.1.1 of \cite{BL}.
\end{remark}

\begin{definition}
\index{spaces!complex interpolation|textbf}
\index{interpolation!complex|textbf}
\index[notation]{AzeroAone@$(A_0,A_1)_{[\theta]}$; interpolation space}
\label{cInterpDef}
Let $(A_0,A_1)$ be a Banach couple and $0 \leq \theta \leq 1$. Then
\begin{equation}
(A_0,A_1)_{[\theta]} = \{ a \in A_0 + A_1 \mid a = f(\theta) \text{ for some } f \in \mathscr{F}(A_0,A_1) \}
\end{equation}
and we equip if with the norm
\begin{equation}
\norm{a}_{[\theta]} =  \norm{a}_{(A_0,A_1)_{[\theta]}} = \inf \{ \norm{f}_{\mathscr{F}(A_0,A_1)} \mid f(\theta) = a, f \in \mathscr{F}(A_0,A_1) \}.
\end{equation}
The structure $\big((A_0,A_1)_{[\theta]}, \norm{\cdot}_{[\theta]}\big)$ is called a \emph{complex interpolation space}.
\end{definition}

\begin{theorem}
\label{cInterpWelldefined}
Let $\wbar{A} = (A_0,A_1)$ and $\wbar{B} = (B_0,B_1)$ be Banach couples\index{Banach couple} and $0 \leq \theta \leq 1$. Then $\wbar{A}_{[\theta]}$ and $\wbar{B}_{[\theta]}$ are Banach spaces with continuous embeddings\footnote{$A_0 \cap A_1$ is equipped with the norm $\norm{a}_{A_0\cap A_1} = \max(\norm{a}_{A_0}, \norm{a}_{A_1})$ and $A_0 + A_1$ is equipped with $\norm{a}_{A_0+A_1} = \inf_{a=a_0+a_1} \norm{a_0}_{A_0} + \norm{a_1}_{A_1}$} $A_0\cap A_1 \hookrightarrow \wbar{A}_{[\theta]} \hookrightarrow A_0 + A_1$ and the same for $B$. Moreover if
\begin{equation}
\begin{split}
T:A_0 \to B_0 &\quad \text{with norm } M_0 \\
T:A_1 \to B_1 &\quad \text{with norm } M_1
\end{split}
\end{equation}
then $T:\wbar{A}_{[\theta]} \to \wbar{B}_{[\theta]}$ with norm at most $M_0^{1-\theta} M_1^\theta$.
\end{theorem}
\begin{proof}
See theorem 4.1.2 in \cite{BL} and the definitions of intermediate spaces and exact interpolation functors 2.4.1, 2.4.3 in that same book.
\end{proof}

\begin{theorem}[Multilinear interpolation]\index{interpolation!multilinear operators}
\label{multInterp}
Let $\wbar{A}$, $\wbar{B}$ and $\wbar{X}$ be Banach couples\index{Banach couple} and $0 \leq \theta \leq 1$. Assume that  $T: (A_0 \cap A_1) \times (B_0 \cap B_1) \to (X_0 \cap X_1)$ is multilinear and
\begin{equation}
\begin{split}
\norm{T(a,b)}_{X_0} \leq M_0 \norm{a}_{A_0} \norm{b}_{B_0} \\
\norm{T(a,b)}_{X_1} \leq M_1 \norm{a}_{A_1} \norm{b}_{B_1}
\end{split}
\end{equation}
for $a \in A_0 \cap A_1$ ad $b \in B_0 \cap B_1$. Then $T$ can be uniquely extended to a multilinear mapping $\wbar{A}_{[\theta]} \times \wbar{B}_{[\theta]} \to \wbar{X}_{[\theta]}$ with $\norm{T(a,b)}_{X_{[\theta]}} \leq M_0^{1-\theta}M_1^\theta \norm{a}_{A_{[\theta]}} \norm{b}_{B_{[\theta]}}$.
\end{theorem}
\begin{proof}
See theorem 4.4.1 in \cite{BL}.
\end{proof}

\begin{example}
\label{example1}
Let $0\leq\theta\leq 1$. Let's prove that $(A,A)_{[\theta]} = A$ with equal norm to get a hold of the definitions. Let $a \in (A,A)_{[\theta]}$. Then there is $f \in \mathscr{F}(A,A)$ such that $a = f(\theta)$. We may assume that $\norm{f}_{\mathscr{F}} \leq (1+ \epsilon) \norm{a}_{[\theta]}$ by the definition of the norm in $(A,A)_{[\theta]}$. Now $a = f(\theta) \in A+A = A$, and
\begin{multline}
\norm{a}_A = \norm{f(\theta)}_A \leq \max\big( \sup \norm{f(it)}_A , \sup \norm{f(1+it)}_A \big) \\
= \norm{f}_{\mathscr{F}} \leq (1+\epsilon) \norm{a}_{[\theta]}
\end{multline}
because of the Phragm\'en-Lindel\"of principle\index{Phragm\'en-Lindel\"of principle}. This is allowed since $f$ is bounded on $\wbar{S}$. Taking similar $f \in \mathscr{F}$ while letting $\epsilon \to 0$ gives $\norm{a}_A \leq \norm{a}_{[\theta]}$.

Now let $a \in A$. Let's construct a suitable $f \in \mathscr{F}(A,A)$. Let
\begin{equation}
f(z) = e^{\epsilon(z-\theta)^2}a = e^{\epsilon( (\Re z - \theta)^2 - (\Im z)^2 + 2i\Im z ( \Re z - \theta))} a.
\end{equation}
The function $f$ is clearly continuous and bounded on $\wbar{S}$ and analytic on $S$. The continuity from the boundary to the respective spaces follows since we have just one Banach space. Finally, $\norm{f(it)}_A = \exp(\epsilon(\theta^2 - t^2)) \norm{a}_A \to 0$ as $\abs{t} \to \infty$. The same holds for $f(1+it)$, so $f \in \mathscr{F}(A,A)$. Also $f(\theta) = a$, so $a \in (A,A)_{[\theta]}$. Now
\begin{multline}
\norm{a}_{[\theta]} \leq \norm{f}_{\mathscr{F}} = \max( \sup \norm{f(it)}_A, \norm{f(1+it)}_A ) \\
= \max( \sup e^{\epsilon(\theta^2 - t^2)}, \sup e^{\epsilon((1-\theta)^2 - t^2)}) \norm{a}_A \leq e^\epsilon \norm{a}_A.
\end{multline}
Letting $\epsilon \to 0$ shows that $\norm{a}_{[\theta]} \leq \norm{a}_A$.
\end{example}

\begin{example}
\label{example2}
We also have $(L^1,L^\infty)_{[\frac{1}{p}]} = L^p$. The proof is based on choosing
\begin{equation}
f = e^{\epsilon(z^2-\frac{1}{p^2})} \abs{a}^{p(1-z)} \frac{a}{\abs{a}}
\end{equation}
and using \emph{the three lines theorem}\index{three lines theorem}. For details, check theorem 5.1.1 in \cite{BL}.
\end{example}

\begin{remark}
It would seem that the direction $(A_0,A_1)_{[\theta]} \hookrightarrow X$ requires often the use of complex analysis, while the other one doesn't. In example \ref{example1}, we used the Phragm\'en-Lindel\"of principle when proving that $(A,A)_{[\theta]} \hookrightarrow A$. In example \ref{example2}, the three lines theorem comes into play when showing that $(L^1,L^\infty)_{[\theta]} \hookrightarrow L^p$. Lastly, the proof of the next theorem will require properties of the Poisson kernel of $S$ when showing that same direction.
\end{remark}

We will not write out the domain $\R^n$. The proof works for any domain.

\begin{theorem}
\index{interpolation!A-valued@$A$-valued Lorentz spaces}
\label{lorentzInterp}
Let $(A_0,A_1)$ be a compatible Banach couple\index{Banach couple}, $1<p_j<\infty$ and $1\leq q_j < \infty$. Let $0<\theta<1$ and $\frac{1}{p} = \frac{1-\theta}{p_0} + \frac{\theta}{p_1}$, $\frac{1}{q} = \frac{1-\theta}{q_0} + \frac{\theta}{q_1}$. Then
\begin{multline}
L^{\left(p,p\min(\frac{q_0}{p_0}, \frac{q_1}{p_1})\right)}\big((A_0,A_1)_{[\theta]}\big) \\
\subset \left( L^{(p_0,q_0)}(A_0), L^{(p_1,q_1)}(A_1)\right)_{[\theta]} \\
\subset L^{(p,q)}\big((A_0,A_1)_{[\theta]}\big)
\end{multline}
and
\begin{equation}
\left( L^{(p,\infty)}(A_0), L^{(p,\infty)}(A_1) \right)_{[\theta]} = L^{(p,\infty)}\big( (A_0,A_1)_{[\theta]} \big)
\end{equation}
with corresponding norm estimates.
\end{theorem}
\begin{proof}
Since $(A_0,A_1)$ is a Banach couple, so are the other pairs of spaces in the theorem. We may interpolate. The idea is to take $a \in L^{(\cdot,\cdot)}\big((A_0,A_1)_{[\theta]}\big)$ and then, for each $x$, take an analytic $A_0 + A_1$-valued function $g_x(z)$ satisfying $g_x(\theta) = a(x)$. After that we show that $x\mapsto g_x(z)$ is a strongly measurable function, so $z\mapsto g_\cdot (z)$ would actually be in $\mathscr{F}\big(L^{(\cdot,\cdot)}(A_0), L^{(\cdot,\cdot)}(A_1)\big)$. Simple functions are dense in all of these spaces when $q,q_j<\infty$ and countably simple functions are so for $q=\infty$ by theorem \ref{lorentzProperties}. Using these makes the above much easier.

Consider the case of $q_j,q<\infty$ first. Note that $p < \infty$, so the simple functions must have  support with finite measure. Let
\begin{multline}\index{functions!simple}
\Xi = \Big\{ s :\R^n \to A_0\cap A_1 \,\Big\vert\, \exists N\in \N, a_k \in A_0\cap A_1, E_k \subset \R^n, m(E_j) < \infty, \\
E_j\cap E_k = \emptyset \text{ for } j \neq k, \text{and such that } s(x) = \sum_{k=0}^N a_k \chi_{E_k}(x) \Big\}
\end{multline}
It is enough to assume that $a\in\Xi$. This is because of the following. The set $A_0 \cap A_1$ is dense in $(A_0,A_1)_{[\theta]}$ by theorem 4.2.2. of \cite{BL}. Hence $\Xi$ is dense in $L^{(p,q)}((A_0,A_1)_{[\theta]})$. Moreover $\Xi$ is dense in $L^{(p_0,q_0)}(A_0) \cap L^{(p_1,q_1)}(A_1)$, hence in $\left( L^{(p_0,q_0)}(A_0), L^{(p_1,q_1)}(A_1) \right)_{[\theta]}$ too by that same theorem.

Let $a \in \Xi \subset L^{(p,q)}\big((A_0,A_1)_{[\theta]}\big)$. To simplify notation we assume that $\norm{a}_{(p,p\min(q_0/p_0,q_1/p_1))} = 1$ and write
\begin{equation}
a(x) = \sum_{k=0}^N a_k \chi_{E_k}(x).
\end{equation}
Let $\epsilon > 0$. We have $a(x) \in (A_0,A_1)_{[\theta]}$ for each $x\in\R^n$. Then, for $x\in\R^n$, there exists $g_x \in \mathscr{F}(A_0,A_1)$ such that $\norm{g_x}_{\mathscr{F}(A_0,A_1)} \leq (1+\epsilon) \abs{a(x)}_{(A_0,A_1)_{[\theta]}}$ and $g_x(\theta) = a(x)$. If $a(x) = a(y)$, take $g_x = g_y$. Define
\begin{equation}
\phi(z) = g(z) \abs{a}_{(A_0,A_1)_{[\theta]}}^{p(\frac{1}{p_0}-\frac{1}{p_1})(z-\theta)}.
\end{equation}
Now, given any $z \in \wbar{S}$, $\phi(z)$ is strongly measurable\footnote{Because in fact $g(z) = \sum_{k=0}^N b_k(z) \chi_{E_k}$, where $b_k \in \mathscr{F}(A_0,A_1)$ gives $b_k(\theta) = a_k$.} $\R^n \to A_0+A_1$, $\phi$ is analytic $S \to L^{(p_0,q_0)}(A_0) + L^{(p_1,q_1)}(A_1)$, continuous on $\wbar{S}$, $\phi(it) \in L^{(p_0,q_0)}(A_0)$, $\phi(1+it) \in L^{(p_1,q_1)}(A_1)$, they are continuous and tend to zero as $\abs{t} \to \infty$. Hence $\phi \in \mathscr{F}\big(L^{(p_0,q_0)}(A_0), L^{(p_1,q_1)}(A_1)\big)$. Moreover $\phi(\theta) = a$. Now
\begin{multline}
\norm{a}_{\left(L^{(p_0,q_0)}(A_0), L^{(p_1,q_1)}(A_1)\right)_{[\theta]}} \leq \norm{\phi}_{\mathscr{F}\left(L^{(p_0,q_0)}(A_0), L^{(p_1,q_1)}(A_1)\right)} \\
= \max\left(\sup_{t\in\R} \norm{\phi(it)}_{L^{(p_0,q_0)}(A_0)}, \sup_{t\in\R} \norm{\phi(1+it)}_{L^{(p_1,q_1)}(A_1)} \right).
\end{multline}
Let's estimate the first supremum. Note that $\norm{\abs{g}^r}_{p,q} = \norm{g}_{pr,qr}^r$ by theorem \ref{lorentzProperties} and $\norm{g}_{p,q} \leq \norm{g}_{(p,q)} \leq \frac{p}{p-1}\norm{g}_{p,q}$. Then
\begin{multline}
\norm{\phi(it)}_{L^{(p_0,q_1)}(A_0)} = \norm{ \abs{g(it)}_{A_0} \abs{a}_{(A_0,A_1)_{[\theta]}}^{-\theta p (\frac{1}{p_0} - \frac{1}{p_1})} }_{L^{(p_0,q_0)}(\R)} \\
\leq \norm{ \norm{g}_{\mathscr{F}(A_0,A_1)} \abs{a}_{(A_0,A_1)_{[\theta]}}^{-\theta p (\frac{1}{p_0} - \frac{1}{p_1})} }_{L^{(p_0,q_0)}(\R)} \leq (1+\epsilon)\norm{ \abs{a}_{(A_0,A_1)_{[\theta]}}^{1-\theta p (\frac{1}{p_0} - \frac{1}{p_1})} }_{L^{(p_0,q_0)}(\R)} \\
=(1+\epsilon) \norm{ \abs{a}_{(A_0,A_1)_{[\theta]}}^{p/p_0 }}_{L^{(p_0,q_0)}(\R)} \leq C_{p_0} (1+\epsilon) \norm{ \abs{a}_{(A_0,A_1)_{[\theta]}} }^{p/p_0}_{L^{(p,pq_0/p_0)}(\R)}\\
= C_{p_0} (1+\epsilon) \norm{a}_{L^{(p,pq_0/p_0)}\left((A_0,A_1)_{[\theta]}\right)}^{p/p_0}.
\end{multline}
We get similarly
\begin{equation}
\norm{\phi(1+it)}_{L^{(p_1,q_1)}(A_1)} \leq \cdots \leq C_{p_1}(1+\epsilon) \norm{a}_{L^{(p,pq_1/p_1)}\left((A_0,A_1)_{[\theta]}\right)}^{p/p_1}.
\end{equation}
Reducing the second parameter of the Lorentz spaces gives a smaller space, and we made the assumption of $\norm{a}_{(p,p\min(q_0/p_0,q_1/p_1))} = 1$, so
\begin{equation}
\norm{a}_{\left(L^{(p_0,q_1)}(A_0), L^{(p_1,q_1)}(A_1)\right)_{[\theta]}} \leq C_{p_0,p_1} \norm{a}_{L^{\left(p,p\min(\frac{q_0}{p_0}, \frac{q_1}{p_1})\right)}\left((A_0,A_1)_{[\theta]}\right)}.
\end{equation}

\medskip
The other direction requires Minkowski's integral inequality of lemma \ref{minkowskiInt} and the inequality
\begin{multline}
\abs{f(\theta)}_{(A_0,A_1)_{[\theta]}} \leq \left(\frac{1}{1-\theta} \int_{-\infty}^\infty \abs{f(i\tau)}_{A_0} P_0(\theta,\tau) d\tau \right)^{1-\theta}\\
\cdot \left(\frac{1}{\theta} \int_{-\infty}^\infty \abs{f(1+i\tau)}_{A_1} P_1(\theta,\tau) d\tau \right)^{\theta}
\end{multline}
proven in lemma 4.3.2 of \cite{BL}. Here $f \in \mathscr{F}(A_0,A_1)$ and
\begin{equation}
P_j(s+it,\tau) = \frac{e^{-\pi(\tau-t)}\sin \pi s}{\sin^2\pi s + (\cos \pi s - e^{ij\pi - \pi(\tau - t)})^2}
\end{equation}
is the Poisson kernel\index{Poisson kernel} of the strip $S$.

Let $a \in \left(L^{(p_0,q_0)}(A_0), L^{(p_1,q_1)}(A_1) \right)_{[\theta]}$. Then there is a corresponding analytic $f \in \mathscr{F}\left(L^{(p_0,q_0)}(A_0), L^{(p_1,q_1)}(A_1) \right)$ such that $f(\theta) = a$ and whose norm is bounded by $\norm{f}_{\mathscr{F}} \leq (1+\epsilon) \norm{a}_{[\theta]}$. Note that $\frac{1}{p} = \frac{1-\theta}{p_0} + \frac{\theta}{p_1}$ and $\frac{1}{q} = \frac{1-\theta}{q_0} + \frac{\theta}{q_1}$, so the generalized H\"older's inequality given in theorem 3.4 of \cite{ONeil} allows us to take the norms of the factors in the product. Everything is then ready:
\begin{multline}
\norm{a}_{L^{(p,q)}\left((A_0,A_1)_{[\theta]}\right)} = \norm{ \abs{f(\theta)}_{(A_0,A_1)_{[\theta]}} }_{L^{(p,q)}(\R)} \\
\leq \norm{\left(\frac{1}{1-\theta} \int_{-\infty}^\infty \abs{f(i\tau)}_{A_0} P_0(\theta,\tau) d\tau \right)^{1-\theta}}_{(\frac{p_0}{1-\theta}, \frac{q_0}{1-\theta})} \\
\cdot \norm{\left(\frac{1}{\theta} \int_{-\infty}^\infty \abs{f(1+i\tau)}_{A_1} P_1(\theta,\tau) d\tau \right)^{\theta}}_{(\frac{p_1}{\theta},\frac{q_1}{\theta})} \\
\leq C_{p_0,p_1,\theta} \norm{\int_{-\infty}^\infty \abs{f(i\tau)}_{A_0} P_0(\theta,\tau) d\tau}_{(p_0,q_0)}^{1-\theta} \\
\cdot\norm{\int_{-\infty}^\infty \abs{f(1+i\tau)}_{A_1} P_1(\theta,\tau) d\tau }_{(p_1,q_1)}^\theta \\
\leq C_{p_0,p_1,\theta} \left( \int_{-\infty}^\infty \norm{f(i\tau)}_{L^{(p_0,q_0)}(A_0)} P_0(\theta,\tau) d\tau  \right)^{1-\theta} \\
\cdot \left(\int_{-\infty}^\infty \norm{f(1+i\tau)}_{L^{(p_1,q_1)}(A_1)} P_1(\theta,\tau) d\tau \right)^\theta \\
\leq C_{p_0,p_1,\theta} \norm{f}_{\mathscr{F}\left(L^{(p_0,q_0)}(A_0), L^{(p_1,q_1)}(A_1)\right)} \\
\leq C_{p_0,p_1,\theta} (1+\epsilon) \norm{a}_{\left(L^{(p_0,q_0)}(A_0), L^{(p_1,q_1)}(A_1)\right)_{[\theta]}} < \infty.
\end{multline}

\medskip
The last claim follows similarly, except that we use
\begin{multline}\index{functions!countably valued}
\Xi = \Big\{ s \in L^{(p,\infty)}(A_0\cap A_1) \,\Big\vert\, \exists a_k \in A_0\cap A_1, E_k \subset \R^n, m(E_j) < \infty, \\
E_j\cap E_k = \emptyset \text{ for } j \neq k, \text{and such that } s(x) = \sum_{k=0}^\infty a_k \chi_{E_k}(x) \Big\},
\end{multline}
which is dense in $L^{(p,\infty)}(A_0\cap A_1)$ by theorem \ref{lorentzProperties}. We get density in $L^{(p,\infty)}(A_0) \cap L^{(p,\infty)}(A_1)$ and $L^{(p,\infty)}\left((A_0,A_1)_{[\theta]}\right)$ because $A_0 \cap A_1$ is dense in $(A_0,A_1)_{[\theta]}$. All other steps are the same, but with simpler expressions.
\end{proof}

\begin{remark}
The same proof works for Lorentz spaces defined on a domain.
\end{remark}

\begin{remark}
If $\frac{q_0}{p_0} = \frac{q_1}{p_1}$, then the theorem shows that
\begin{equation}
\left( L^{(p_0, q_0)}(A_0), L^{(p_1,q_1)}(A_1) \right)_{[\theta]} = L^{(p,q)}\left((A_0,A_1)_{[\theta]}\right).
\end{equation}
Maybe a better choice of $f$ could prove this without assuming anything from our parameters.
\end{remark}

\begin{remark}
Why can't we have $p_0\neq p_1$ when $q_0=\infty$? Maybe we could, but this proof won't work then. The problem is to find a set $\Xi$ of quite ``simple'' functions which would be dense in all the spaces considered at the same time. On the other hand, Adams and Fournier claim this result in 7.56 \cite{adams}, assuming that $A_0 = A_1$. In that case it would follow from reiteration with real interpolation e.g. by theorem 4.7.2 of \cite{BL}.
\end{remark}

\begin{remark}
By Cwikel $\big(L^{p_0}(A_0), L^{p_1}(A_1)\big)_{\theta,q}$ is not necessarily a Lorentz space \cite{cwikel}. So it is not possible to use reiteration to prove our claim in general if $A_0 \neq A_1$.
\end{remark}

\subsection{Lorentz-Sobolev spaces}

\begin{definition}
\index{spaces!bounded continuous functions|textbf}
\index[notation]{BC@$BC(X,A)$; bounded continuous functions}
Let $X \subset \R^n$ be any nonempty set and $A$ be a Banach space. Then the \emph{space of bounded continuous $A$-valued functions} is
\begin{equation}
BC(X, A) = \{ f:X \to A \mid \text{$f$ is continuous and bounded} \},
\end{equation}
equipped with the norm $\norm{f}_{BC(X,A)} = \sup_{z_0 \in X} \abs{f(z_0)}_A$.
\end{definition}
\begin{remark}
This is a Banach space.
\end{remark}

\begin{definition}\index{spaces!Lorentz-Sobolev|textbf}
Let $\O \subset \R^n$ open, $1< p<\infty$, $1\leq q \leq \infty$ and $k\in\N$. Define the \emph{Lorentz-Sobolev space} $W^{k,(p,q)}(\O)$ as follows:
\begin{equation}
W^{k,(p,q)}(\O) = \{ f \in L^{(p,q)}(\O) \mid D^\alpha f \in L^{(p,q)}(\O) \text{ for } \abs{\alpha} \leq k\}
\end{equation}
with norm
\begin{equation}
\norm{f}_{W^{k,(p,q)}(\O)} = \norm{f}_{L^{(p,q)}(\O)} + \sum_{\norm{\alpha}\leq k} \norm{D^\alpha f}_{L^{(p,q)}(\O)}.
\end{equation}
where $D^\alpha$ is the distribution derivative in $\O$.
\end{definition}

\begin{theorem}
\label{WkpqWelldefined}
Let $\O \subset \R^n$ be an open set satisfying the cone and segment conditions\footnote{See for example 4.5 and 4.6 in \cite{adams}. The cone condition prevents cusps while the segment condition ensures that the domain is never on both sides of the boundary, i.e. $]-1,0[^2 \cup ]0,1[^2$ is not allowed. Bounded Lipschitz domains have this property.}, $1< p<\infty$, $1\leq q \leq \infty$ and $k\in\N$. Then the space $W^{k,(p,q)}(\O)$ is a well defined Banach space with the following properties:
\begin{enumerate}
\item The restrictions of $C_0^\infty(\R^n)$ test functions to $\O$ are dense in $W^{k,(p,q)}(\O)$ for $q < \infty$
\item We have the continuous embedding $W^{k,(\frac{n}{k},1)}(\O) \hookrightarrow BC(\wbar{\O})$ for $k\geq 1$
\end{enumerate}
\end{theorem}
\begin{proof}
The proofs are more or less the same as for the usual Sobolev spaces $W^{k,p}(\O)$, but using the result
\begin{equation}
\norm{f g}_{L^1(\O)} \leq \norm{f}_{L^{(p,q)}(\O)} \norm{g}_{L^{(p',q')}(\O)} \qquad \frac{1}{p}+\frac{1}{p'}=1 \quad \frac{1}{q}+\frac{1}{q'} = 1
\end{equation}
proven in \cite{ONeil} instead of the usual H\"older's inequality. We will refer to Adams and Fournier \cite{adams}. Completeness follows like in 3.3 using the completeness of $L^{(p,q)}(\O)$. The density of test functions follows by proving 3.22, 3.16, 2.29 and 2.19. Mimicking the proof of theorem 2.19 requires the density of simple functions, which requires $q<\infty$.

The density of the restriction of test functions to $\O$, the property that $\abs{x}^{k-n} \in L^{(\frac{n}{n-k},\infty)}$ and 4.15 imply the embedding $\operatorname{Id} : W^{k,(\frac{n}{k},1)}(\O) \hookrightarrow L^\infty(\O)$, and hence also into $BC(\O)$. We still need to show that the elements can be extended uniquely to the boundary. But this follows directly from the fact that elements in $W^{k,(\frac{n}{k},1)}(\O)$ can be approximated by restrictions of test functions, and those can be extended uniquely.
\end{proof}

\begin{remark}
The idea for such spaces did arise quite naturally after proving the estimate in theorem \ref{BIGTHM}. The fact that $W^{1,p} \hookrightarrow C^{1-2/p}$ for $p>2$ and $W^{1,2} \hookrightarrow \cap_{q<\infty} L^q$ gives a natural hint for the embedding proved here because $L^{2+\epsilon}_{loc} \hookrightarrow L^{(2,1)}_{loc} \hookrightarrow L^2_{loc}$. Moreover, the embeddings into $L^\infty$ can be proven using integrals with kernels having weak singularities. Operators with kernels in weak $L^p$ spaces work well with Lorentz spaces because of O'Neil's inequality \cite{ONeil}.
\end{remark}

\begin{remark}
The idea of combining Lorentz and Sobolev spaces is not new. See for example \cite{kkm} and \cite{neves}. In fact, in the first one, the authors consider functions $f \in W^{1,1}_{loc}$ such that $\nabla f \in L^{(n,1)}$, and show that those are continuous. They also prove that if $X$ is a rearrangement invariant Banach space, then $\{u \mid \nabla u \in X\} \hookrightarrow AC^n$ if and only if $X \hookrightarrow L^{(n,1)}$. Here $AC^n$ denotes the space of $n$-absolutely continuous functions, see \cite{nAbsCont}.
\end{remark}

\begin{remark}
Is it reasonable to require both $f$ and $\nabla f$ in the same $L^{(p,q)}(\O)$? Consider for example $\abs{x}^{-a}$ in the spaces $W^{k,(p,\infty)}$. We have $\abs{x}^{-a} \in L^{(p,\infty)}$ if and only if $p = \frac{n}{a}$, but $\abs{\nabla \abs{x}^{-a}} = a \abs{x}^{-a-1} \in L^{(p^*,\infty)}$ if and only if $p^* = \frac{n}{a+1}$. This would suggest that $W^{1,(p,\infty)}$ should be defined by taking the norms $\norm{f}_{(p,\infty)} + \norm{\nabla f}_{(p^*,\infty)}$, where $\frac{1}{p^*} = \frac{1}{n} + \frac{1}{p}$. This is the famous Sobolev conjugate. One could wonder whether the Sobolev embedding theorems give a good choice for the norm also in the case of the usual $W^{1,p}$-spaces. Moreover, our results will have mixed norms like $\norm{f}_{\infty} + \norm{\nabla f}_{(2,1)}$ in the estimates in 2D. Anyway, this is not a big problem when working on a bounded domain. Hence we shall be content with having the same Lorentz space for all the derivatives.
\end{remark}

\subsection{Lorentz-Bessel potential spaces}
\begin{definition}
\index{spaces!Lorentz-Bessel potential|textbf}
\index[notation]{Hspq@$H^{s,(p,q)}$; Lorentz-Bessel potential space}
\label{HspqDef}
Let $1<p<\infty$, $1\leq q \leq \infty$ and $s \in \R$. Define the \emph{Lorentz-Bessel potential space} $H^{s,(p,q)}(\R^n)$ as follows:
\begin{equation}
H^{s,(p,q)}(\R^n) = \left\{ f \in L^{(p,q)}(\R^n) \middle\lvert \F^{-1} \left((1+\abs{\cdot}^2)^{s/2} \F f\right) \in L^{(p,q)}(\R^n) \right\}
\end{equation}
with norm
\begin{equation}
\norm{f}_{H^{s,(p,q)}(\R^n)} = \norm{\F^{-1} \left((1+\abs{\cdot}^2)^{s/2} \F f\right)}_{L^{(p,q)}(\R^n)}.\end{equation}
If $\O \subset \R^n$ is an open set, then $H^{s,(p,q)}(\O)$ consists of the restrictions of $H^{s,(p,q)}(\R^n)$ distributions to $\O$. Hence
\begin{equation}
H^{s,(p,q)}(\O) = \{ f_{|\O} \mid f \in H^{s,(p,q)}(\R^n) \}, \quad \norm{f}_{H^{s,(p,q)}(\O)} = \inf_{g_{|\O} = f} \norm{g}_{H^{s,(p,q)}(\R^n)}.
\end{equation}
The elements of these spaces are considered as distributions in $\R^n$ and $\O$ respectively.
\end{definition}

\begin{theorem}
\label{HspqWelldefined}
Let $\O \subset \R^n$ be an open set, $s\in\R$, $1<p<\infty$ and $1\leq q \leq \infty$. Then the space $H^{s,(p,q)}(\O)$ is a well defined Banach space and we have the equivalence $W^{k,(p,q)}(\R^n) = H^{k,(p,q)}(\R^n)$ for $k \in \N$.
\end{theorem}
\begin{proof}
Completeness in the whole domain follows from the completeness of $L^{(p,q)}(\R^n)$ and standard argument like in 6.2.2 of \cite{BL}. We write $H(\R^n) = H^{s,(p,q)}(\R^n)$ and $H(\O) = H^{s,(p,q)}(\O)$. Let $(f_k) \subset H(\O)$ be a sequence such that the series
\begin{equation}
\sum_{k=0}^\infty \norm{f_k}_{H(\O)}
\end{equation}
converges in $\R$. By the definition of the norm of $H(\O)$, we have a sequence $g_k \in H(\R^n)$ such that $\norm{g_k}_{H(\R^n)} \leq 2 \norm{f_k}_{H(\O)}$ and ${g_k}_{|\O} = f_k$. Now
\begin{equation}
\sum_{k=0}^\infty \norm{g_k}_{H(\R^n)} \leq 2 \sum_{k=0}^\infty \norm{f}_{H(\O)} < \infty,
\end{equation}
and since $H(\R^n)$ is complete, the series $\sum g_k$ converges to $g \in H(\R^n)$ in the norm. Now
\begin{equation}
\norm{ \sum_{k=0}^m f_k - g_{|\O}}_{H(\O)} \leq \norm{\sum_{k=0}^m g_k - g}_{H(\R^n)} \longrightarrow 0
\end{equation}
as $m \to \infty$. Hence $\sum f_k$ converges to $g_{|\O} \in H(\O)$. All norm-convergent series converge, so $H(\O)$ is complete.

The equivalence of $W^{k,(p,q)}$ and $H^{k,(p,q)}$ follows by standard arguments, for example like in theorem 6.2.3 of \cite{BL}. The only modification is that the Mihlin multiplier theorem\index{Mihlin multiplier theorem|textbf}, 6.1.6 in \cite{BL}, gives boundedness in $L^{(p,q)}(\R^n)$ too by real interpolation.
\end{proof}

\begin{remark}
We didn't prove that $W^{k,(p,q)}(\O) = H^{k,(p,q)}(\O)$. This would require showing that there is an extension operator $W^{k,(p,q)}(\O) \to W^{k,(p,q)}(\R^n)$. This would seem to be true for regular enough $\O$, for example by using Calder\'on's construction in theorem 12 of \cite{calderonExtensionOp} and the fact that Calder\'on-Zygmund operators map $L^{(p,q)}$ to $L^{(p,q)}$ by real interpolation. Anyway, it's easy to see that $H^{k,(p,q)}(\O) \subset W^{k,(p,q)}(\O)$ by the theorem above.
\end{remark}

\begin{remark}
One could define $W^{s,(p,q)}(\O)$ by interpolation, but it is not a-priori clear whether we would get $H^{s,(p,q)}(\O)$. It is true when $\O$ has a total extension operator.
\end{remark}

\begin{remark}
All these kinds of spaces are not new. See for example \cite{kkm} and 5.1, 5.3 in \cite{neves}.
\end{remark}

\begin{definition}\index{retract|textbf}
Let $A$ and $B$ be normed linear spaces. Then $B$ is a \emph{retract of $A$}, if there are bounded linear operators $\mathscr{I}: B \to A$ and $\mathscr{P}: A \to B$ such that $\mathscr{P} \circ \mathscr{I}$ is the identity in $B$.
\begin{equation}
\begin{tikzcd}
B \arrow{dr}[swap]{\mathscr{I}} \arrow{rr}{\operatorname{Id}} &  & B \\
	& A \arrow{ur}[swap]{\mathscr{P}}&
\end{tikzcd}
\end{equation}
\end{definition}

\begin{remark}
\label{retractInterp}
This is extremely useful when interpolating. If $(B_0,B_1)$ are retracts of $(A_0,A_1)$ with the mappings $\mathscr{I}$ and $\mathscr{P}$, then $(B_0,B_1)_{[\theta]}$ is a retract of $(A_0,A_1)_{[\theta]}$ with the same mappings. This is a direct consequence of the interpolation property.
\end{remark}

\begin{theorem}
\index[notation]{Lpqlstwo@$L^{(p,q)}(\R^n,l^s_2)$}
\label{Hretract}
Let $s \in \R$, $1<p<\infty$ and $1\leq q \leq \infty$. Then $H^{s,(p,q)}(\R^n)$ is a retract of $L^{(p,q)}(\R^n, l^s_2)$. Moreover the mappings $\mathscr{I}$ and $\mathscr{P}$ do not depend on any of the parameters.
\begin{equation}
\begin{tikzcd}
H^{s,(p,q)} \arrow{dr}[swap]{\mathscr{I}} \arrow{rr}{\operatorname{Id}} & & H^{s,(p,q)} \\
 & L^{(p,q)}(\ell^s_2) \arrow{ur}[swap]{\mathscr{P}} &
\end{tikzcd}
\end{equation}
\end{theorem}
\begin{proof}
The proof is a word by word replica of theorem 6.4.3 in \cite{BL}. The only change is the fact that the Mihlin multiplier theorem\index{Mihlin multiplier theorem} gives boundedness in $L^{(p,q)}$ by real interpolation.
\end{proof}

The symbol $(\cdot,\cdot)_{[\theta]}$ represents complex interpolation as in \cite{BL}. See definition \ref{cInterpDef}. We are ready to interpolate $H^{s,(p,q)}$.
\begin{corollary}
\label{HspqInterp}\index{interpolation!Lorentz-Bessel spaces $H^{s,(p,q)}$}
Let $s_0<s_1$ be real numbers and $1<p<\infty$, $1\leq q \leq \infty$. Let $0<\theta<1$. Then
\begin{equation}
\left(H^{s_0,(p,q)}(\R^n), H^{s_1,(p,q)}(\R^n)\right)_{[\theta]} = H^{s,(p,q)}(\R^n)
\end{equation}
with equivalent norms, where $s = (1-\theta)s_0 + \theta s_1$.
\end{corollary}
\begin{proof}
Let $\mathscr{I}$ and $\mathscr{P}$ be the injection and projection of theorem \ref{Hretract}. According to remark \ref{retractInterp} the interpolation space $\left(H^{s_0,(p,q)}, H^{s_1,(p,q)}\right)_{[\theta]}$ is a retract of $\left( L^{(p,q)}(\ell^{s_0}_2), L^{(p,q)}(\ell^{s_1}_2)\right)_{[\theta]}$. By theorem \ref{lorentzInterp} and the vector-valued result of theorem 5.6.3 in \cite{BL} we know that the latter is just $L^{(p,q)}(\ell^s_2)$, of which $H^{s,(p,q)}$ is a rectract. Hence
\begin{multline}
\left(H^{s_0,(p,q)}, H^{s_1,(p,q)}\right)_{[\theta]} = \mathscr{PI} \left(H^{s_0,(p,q)}, H^{s_1,(p,q)}\right)_{[\theta]} \\
\subset \mathscr{P} \left( L^{(p,q)}(\ell^{s_0}_2), L^{(p,q)}(\ell^{s_1}_2)\right)_{[\theta]} = \mathscr{P} L^{(p,q)}(\ell^s_2) \subset H^{s,(p,q)}.
\end{multline}
On the other hand
\begin{multline}
H^{s,(p,q)} = \mathscr{PI}H^{s,(p,q)} \subset \mathscr{P} L^{(p,q)}(\ell^s_2) \\
= \mathscr{P} \left( L^{(p,q)}(\ell^{s_0}_2), L^{(p,q)}(\ell^{s_1}_2)\right)_{[\theta]} \subset \left(H^{s_0,(p,q)}, H^{s_1,(p,q)}\right)_{[\theta]}.
\end{multline}
The operators are bounded, which implies that the norms are equivalent.
\end{proof}

\begin{remark}
Without assuming that $p_0=p_1=p$ and $q_0=q_1=q$, but with $q_0,q_1 < \infty$, we would have gotten
\begin{equation}
H^{s,\left(p,p\min(\frac{q_0}{p_0}, \frac{q_1}{p_1})\right)} \hookrightarrow \left(H^{s_0,(p_0,q_0)}, H^{s_1,(p_0,q_0)}\right)_{[\theta]} \hookrightarrow H^{s,(p,q)}
\end{equation}
where $s = (1-\theta)s_0 + \theta s_1$, $\frac{1}{p} = \frac{1-\theta}{p_0} + \frac{\theta}{p_1}$ and $\frac{1}{q} = \frac{1-\theta}{q_0} + \frac{\theta}{q_1}$. This gives an equality if $\frac{p_0}{q_0} = \frac{p_1}{q_1}$.
\end{remark}

\begin{remark}
We would need an extension operator for the same result in a domain. Without it, we can still see that
\begin{equation}
H^{s,(p,q)}(\O) \hookrightarrow \left( L^{(p,q)}(\O), H^{1,(p,q)}(\O) \right)_{[s]} \hookrightarrow \left(L^{(p,q)}(\O), W^{1,(p,q)}(\O) \right)_{[s]}
\end{equation}
for $0<s<1$.
\end{remark}

\subsection{The Cauchy operator}
We start by proving some estimates for the Cauchy-operators $\Ca$, $\Cab$. They are the right inverse operators for $\db$, $\d$ respectively.

\begin{definition}
\index{Cauchy operator|textbf}
\index[notation]{C@$\Ca, \Cab$; Cauchy operators}
\label{cauchyOpDef}
For $f \in \mathscr{E}'(\C)$, we define
\begin{equation}
\Ca f = \frac{1}{\pi z} \ast f, \qquad \Cab = \frac{1}{\pi \wbar{z}} \ast f.
\end{equation}
Whenever $f$ is not compactly supported, we assume that it is in a certain normed space, in which it can be approximated by compactly supported distributions. Moreover we have to show that $\Ca$ and $\Cab$ are bounded operators in those cases.
\end{definition}

\begin{lemma}
\label{ONeilLemma1}
For $f \in L^{(2,1)}(\C)$ we have
\begin{equation}
\Ca:L^{(2,1)}(\C) \to BC(\C), \qquad \norm{\Ca f}_{BC} \leq \frac{2}{\sqrt{\pi}} \norm{f}_{L^{(2,1)}},
\end{equation}
and if $f \in L^1(\C)$ we have
\begin{equation}
\Ca:L^1(\C) \to L^{(2,\infty)}(\C), \qquad \norm{\Ca f}_{L^{(2,\infty)}} \leq \frac{2}{\sqrt{\pi}} \norm{f}_{L^1}
\end{equation}
\end{lemma}
\begin{proof}
The Cauchy operators are well-defined here because functions in $L^1$ and $L^{(2,1)}$ can be approximated by test functions according to lemma \ref{lorentzProperties}. We shall use Young's inequality for Lorentz spaces. Note that the kernel $\frac{1}{\pi z} \in L^{(2,\infty)}(\C)$ with norm $\norm{\frac{1}{\pi z}}_{L^{(2,\infty)}(\C)} = \frac{2}{\sqrt{\pi}}$.

The first estimate is given in \cite{ONeil} with proof as an exercise. Using H\"older's inequality and $\norm{k\ast f}_\infty \leq \int_0^\infty f^*(t)k^*(t) dt \leq \int_0^\infty f^{**}(t)k^{**}(t)dt$ proven in corollary 1.8 of \cite{ONeil} we get it easily:
\begin{multline}
\norm{k\ast f}_{L^\infty(\C)} \leq \int_0^\infty  k^{**}(t)f^{**}(t) dt = \int_0^\infty t^{1/2}k^{**}(t) \cdot t^{-1/2}
 f^{**}(t) dt \\
 \leq \int_0^\infty t^{1/2}f^{**}(t) \frac{dt}{t} \,\, \sup_{t>0} t^{1/2}k^{**}(t) = \norm{f}_{L^{(2,1)}(\O)} \norm{k}_{L^{(2,\infty)}(\C)}.
 \end{multline}
Test functions map to continuous functions by according to Vekua, I.6.2 \cite{vekua}. They are dense in $L^{(2,1)}(\O)$ because of theorem \ref{WkpqWelldefined}. Hence the estimate implies continuity.
 
The second estimate follows from Young's inequality and real interpolation. For $p=1,\infty$ we have $\norm{k\ast f}_p \leq \norm{k}_p \norm{f}_1$. By the properties of real interpolation between $L^p$-spaces, for example 7.26 in \cite{adams}\footnote{They write $L^{p,q}$ for our space $L^{(p,q)}$.}, we get
\begin{equation}
\norm{u}_{L^{(2,\infty)}} = \norm{u}_{(L^1,L^\infty)_{\frac{1}{2},\infty}},
\end{equation}
which implies the claim.
\end{proof}
\begin{remark}
The continuity of $\Ca \phi$, $\phi$ a test function, can be seen in another way. $\Ca$ maps $L^{(2,1)}$ locally into the space $W^{1,(2,1)}$ of \ref{WkpqWelldefined}. This space can be embedded into $BC$.
\end{remark}

We need to take care of boundary terms\index{boundary term} when considering the left inverses of $\d, \db$. The next lemma will be used for that.
\begin{lemma}
\label{boundaryInt}
Let $\O$ be a bounded Lipschitz domain and $g \in L^1(\d \O)$. Then
\begin{equation}
\norm{\frac{1}{2\pi}\int_{\d\O} \frac{g(z')}{z-z'}\eta(z')d\sigma(z')}_{L^{(2,\infty)}(\C)} \leq \pi^{-\frac{3}{2}}\norm{g}_{L^1(\partial\O)}
\end{equation}
\end{lemma}
\begin{remark}
\index[notation]{nu@$\nu$; normal vector}
\index[notation]{eta@$\eta$; complex normal vector}
We write $\nu(z)$ for the normal vector at $z$ pointing outwards the region of integration in $\Ca$. The complex one is $\eta(z) = \nu_1(z) + i \nu_2(z)$.
\end{remark}
\begin{proof}
By Minkowski's integral inequality we see that 
\begin{equation}
\norm{ \frac{1}{2\pi}\int_{\d\O} g(z') h(z-z') \eta(z')d\sigma(z') }_{L^p(\C)} \leq \frac{1}{2\pi}\norm{g}_{L^1(\partial\Omega)} \norm{h}_{L^p(\C)}
\end{equation}
for $1\leq p \leq \infty$. By the properties of real interpolation between $L^p$-spaces, 7.26 in \cite{adams}, we have
\begin{equation}
\norm{u}_{L^{(2,\infty)}} = \norm{u}_{(L^1,L^\infty)_{\frac{1}{2},\infty}}.
\end{equation}

Note that $z\mapsto \frac{1}{\pi z} $ is in $L^{(2,\infty)}(\C)$ with norm $\norm{\frac{1}{\pi z}}_{L^{(2,\infty)}(\C)} = \frac{2}{\sqrt{\pi}}$. This implies the claim by real interpolation's interpolation property applied to the operator $h \mapsto \frac{1}{2\pi}\int_{\d\O} g(z') h(\cdot-z') \eta(z') d\sigma(z')$.
\end{proof}

\begin{lemma}
\label{cauchyIntPart}\index{integration by parts|textbf}
Let $\O$ be a bounded Lipschitz domain in $\C$. The mapping $\db \Ca$ is the identity on $\mathscr{E}'(\C)$ and $\d \Ca : L^{(p,q)}(\C) \to L^{(p,q)}(\C)$ for $1<p<\infty$, $1\leq q \leq \infty$. Let $f \in W^{1,1}(\O)$. Then all the terms of the following expression are in $L^{(2,\infty)}(\C)$ and 
\begin{equation}
\Ca \chi_\O \db f (z) = \chi_\O f (z) + \frac{1}{2\pi} \int_{\partial \O} \eta(z') \frac{\Tr f(z')}{z-z'} d\sigma(z'),
\end{equation}
in the distribution sense. Here $\chi_\O g$ is always extended as zero outside of $\O$.
\end{lemma}
\begin{proof}
The first claim follows from $\db \frac{1}{\pi z} = \delta_0$, which can be proven by taking the Fourier transform, or like in Vekua \cite{vekua}, chapter I, \S 4, equation (4.9). Note that $\d \Ca$ is the well known Beurling operator\index{Beurling operator|textbf} $\Pi$\index[notation]{Pi@$\Pi, \wbar\Pi, \d\Ca, \db\Cab$; Beurling operator}, see for example Vekua \cite{vekua} chapter I, \S 9. That book gives the result $\Pi: L^p \to L^p$ for $1<p<\infty$. Using real interpolation we get it for $L^{(p,q)} \to L^{(p,q)}$, $1<p<\infty$, $1\leq q \leq \infty$.

We see that $\Ca : L^1(\C) \to L^{(2,\infty)}(\C)$ by lemma \ref{ONeilLemma1}. The boundary integral is also in $L^{(2,\infty)}(\C)$ by lemma \ref{boundaryInt} since $\Tr : W^{1,1}(\O) \to L^1(\d \O)$. We have also $f \in W^{1,1}(\O) \hookrightarrow L^2(\O) \hookrightarrow L^{(2,\infty)}(\O)$ by Sobolev embedding, so $\chi_\O f \in L^{(2,\infty)}(\C)$ when extended as zero.

Let $\phi \in C^\infty_0(\C)$. Then
\begin{multline}
\langle \db (\chi_\O f), \phi \rangle = - \langle \chi_\O f, \db \phi \rangle = - \int_\O f \db \phi dm \\
= - \int_{\d\O} \Tr f \tfrac{\eta}{2} \phi d\sigma + \int_\O \db f \phi dm = \left\langle - \tfrac{\eta}{2} \Tr f d\sigma_{\d\O} + \chi_\O \db f, \phi \right\rangle
\end{multline}
by integrating by parts. See for example Ne{\v{c}}as \cite{necas}, theorem 3.1.2. After that, we get
\begin{multline}
\chi_\O f = \delta_0 \ast (\chi_\O f) = \db \frac{1}{\pi z} \ast (\chi_\O f) = \frac{1}{\pi z} \ast \db (\chi_\O f) \\
= \frac{1}{\pi z} \ast \left( - \tfrac{\eta}{2} \Tr f d\sigma_{\d\O} + \chi_\O \db f \right) = - \frac{1}{2\pi} \int_{\partial \O} \eta(z') \frac{\Tr f(z')}{z-z'} d\sigma(z') + \Ca \chi_\O \db f,
\end{multline}
because $\chi_\O f \in \mathscr{E}'(\C)$.
\end{proof}

\vfill
\section{Using the stationary phase method}
\subsection{The main term}
\begin{lemma}[Mean-value inequality]\index{mean value inequality|textbf}
\label{MVI}
Let $f : X \to \C$, $X \subset \C$ be convex, $f \in C^1(\wbar{X})$. Then for all $x,y \in X$
\begin{equation}
\abs{f(x) - f(y)} \leq
	\begin{cases}
		\norm{\sqrt{\abs{\d_1 f}^2 + \abs{\d_2 f}^2}}_{L^\infty(X)} \abs{x - y} \\
		\sqrt{2} \norm{\sqrt{\abs{\d f}^2 + \abs{\db f}^2}}_{L^\infty(X)} \abs{x - y}
	\end{cases}.
\end{equation}
\end{lemma}

\begin{proof}
By theorem 7.20 in Rudin \cite{rudin}
\begin{equation}
\begin{split}
\lvert &f(x) - f(y)\rvert = \abs{\int_0^1 \frac{d}{dt} f \left(tx+(1-t)y\right) dt} \\
&= \abs{\int_0^1 \left( \Re \nabla f \cdot (x_1-y_1,x_2-y_2) + i \Im \nabla f \cdot (x_1-y_1,x_2-y_2) \right) dt} \\
&\leq \int_0^1 \sqrt{ \abs{\Re \nabla f}^2 + \abs{\Im \nabla f}^2 } \, \abs{x - y} dt \leq \norm{\abs{\nabla f}}_{L^\infty(X)} \abs{x - y}.
\end{split}
\end{equation}
Note that $\abs{\Re \nabla f}^2 + \abs{\Im \nabla f}^2 = \abs{\d_1 f}^2 + \abs{\d_2 f}^2 = 2\left( \abs{\d f}^2 + \abs{\db f}^2 \right)$, from which the claim follows.
\end{proof}

\begin{lemma}
\label{GaussianHolderNorm}
Let $s \geq 0$ and $\xi \in \C$. Then
\begin{equation}
\abs{1 - e^{-i(\xi^2+\wbar{\xi}^2)}} \leq 2^{1+s/2}\abs{\xi}^s.
\end{equation}
\end{lemma}
\begin{proof}
The two cases to consider are $\abs{\xi} < \frac{1}{\sqrt{2}}$ and $\abs{\xi} \geq \frac{1}{\sqrt{2}}$. We use lemma \ref{MVI} to get the first case.
\begin{multline}
\sup_{\abs{\xi} \leq 2^{-1/2}} \frac{\abs{1 - e^{-i(\xi^2 + \wbar{\xi}^2)}}}{\abs{\xi}^s} \\
\leq  \sup_{\abs{\xi} \leq 2^{-1/2}} \sqrt{2} \norm{\sqrt{\abs{-2ize^{-i(z^2+\wbar{z}^2)}}^2 + \abs{-2i\wbar{z}e^{-i(z^2+\wbar{z}^2)}}^2}}_{\mathrlap{L^\infty(B(2^{-1/2}))}} \,\,\, \abs{\xi}^{1-s} \\
\leq \sqrt{2}\cdot 2 \cdot \sqrt{2} \cdot 2^{-1/2} (2^{-1/2})^{1-s} \leq 2^{1 + s/2}
\end{multline}
The second case follows because $\xi^2 + \wbar{\xi}^2 \in \R$.
\begin{equation}
\sup_{\abs{\xi} \geq 2^{-1/2}} \frac{\abs{1 - e^{-i(\xi^2 + \wbar{\xi}^2)}}}{\abs{\xi}^s} \leq \sup_{\abs{\xi} \geq 2^{-1/2}} \frac{2}{\abs{\xi}^s} = 2^{1+s/2}
\end{equation}
\end{proof}

From now on we denote $R = (z-z_0)^2 + (\wbar{z} - \wbar{z_0})^2$,\index[notation]{R@$R$; the phase function $(z-z_0)^2+(\wbar{z}-\wbar{z_0})^2$} where $z_0$ is a point in $\C$. The stationary phase method is based on the Fourier transform of a complex Gaussian. See lemma \ref{convKernTransf} for details.

\begin{lemma}[Stationary phase]
\label{stationaryPhase}\index{stationary phase|textbf}
Let $\tau > 0$ and $s \geq 0$. If $Q\in H^{s,2}(\C)$ then
\begin{equation}
\norm{Q - \frac{2\tau}{\pi} \int_\C e^{i\tau R}Q(z)\,dm(z)}_{L^2(\C,z_0)} \leq 2 \tau^{-s/2} \norm{Q}_{H^{s,2}(\C)}.
\end{equation}
\end{lemma}
\begin{proof}
A direct calculations using the Fourier transform 
\begin{equation}
\mathscr{F}\big(\frac{2\tau}{\pi} e^{i\tau (z^2+ \wbar{z}^2)} \big)(\xi) = \frac{1}{2\pi} e^{-i \frac{\xi^2 + \wbar{\xi}^2}{16\tau}}
\end{equation}
of lemma \ref{convKernTransf} and lemma \ref{GaussianHolderNorm}:
\begin{multline}
\norm{Q - \frac{2\tau}{\pi} \int_\C e^{i\tau R}Q(z)\,dm(z)}_{L^2(\C,z_0)} = \norm{ \what{Q} - e^{-i\frac{\xi^2 + \wbar{\xi}^2}{16\tau}} \what{Q} }_{L^2(\C)} \\
\leq 4^{-s} \tau^{-s/2} \norm{\frac{\abs{1-e^{-i\big((\frac{\xi}{4\sqrt{\tau}})^2 + (\wbar{\frac{\xi}{4\sqrt{\tau}}})^2\big)}}} {\abs{\frac{\xi}{4\sqrt{\tau}}}^s} \, \abs{\xi}^s \what{Q}}_{L^2(\C)} \\
\leq 2^{1-3s/2} \tau^{-s/2} \norm{\abs{\xi}^s \what{Q}}_{L^2(\C)} \leq 2 \tau^{-s/2} \norm{Q}_{H^{s,2}(\C)}.
\end{multline}
\end{proof}

\subsection{Handling the error term}
Since the potential $q$ of our Schr\"odinger equation won't be zero, our solutions will have an error term. It has to be handled separately when using stationary phase. Note that $R$ depends only on $z_0-z$, but the integral involving the error will not be a convolution operator. That's because the error term $r$ will depend on $z_0$ too. Nevertheless, we may prove an $L^\infty$ estimate for this operator, and ignore whether $r$ or $Q$ would depend on $z_0$.

If $Q,r \in W^{1,p}$ with $p>2$, we could get the estimate
\begin{equation}
\norm{\int_\O \frac{2\tau}{\pi} e^{i\tau R} Q r (z) dm(z)}_2 \leq C_\O \norm{Q}_{1,p} \sup_{z_0} \norm{r}_{1,p}
\end{equation}
for example as in \cite{lis}, which follows ideas of Bukhgeim \cite{bukhgeim}. We could try to follow that reasoning in our case of $Q \in H^{1,(2,1)}$ and $r \in H^{1,(2,\infty)}$. It would start like this
\begin{multline}
\int 2\tau i e^{i\tau R} Qr dm = \int\frac{\d e^{i\tau R}}{z_0-z} Qr dm = \pi \chi_\O e^{i\tau R} Q r \\
+ \frac{1}{2} \int_{\d\O} \eta(z') \frac{e^{i\tau R(z'-z_0)} Qr(z')}{z_0-z'} d\sigma(z')  - \int \frac{e^{i\tau R}}{z_0 -z } \d(Qr) dm
\end{multline}
where we have used lemma \ref{cauchyIntPart}. It is true even when $r$ depends on $z_0$. But then we would have to estimate $r_{z_0} (z_0)$ in the first term. This is no problem when $r \in L^\infty$, which we will have. However, for reasons related to the limitations of the interpolation result of theorem \ref{lorentzInterp}, we may only use the $H^{1,(2,\infty)}$-norm of $r$. This space does not embed into $L^\infty$. Hence, if we want to follow this path, we should analyze what happens in theorem \ref{BIGTHM} and corollary \ref{corollary1} as the outer variable approaches $z_0$. In any case, that's not needed here. The final rate $\tau^{1-s/3}$ given by the next approach is as good as $\tau^{1-s}$, because we are more interested in the case $s \to 0$.

\begin{theorem}
\index{error term estimate|textbf}
\label{errorTermIntegral}
Let $\O\subset\C$ be a bounded Lipschitz domain, $\tau > 0$, $0<s<1$ and $z_0\in\C$. If $Q \in H^{s,(2,1)}(\O)$ and $r \in H^{s,(2,\infty)}(\O)$, then
\begin{equation}
\abs{\int_\O \frac{2\tau}{\pi} e^{i\tau R} Q r (z) dm(z)} \leq C_\O \tau^{1-s/3} \norm{Q}_{s,(2,1)} \norm{r}_{s,(2,\infty)}.
\end{equation}
\end{theorem}
\begin{proof}
The claim follows by using complex interpolation on the multilinear mapping
\begin{equation}
T: (Q, r) \mapsto \int_\O \frac{2\tau}{\pi} e^{i\tau R} Q r (z) dm(z), \quad T: H^{s,(2,1)}(\O) \times H^{s,(2,\infty)}(\O) \to \C .
\end{equation}
Consider the case corresponding to $s=0$ first. By H\"older's inequality for Lorentz spaces, theorem 3.5 in \cite{ONeil}, we get directly
\begin{equation}
\abs{\int_\O \frac{2\tau}{\pi} e^{i\tau R} Q r (z) dm(z)} \leq \frac{2\tau}{\pi} \norm{Q r}_{1} \leq C_\O \tau \norm{Q}_{(2,1)} \norm{r}_{(2,\infty)}.
\end{equation}

Take $h \in C^\infty_0(\O)$ as given by lemma \ref{h2} to prove the other limiting case. We split the integral like this
\begin{multline}
\int_\O \frac{2\tau}{\pi} e^{i\tau R} Q r (z) dm(z) \\
= \int_\O \frac{2\tau}{\pi} e^{i\tau R}(\wbar{z}-\wbar{z_0})h Q r (z) dm(z) + \int_\O \frac{2\tau}{\pi} e^{i\tau R} (1 - (\wbar{z}-\wbar{z_0})h)Q r (z) dm(z) \\
= \frac{1}{i\pi} \int_\O \db e^{i\tau R} h Q r (z) dm(z) + \int_\O \frac{2\tau}{\pi} e^{i\tau R} (1 - (\wbar{z}-\wbar{z_0})h)Q r (z) dm(z).
\end{multline}
Estimate the second term first. The generalized H\"older's inequality implies
\begin{multline}
\abs{\int_\O \frac{2\tau}{\pi} e^{i\tau R} (1 - (\wbar{z}-\wbar{z_0})h)Q r (z) dm(z)} \leq C_\O \tau \norm{\left( 1 - (\wbar{z}-\wbar{z_0})h\right) Q r}_{1} \\
\leq C_\O \tau \norm{1 - (\wbar{z}-\wbar{z_0})h}_{(2,1)} \norm{Qr}_{(2,\infty)} \\
\leq C_\O \tau \norm{1 - (\wbar{z}-\wbar{z_0})h}_{(2,1)} \norm{Q}_{\infty} \norm{r}_{(2,\infty)} \\
\leq C_\O \tau \norm{1 - (\wbar{z}-\wbar{z_0})h}_{(2,1)} \norm{Q}_{1,(2,1)} \norm{r}_{1,(2,\infty)} 
\end{multline}
by the embedding $H^{1,(2,1)}(\O) \hookrightarrow W^{1,(2,1)}(\O) \hookrightarrow L^\infty(\O)$ of theorems \ref{WkpqWelldefined} and \ref{HspqWelldefined}. Integrate the first term by parts. Then
\begin{multline}
\abs{ \int_\O \db e^{i\tau R} h Q r (z) dm(z)} = \abs{ \int_\O e^{i\tau R} \db(h Q r )(z) dm(z) } \\
\leq \norm{\db h}_{(2,1)} \norm{Qr}_{(2,\infty)} + \norm{h}_\infty \left( \norm{\db Q}_{(2,1)} \norm{r}_{(2,\infty)} + \norm{Q}_{(2,1)} \norm{\db r}_{(2,\infty)}\right) \\
\leq \norm{\db h}_{(2,1)} \norm{Q}_\infty \norm{ r}_{(2,\infty)} + 2 \norm{h}_\infty \norm{Q}_{1,(2,1)} \norm{r}_{1,(2,\infty)} \\
\leq C_\O \left( \norm{\db h}_{(2,1)} + \norm{h}_\infty \right) \norm{Q}_{1,(2,1)} \norm{r}_{1,(2,\infty)}
\end{multline}
because of $H^{1,(2,1)}(\O) \hookrightarrow W^{1,(2,1)}(\O) \hookrightarrow L^\infty(\O)$. Corollary \ref{h2} gives
\begin{equation}
\tau \norm{1 - (\wbar{z}-\wbar{z_0})h(z)}_{L^{(2,1)}(\O)} + \norm{h}_{L^\infty(\O)} + \norm{\db h}_{L^{(2,1)}(\O)} \leq C_\O \tau^{2/3}.
\end{equation}
The claim follows by interpolation as $H^{s,(p,q)}(\O) \hookrightarrow ( L^{(p,q)}(\O), H^{1,(p,q)}(\O) )_{[s]}$. This can be seen as follows: Let $f \in H^{s,(p,q)}(\O)$ and take $g \in H^{s,(p,q)}(\C)$ such that $g_{|\O} = f$ and $\norm{g}_{H^{s,(p,q)}(\C)} \leq 2 \norm{f}_{H^{s,(p,q)}(\O)}$. Now
\begin{multline}
\norm{f}_{( L^{(p,q)}(\O), H^{1,(p,q)}(\O) )_{[s]}} = \norm{ g_{|\O} }_{( L^{(p,q)}(\O), H^{1,(p,q)}(\O) )_{[s]}} \\
\leq \norm{g}_{ ( L^{(p,q)}(\C), H^{1,(p,q)}(\C) )_{[s]}} = \norm{g}_{H^{s,(p,q)}(\C)} \leq 2 \norm{f}_{H^{s,(p,q)}(\O)}.
\end{multline}
The fact that $(\C,\C)_{[s]} = \C$ is the last small missing piece of the proof.
\end{proof}

\vfill
\section{Bukhgeim type solutions}
\subsection{A Carleman estimate}
\label{carlemanSubsection}
Remember that we write $R = (z-z_0)^2 + (\wbar{z} - \wbar{z_0})^2$.\index[notation]{R@$R$; the phase function $(z-z_0)^2+(\wbar{z}-\wbar{z_0})^2$} The next theorem is the heart of this whole thesis. The primary goal is to have the right-hand side vanish as fast as possible as $\tau$ grows. This requirement makes us study the function spaces $H^{s,(p,q)}$, which were the basis of chapter \ref{newSpacesSection}. The rest is quite straightforward after the next theorem. Continue by proving estimates for $\Ca(e^{-i\tau R}\chi_\O \Cab( e^{i\tau R}\chi_\O q f))$, use them to find solutions to $(\Delta + q)e^{i\tau (z-z_0)^2}f = 0$ and use stationary phase to estimate $\norm{q_1 - q_2}$ using the boundary data. Only technical details prevent this from being trivial.

\begin{theorem}[The main Carleman estimate]
\index{Carleman estimate|textbf}
\label{BIGTHM}
Let $\O \subset \C$ bounded and Lipschitz, $z_0\in\C$ and $\tau > 1$. If $a \in BC(\wbar{\O}) $ and $\db a \in L^{(2,1)}(\O)$ then $\Ca(e^{-i\tau R} \chi_\O a) \in L^{(2,\infty)}(\C) \cap BC(\C)$ and we have the norm estimates
\begin{equation}
\begin{split}
\norm{\Ca(e^{-i\tau R}\chi_\O a)}_{L^{(2,\infty)}(\C)} &\leq C_\O \tau^{-1} (1+\ln \tau) \Big(\norm{\db a}_{L^{(2,1)}(\O)} + \norm{a}_{BC(\wbar{\O})}\Big), \\
\norm{\Ca(e^{-i\tau R}\chi_\O a)}_{BC(\C)} &\leq C_\O \tau^{-1/3} \Big(\norm{\db a}_{L^{(2,1)}(\O)} + \norm{a}_{BC(\wbar{\O})}\Big).
\end{split}
\end{equation}
\end{theorem}
\begin{proof}
The mapping properties of $\Ca(e^{i\tau R}\chi_\O \cdot)$ follow from the corresponding mapping properties of each term, all given by the same lemmas that imply the norm estimates used here.

Let $h \in W^{1,1}(\O)$. Now
\begin{multline}
\label{introduceH}
\Ca\big(e^{i\tau R} \chi_\O a\big) = \Ca\big(e^{i\tau R} (1-(\wbar{\cdot}-\wbar{z_0})h) \chi_\O a \big) +  \Ca\big(e^{i\tau R}(\wbar{\cdot}-\wbar{z_0})h \chi_\O a \big) \\
= \Ca\big(e^{i\tau R} (1-(\wbar{\cdot}-\wbar{z_0})h) \chi_\O a \big) + \frac{1}{2i\tau}\Ca\big(\db(e^{i\tau R}) h \chi_\O a \big)
\end{multline}
We get
\begin{multline}\index{integration by parts}
\label{intByPartsH}
\Ca\big(\chi_\O \db (e^{i\tau R}) h a\big) = \chi_\O e^{i\tau R} h(z) a  \\
+ \frac{1}{2\pi} \int_{\d \O} \eta(z') \frac{e^{i\tau R} \Tr h(z') a(z')}{z-z'} d\sigma(z') \\
- \Ca(\chi_\O e^{i\tau R} \db h a) - \Ca(\chi_\O e^{i\tau R} h \db a)
\end{multline}
by lemma \ref{cauchyIntPart} because $e^{i\tau R}ha \in W^{1,1}(\O)$.

Take $h$ as in lemma \ref{h1} to prove the first estimate. By lemma \ref{ONeilLemma1}
\begin{multline}
\norm{\Ca\big(e^{i\tau R} (1-(\wbar{\cdot}-\wbar{z_0})h) \chi_\O a \big)}_{L^{(2,\infty)}(\C)} \\
\leq \frac{2}{\sqrt{\pi}} \norm{(1- (\wbar{\cdot}-\wbar{z_0})h) a}_{L^1(\O)} \\
\leq \frac{2}{\sqrt{\pi}} \norm{1-(\wbar{\cdot}-\wbar{z_0})h}_{L^1(\O)} \norm{a}_{L^\infty(\O)}.
\end{multline}
Next
\begin{equation}
\norm{\chi_\O e^{i\tau R} h(z) a}_{L^{(2,\infty)}(\C)} \leq \norm{h}_{L^{(2,\infty)}(\O)} \norm{a}_{L^\infty(\O)},
\end{equation}
and by lemma \ref{boundaryInt}
\begin{multline}
\norm{\frac{1}{2\pi} \int_{\d \O} \eta(z') \frac{e^{i\tau R} \Tr h(z') a(z')}{z-z'} d\sigma(z')}_{L^{(2,\infty)}(\C)} \leq \pi^{-\frac{3}{2}} \norm{ \Tr h a}_{L^1(\d\O)} \\
\leq \pi^{-\frac{3}{2}} \norm{h}_{L^1(\d\O)} \norm{a}_{L^\infty(\d\O)} \leq \pi^{-\frac{3}{2}} T_\O \norm{h}_{W^{1,1}(\O)} \norm{a}_{BC(\wbar{\O})},
\end{multline}
where $T_\O < \infty$ is the norm of the trace mapping $\Tr : W^{1,1}(\O) \to L^1(\d\O)$. Again, by lemma \ref{ONeilLemma1}, we get
\begin{equation}
\norm{\Ca(\chi_\O e^{i\tau R} \db h a)}_{L^{(2,\infty)}(\C)} \leq \frac{2}{\sqrt{\pi}} \norm{\db h}_{L^1(\O)} \norm{a}_{L^\infty(\O)},
\end{equation}
and according to \cite{ONeil} we have  we have $\norm{ab}_1 \leq \norm{a}_{(2,\infty)} \norm{b}_{(2,1)}$. Hence
\begin{equation}
\norm{\Ca(\chi_\O e^{i\tau R} h \db a)}_{L^{(2,\infty)}(\C)} \leq \frac{2}{\sqrt{\pi}} \norm{h}_{L^{(2,\infty)}(\O)} \norm{\db a}_{L^{(2,1)}(\O)}.
\end{equation}
Combining everything, using the Sobolev embedding $W^{1,1} \hookrightarrow L^2 \hookrightarrow L^{(2,\infty)}$ and the inequality
\begin{equation}
\tau\norm{1-(\wbar{z}-\wbar{z_0})h}_{L^1(\O)} + \norm{h}_{W^{1,1}(\O)} \leq C_\O (1+\ln \tau)
\end{equation}
of lemma \ref{h1} gives the first estimate.

\smallskip
For the second one, let $h \in C^\infty_0(\O)$ be as in corollary \ref{h2}. Then $\chi_\O h = h$. Continue from \eqref{introduceH} and \eqref{intByPartsH}. The boundary term vanishes,
\begin{multline}
\norm{\Ca\big(e^{i\tau R} (1-(\wbar{\cdot}-\wbar{z_0}) h) a \big)}_{BC(\C)} \\
\leq \frac{2}{\sqrt{\pi}} \norm{(1- (\wbar{\cdot}-\wbar{z_0})h) a}_{L^{(2,1)}(\O)} \\
\leq \frac{2}{\sqrt{\pi}} \norm{1-(\wbar{\cdot}-\wbar{z_0})h}_{L^{(2,1)}(\O)} \norm{a}_{L^\infty(\O)},
\end{multline}
\begin{equation}
\norm{\chi_\O e^{i\tau R} h(z) a}_{BC(\C)} \leq \norm{h}_{BC(\O)} \norm{a}_{BC(\O)},
\end{equation}
\begin{equation}
\norm{\Ca(e^{i\tau R} \db h a)}_{BC(\C)} \leq \frac{2}{\sqrt{\pi}} \norm{\db h}_{L^{(2,1)}(\O)} \norm{a}_{BC(\wbar{\O})},
\end{equation}
and
\begin{equation}
\norm{\Ca(e^{i\tau R} h \db a)}_{BC(\C)} \leq \frac{2}{\sqrt{\pi}} \norm{h}_{L^\infty(\O)} \norm{\db a}_{L^{(2,1)}(\O)}.
\end{equation}
The estimate
\begin{equation}
\tau \norm{1 - (\wbar{z}-\wbar{z_0})h(z)}_{L^{(2,1)}(\O)} + \norm{h}_{L^\infty(\O)} + \norm{\db h}_{L^{(2,1)}(\O)} \leq C_\O \tau^{2/3}
\end{equation}
of corollary \ref{h2} gives the rest.
\end{proof}

\begin{remark}
Dependency and measurability on $z_0$ will be taken care of later. It will turn out that the operator is continuous with respect to it. Actually, we would like to have the dependence in $L^2(\C)$ or $L^{(2,\infty)}(\C)$ for non-compactly supported potentials.
\end{remark}

\begin{remark}
It seems possible to get the map $W^{1,p} \to L^p$ with exponent $\tau^{-\frac{1}{2} - \frac{1}{p}}$ (no logarithm) when $p>2$. But it would require Bloch spaces, $BMO$ in a domain and a related interpolation identity. See section \ref{DoingInWsp}.
\end{remark}

\begin{remark}
This is a Carleman estimate for $\db$. Write $r = e^{i\tau R} \Ca(e^{-i\tau R} a)$ and consider all the derivatives in $\mathscr{D}'(\O)$, where $\chi_\O \equiv 1$. Now
\begin{equation}
a = e^{i\tau R} \db e^{-i\tau R} r = e^{i\tau (\wbar{z}-\wbar{z_0})^2} \db e^{-i\tau (\wbar{z}-\wbar{z_0})^2} r.
\end{equation}
Hence we have the following Carleman estimates\index{Carleman estimate}:
\begin{equation}
\begin{split}
\norm{r}_{(2,\infty)} &\leq C_\O \tau^{-1}(1+\ln \tau) \norm{e^{i\tau (\wbar{z}-\wbar{z_0})^2} \db e^{-i\tau (\wbar{z}-\wbar{z_0})^2} r}_{1,(2,1)}\\
\norm{r}_{C^0} &\leq C_\O \tau^{-1/3} \norm{e^{i\tau (\wbar{z}-\wbar{z_0})^2} \db e^{-i\tau (\wbar{z}-\wbar{z_0})^2} r}_{1,(2,1)}\\
\end{split}
\end{equation}
\end{remark}

\bigskip
\begin{corollary}
\label{corollary1}
Let $\O$ be a bounded Lipschitz domain, $z_0\in\C$ and $\tau>1$. Let $q\in L^{(2,1)}(\O)$ or $q\in W^{1,(2,1)}(\O)$. Write $S f = \Ca\big( e^{-i\tau R}\chi_\O \Cab(e^{i\tau R}\chi_\O qf)\big)$. Then
\begin{equation}
\begin{split}
&\norm{Sf}_{L^{(2,\infty)}(\C)} \leq C_\O \tau^{-1} (1+\ln \tau) \norm{q}_{L^{(2,1)}(\O)} \norm{f}_{L^\infty(\O)}, \\
&\norm{Sf}_{H^{1,(2,\infty)}(\C)} \leq C_\O \tau^{-1} (1+\ln \tau) \norm{q}_{W^{1,(2,1)}(\O)} \norm{f}_{W^{1,(2,1)}(\O)}, \\
&\norm{Sf}_{BC(\C)} \leq C_\O \tau^{-1/3} \norm{q}_{L^{(2,1)}(\O)} \norm{f}_{L^\infty(\O)}, \\
&\norm{Sf}_{H^{1,(2,1)}(\O)} \leq C_\O \tau^{-1/3} \norm{q}_{W^{1,(2,1)}(\O)} \norm{f}_{W^{1,(2,1)}(\O)},
\end{split}
\end{equation}
with corresponding mapping properties.
\end{corollary}
\begin{proof}
The mapping properties follow from those of theorem \ref{BIGTHM} and lemma \ref{ONeilLemma1}. We will need the facts that $\d \Ca, \db \Cab : L^{(p,q)}(\C) \to L^{(p,q)}(\C)$ and that $\d \Cab, \db \Ca$ are the identity in $\mathscr{E}'(\C)$. These are given by lemma \ref{cauchyIntPart}. We can then proceed. If there's no domain in the index of the norm, then that norm is taken in $\O$. Else it is taken where denoted. Now
\begin{multline}
\norm{\Ca\big( e^{-i\tau R}\chi_\O\Cab(e^{i\tau R}\chi_\O qf)\big)}_{L^{(2,\infty)}(\C)} \\
\leq C_\O \tau^{-1}(1+\ln\tau)\big( \norm{\db\Cab(e^{i\tau R} \chi_\O q f)}_{(2,1)} + \norm{\Cab(e^{i\tau R}\chi_\O qf)}_{BC(\wbar{\O})} \big) \\
\leq C_\O \tau^{-1}(1+\ln\tau) \norm{q}_{(2,1)}\norm{f}_\infty,
\end{multline}
and using the second estimate of theorem \ref{BIGTHM} instead of the first one, we get
\begin{equation}
\norm{\Ca\big( e^{-i\tau R}\chi_\O \Cab(e^{i\tau R}\chi_\O qf)\big)}_{BC(\C)} \leq C_\O \tau^{-1/3} \norm{q}_{(2,1)}\norm{f}_\infty.
\end{equation}

Consider the cases where there's a derivative on the left hand side. We will need the identity $W^{1,(2,\infty)}(\C) = H^{1,(2,\infty)}(\C)$ of theorem \ref{HspqWelldefined} and the continuous embedding $L^{(2,\infty)}(\O) \hookrightarrow L^1(\O)$. It is true because $m(\O) < \infty$. Then
\begin{multline}
\norm{\Ca\big(e^{-i\tau R} \chi_\O \Cab(e^{i\tau R}\chi_\O qf)\big)}_{W^{1,(2,\infty)}(\C)} \leq C_\O \big(\norm{\Cab(e^{i\tau R}\chi_\O qf)}_{1} \\
+ \norm{\d\Ca(e^{-i\tau R} \chi_\O \Cab(e^{i\tau R}\chi_\O qf))}_{L^{(2,\infty)}(\C)} + \norm{e^{-i\tau R} \chi_\O \Cab(e^{i\tau R}\chi_\O qf)}_{L^{(2,\infty)}(\C)} \big)\\
\leq C_\O \norm{\Cab(e^{i\tau R}\chi_\O qf)}_{L^{(2,\infty)}(\C)} \leq C_\O \tau^{-1}(1+\ln\tau) \norm{q}_{1,(2,1)} \norm{f}_{1,(2,1)}.
\end{multline}
The last estimate requires a technical trick since we haven't shown that $H^{1,(2,1)}(\O) = W^{1,(2,1)}(\O)$. Let $\phi \in C^\infty_0(\C)$ be constant $1$ on $B(0,R) \supset \O$. Note that $BC(X) \hookrightarrow L^{(2,1)}(X)$ whenever $m(X)<\infty$, so
\begin{multline}
\norm{\Ca\big(e^{-i\tau R} \chi_\O \Ca(e^{i\tau R}\chi_\O qf)\big)}_{1,(2,1)} \leq  \norm{\phi \Ca\big(e^{-i\tau R} \chi_\O \Ca(e^{i\tau R}\chi_\O qf)\big)}_{H^{1,(2,1)}(\C)} \\
\leq C \norm{\phi \Ca\big(e^{-i\tau R} \chi_\O \Ca(e^{i\tau R}\chi_\O qf)\big)}_{W^{1,(2,1)}(\C)} \\
\leq C_{\O,\phi} \Big( \norm{\Ca\big(e^{-i\tau R} \chi_\O \Ca(e^{i\tau R}\chi_\O qf)\big)}_{BC(\supp \phi)} \\
+\norm{\d\Ca\big( e^{-i\tau R} \chi_\O \Cab(e^{i\tau R}\chi_\O qf)\big)}_{L^{(2,1)}(\C)} 
+ \norm{e^{-i\tau R}\chi_\O \Cab(e^{i\tau R}\chi_\O qf)}_{L^{(2,1)}(\C)} \Big) \\
\leq C_{\O,\phi} \norm{\Cab(e^{i\tau R}\chi_\O qf)}_{(2,1)} \leq C_{\O,\phi} \norm{\Cab(e^{i\tau R}\chi_\O qf)}_{BC(\O)} \\
\leq C_{\O,\phi} \tau^{-1/3} \norm{q}_{1,(2,1)} \norm{f}_{1,(2,1)},
\end{multline}
by theorems \ref{BIGTHM} and \ref{WkpqWelldefined}. The cut-off function $\phi$ may be chosen based only on $\O$, so $C_{\O,\phi}$ depends only on the domain.
\end{proof}

\begin{remark}
If $m(\O) = \infty$, then it would make sense to define spaces $\tilde{W}^{1,(2,1)} = \{f \mid f\in BC, \nabla f \in L^{(2,1)}\}$ to avoid the use of the embedding $BC \subset L^{(2,1)}$. But then, on the other hand, we run into problems finding which space $X$ maps $\Ca : X \to L^{(2,1)}$.
\end{remark}

\begin{remark}
The estimate can be called a Carleman estimate for the Laplacian. Write $r = \Ca(e^{-i\tau R} \Cab(e^{i\tau R} qf))$ and consider all the derivatives in $\mathscr{D}'(\O)$ where $\chi_\O \equiv 1$. Now
\begin{equation}
qf = e^{-i\tau R} \d e^{i\tau R} \db r = \tfrac{1}{4} e^{-i\tau (z-z_0)^2} \Delta e^{i\tau(z-z_0)^2} r.
\end{equation}
Hence we have the following Carleman estimates:\index{Carleman estimate}
\begin{equation}
\label{carlemanEstimateLapl}
\begin{split}
&\norm{r}_{(2,\infty)} \leq C_\O \tau^{-1} (1+\ln \tau) \lVert{e^{-i\tau (z-z_0)^2} \Delta e^{i\tau(z-z_0)^2} r}\rVert_{(2,1)}, \\
&\norm{r}_{1,(2,\infty)} \leq C_\O \tau^{-1} (1+\ln \tau) \lVert{e^{-i\tau (z-z_0)^2} \Delta e^{i\tau(z-z_0)^2} r}\rVert_{1,(2,1)}, \\
&\norm{r}_{BC} \leq C_\O \tau^{-1/3} \lVert{e^{-i\tau (z-z_0)^2} \Delta e^{i\tau(z-z_0)^2} r}\rVert_{(2,1)}, \\
&\norm{r}_{1,(2,1)} \leq C_\O \tau^{-1/3} \lVert{e^{-i\tau (z-z_0)^2} \Delta e^{i\tau(z-z_0)^2} r}\rVert_{1,(2,1)}.
\end{split}
\end{equation}
\end{remark}

\begin{definition}
\index[notation]{Ms@$M^s(\O)$}
Let $\O \subset \C$ be a Lipschitz domain. For $ 0 < s < 1$ we write $M^s(\O) = \left( BC(\wbar\O), H^{1,(2,1)}(\O) \right)_{[s]}$.
\end{definition}
\begin{remark}
This is a well defined Banach space since both $BC$ and $W^{1,(2,1)}$ are Banach spaces that can be embedded into $BC(\wbar\O)$ by theorem \ref{WkpqWelldefined}.
\end{remark}

\begin{corollary}
\label{corollary2}
Let $\O \subset \C$ be a bounded Lipschitz domain, $\tau > 1$, $z_0\in\C$ and $0<s<1$. Let $q \in H^{s,(2,1)}(\O)$. Then
\begin{equation}
\begin{split}
\norm{ \Ca\big( e^{-i\tau R}\chi_\O \Cab(e^{i\tau R}\chi_\O qf)\big) }_{H^{s,(2,\infty)}(\O)} &\leq C_\O \tau^{-1}(1+\ln \tau) \norm{q}_{H^{s,(2,1)}(\O)} \norm{f}_{M^s(\O)} \\
\norm{ \Ca\big( e^{-i\tau R}\chi_\O \Cab(e^{i\tau R}\chi_\O qf)\big) }_{M^s(\O)} &\leq C_\O \tau^{-1/3} \norm{q}_{H^{s,(2,1)}(\O)} \norm{f}_{M^s(\O)}
\end{split}
\end{equation}
with similar mapping properties. The map $z_0 \mapsto \Ca\big( e^{-i\tau R}\chi_\O \Cab(e^{i\tau R}\chi_\O qf_{z_0})\big)$ is in
\begin{equation}
BC\big(\wbar\O, W^{1,2}(\O) \cap H^{s,(2,\infty)}(\O) \cap M^s(\O)\big)
\end{equation}
for each $f:(z,z_0) \mapsto f_{z_0}(z)$ bounded and continuous $\wbar\O\times\wbar\O \to \C$.
\end{corollary}
\begin{proof}
All the norms with $W$ on the right-hand side of corollary \ref{corollary1} can be estimated above by norms with $H$. This follows from theorem \ref{HspqWelldefined}. The second estimate follows directly from the definition of $M^s(\O)$, corollary \ref{corollary1}, the inclusion $H^{s,(2,1)}(\O) \hookrightarrow \left( L^{(2,1)}(\O), H^{1,(2,1)}(\O) \right)_{[s]}$ and the multilinear interpolation property of complex interpolation.

The first one requires a bit more careful considerations because we haven't shown that $\left( L^{(2,\infty)}(\O), H^{1,(2,\infty)}(\O) \right)_{[s]} \hookrightarrow H^{s,(2,\infty)}(\O)$.  Interpolation tells us that the operator maps $M^s(\O) \to \left( L^{(2,\infty)}(\C), H^{1,(2,\infty)}(\C) \right)_{[s]} = H^{s,(2,\infty)}(\C)$. This is because of theorem \ref{HspqWelldefined}. Now
\begin{equation}
\norm{g_{|\O}}_{H^{s,(2,\infty)}(\O)} = \inf_{\substack{G\in H^{s,(2,\infty)}(\C) \\ G_{|\O} = g_{|\O}}} \norm{G}_{H^{s,(2,\infty)}(\C)} \leq \norm{g}_{H^{s,(2,\infty)}(\C)}
\end{equation}
for any $g \in H^{s,(2,\infty)}(\C)$. The estimate follows.

It's left to prove pointwise continuity $\wbar\O \to H^{1,(2,1)}(\O)$. This is because of the chains of bounded mappings
\begin{equation}
\begin{array}{c}
H^{1,(2,1)}(\O) \hookrightarrow W^{1,(2,1)}(\O) \hookrightarrow W^{1,2}(\O) \\
H^{1,(2,1)}(\O) \hookrightarrow  BC(\wbar{\O}) \cap H^{1,(2,1)}(\O) \hookrightarrow M^s(\O), \\
H^{1,(2,1)}(\O) \hookrightarrow H^{1,(2,\infty)}(\O) \hookrightarrow H^{s,(2,\infty)}(\O),
\end{array}
\end{equation}
which follow from $L^{(2,1)}(\C) \hookrightarrow L^{(2,\infty)}(\C)$ and $H^{1,(2,\infty)}(\C) \hookrightarrow H^{s,(2,\infty)}(\C)$.

Let $z_0,z_1 \in \wbar\O$ and write $f_j = f_{z_j}$ and $R_j = (z-z_j)^2 + (\wbar{z} - \wbar{z_j})^2$, where $z$ is the variable being operated on. Then, proceed as in the proof of corollary \ref{corollary1}. Take $\phi \in C^\infty_0(\C)$ such that $\phi \equiv 1$ on $B(0,R) \supset \O$. Then
\begin{multline}
\label{corollary2eqn1}
\norm{ \Ca\big( e^{-i\tau R_0}\chi_\O \Cab(e^{i\tau R_0}\chi_\O qf_0)\big) - \Ca\big( e^{-i\tau R_1}\chi_\O \Cab(e^{i\tau R_1}\chi_\O qf_1)\big) }_{H^{1,(2,1)}(\O)} \\
\leq \norm{ \phi \Ca\big( e^{-i\tau R_0}\chi_\O \Cab(e^{i\tau R_0}\chi_\O qf_0)\big) - \phi \Ca\big( e^{-i\tau R_1}\chi_\O \Cab(e^{i\tau R_1}\chi_\O qf_1)\big) }_{W^{1,(2,1)}(\C)} \\
\leq \norm{ \phi\Ca \left( e^{-i\tau R_0} \chi_\O \left( \Cab(e^{i\tau R_0}\chi_\O q f_0) - \Cab (e^{i\tau R_1}\chi_\O qf_1) \right) \right) }_{W^{1,(2,1)}(\C)} \\
+ \norm{ \phi\Ca \left( (e^{-i\tau R_0} - e^{-i\tau R_1} ) \chi_\O \Cab (e^{i\tau R_1} \chi_\O q f_1) \right) }_{W^{1,(2,1)}(\C)}.
\end{multline}
Next note that $\phi\Ca \chi_\O, \phi\Cab \chi_\O : L^\infty(\O) \to W^{1,(2,1)}(\C)$ by the boundedness of $\O$ and lemmas \ref{ONeilLemma1} and \ref{cauchyIntPart}. The first term in \eqref{corollary2eqn1} can be estimates as
\begin{multline}
\ldots \leq C_{\O,\phi} \norm{ \Cab\left((e^{i\tau R_0} f_0 - e^{i\tau R_1} f_1)\chi_\O q\right) }_{L^\infty(\O)} \\
\leq C_{\O,\phi} \norm{e^{i\tau R_0}f_0 - e^{i\tau R_1}f_1}_{L^\infty(\O)} \norm{q}_{L^{(2,1)}(\O)} \longrightarrow 0
\end{multline}
as $z_0 \to z_1$ by the uniform continuity of $e^{i\tau R} f$ in $\wbar\O\times \wbar\O$. The second term is handled quite similarly. We continue from \eqref{corollary2eqn1}
\begin{multline}
\ldots \leq C_{\O,\phi} \norm{ (e^{-i\tau R_0} - e^{-i\tau R_1}) \Cab (e^{i\tau R_1}\chi_\O qf_1) }_{L^\infty(\O)} \\
\leq C_{\O,\phi} \norm{ e^{-i\tau R_0} - e^{-i\tau R_1} }_{L^\infty(\O)} \norm{\Cab (e^{i\tau R_1} \chi_\O qf_1)}_{L^\infty(\O)} \\
\leq C_{\O,\phi} \norm{e^{-i\tau R_0} - e^{-i\tau R_1}}_{L^\infty(\O)} \norm{q}_{L^{(2,1)}(\O)} \norm{f_1}_{L^\infty(\O)} \longrightarrow 0
\end{multline}
as $z_0 \to z_1$ because $\wbar\O$ is compact and $e^{-i\tau R}$ is continuous.
\end{proof}

\begin{remark}
The equality $\left( L^{(2,\infty)}(\O), H^{1,(2,\infty)}(\O) \right)_{[s]} = H^{s,(2\infty)}(\O)$ would follow from the existence of a \emph{strong extension operator}\index{extension operator} $E$ mapping both $E:L^{(2,\infty)}(\O) \to L^{(2,\infty)}(\C)$ and $E:H^{1,(2,\infty)}(\O) \to H^{1,(2,\infty)}(\C)$.
\end{remark}

\begin{remark}
As in an earlier remark, we get the usual form of the Carleman estimates\index{Carleman estimate} by writing $r = \Ca(e^{-i\tau R} \Cab(e^{i\tau R} qf))$:
\begin{equation}
\label{carlemanEstimateLaplFractDer}
\begin{split}
&\norm{r}_{s,(2,\infty)} \leq C_\O \tau^{-1}(1+\ln \tau) \lVert e^{-i\tau(z-z_0)^2} \Delta e^{i\tau(z-z_0)^2} r\rVert_{s,(2,1)}\\
&\norm{r}_{M^s} \leq C_\O \tau^{-1/3} \lVert e^{-i\tau(z-z_0)^2} \Delta e^{i\tau(z-z_0)^2} r \rVert_{s,(2,1)}
\end{split}
\end{equation}
\end{remark}

\subsection{Bukhgeim's oscillating solutions}

We can now show that $(\Delta + q)u=0$ has Bukhgeim's oscillating solutions $u = e^{i\tau(z-z_0)^2}f$, with $f-1$ small enough. See \cite{bukhgeim} for the original article. Smallness can not be proven at the same time as existence. Instead, we fetch $f$ from $M^s$ and show that its norm is small in $H^{s,(2,\infty)}$. It is basically a boot-strapping argument.

\begin{theorem}
\index{solutions!Bukhgeim's|textbf}
\label{solEx}
Let $\O \subset \C$ be a bounded Lipschitz domain, $0< s<1$, $q \in H^{s,(2,1)}(\O)$ and\footnote{Here $C_\O$ is the constant in the second estimate of corollary \ref{corollary2}.} $\tau > \max\{1,(C_\O\norm{q}_{s,(2,1)})^3\}$. Then, for each $z_0 \in \wbar\O$, there is a unique $f_{z_0} \in M^s(\O)$ such that
\begin{equation}
\label{fixedPeqn}
f_{z_0} = 1 - \tfrac{1}{4} \Ca \left( e^{-i\tau R} \chi_\O \Cab(e^{i\tau R} \chi_\O q f_{z_0} ) \right).
\end{equation}
The map $z_0 \!\mapsto\! f_{z_0}$ is continuous $\wbar\O \to W^{1,2}(\O)\cap H^{s,(2,\infty)}(\O)\cap M^s(\O)$ and we have the norm estimates
\begin{equation}
\begin{split}
&\norm{f_{z_0} - 1}_{s,(2,\infty)} \leq C_{\O,s} \tau^{-1}(1+\ln \tau) \norm{q}_{s,(2,1)},\\
&\norm{f_{z_0}}_{1,2} \leq C_{\O,s} (1+ \norm{q}_{(2,1)}).
\end{split}
\end{equation}
\end{theorem}
\begin{proof}
For $\tau > \max\{1,(C_\O\norm{q}_{s,(2,1)})^3\}$ and $z_0 \in \wbar\O$ define
\begin{equation}
T_{z_0} : f\mapsto 1-\frac{1}{4} \Ca\left( e^{-i\tau R} \chi_\O \Cab(e^{i\tau R}\chi_\O q f) \right).
\end{equation}
We have $T_{z_0} : M^s(\O) \to M^s(\O)$ by corollary \ref{corollary2} and
\begin{multline}
\norm{ T_{z_0} f - T_{z_0} f' }_{M^s(\O)} = \tfrac{1}{4} \norm{ \Ca\left( e^{-i\tau R} \chi_\O \Cab(e^{i\tau R}\chi_\O q(f-f'))\right) }_{M^s(\O)} \\
\leq \tfrac{1}{4}C_{\O} \tau^{-1/3} \norm{q}_{H^{s,(2,1)}(\O)} \norm{f-f'}_{M^s(\O)},
\end{multline}
so $T$ is a contraction in $M^s(\O)$. Moreover, corollary \ref{corollary2} implies that the map $z_0 \mapsto T_{z_0} f$ is continuous. Banach's fixed point theorem\index{Banach's fixed point theorem} shows that there is a unique $f_{z_0} \in M^s(\O)$ satifying \eqref{fixedPeqn} and it depends continuously on $z_0$. See for example \cite{sauvigny}, VII \S1, theorem 3.

The solution $(z_0,z) \mapsto f_{z_0}(z)$ is bounded and continuous $\wbar\O\times\wbar\O \to \C$ because $M^s(\O) \hookrightarrow BC(\wbar\O)$. Hence corollary \ref{corollary2} gives us the continuity $\wbar\O \to W^{1,2}(\O) \cap H^{s,(2,\infty)}(\O) \cap M^s(\O)$. Then use the first inequality of the same corollary to get
\begin{equation}
\norm{f_{z_0}-1}_{s,(2,\infty)} \leq \tfrac{1}{4}C_\O \tau^{-1} (1+\ln \tau) \norm{q}_{s,(2,1)} \norm{f_{z_0}}_{M^s}.
\end{equation}
We have $\Cab : L^{(2,1)} \to BC(\C)$ and $\Ca \chi_\O : BC(\wbar{\O}) \hookrightarrow L^{(2,1)}(\O) \to W^{1,(2,1)}(\O) \hookrightarrow W^{1,2}(\O)$ by the boundedness of $\O$, the fact that $L^{(2,1)} \hookrightarrow L^2$ and lemmas \ref{ONeilLemma1} and \ref{cauchyIntPart}. This implies
\begin{equation}
\norm{f_{z_0}}_{1,2} \leq C_{\O}( 1+ \norm{q}_{(2,1)} \norm{f}_{M^s}).
\end{equation}
The claims follow since $\norm{T_{z_0}}_{M^s\to M^s} \leq \frac{1}{4}$, so $\norm{f_{z_0}}_{M^s} \leq \norm{1}_{M^s} + \tfrac{1}{4} \norm{f_{z_0}}_{M^s}$, which implies $\norm{f_{z_0}}_{M^s} \leq \frac{4}{3}\norm{1}_{M^s}$.
\end{proof}

\begin{corollary}
\label{solExDE}
Let $\O \subset \C$ be a bounded Lipschitz domain, $0< s<1$, $q \in H^{s,(2,1)}(\O)$ and $\tau > \max\{1,(C_\O\norm{q}_{s,(2,1)})^3\}$. Let $f_{z_0} \in M^s(\O)$ be as in theorem \ref{solEx}. Then
\begin{equation}
\Delta \left( e^{i\tau(z-z_0)^2} f_{z_0} \right) + q e^{i\tau(z-z_0)^2} f_{z_0} = 0
\end{equation}
in $\mathscr{D}'(\O)$.
\end{corollary}
\begin{proof}
Note that $\chi_\O \equiv 1$ in $\mathscr{D}'(\O)$ and keep in mind that
\begin{equation}
f_{z_0} = 1 - \tfrac{1}{4} \Ca( e^{-i\tau R} \chi_\O \Cab(e^{i\tau R}\chi_\O q f_{z_0})).
\end{equation}
We get
\begin{multline}
\db \left( e^{i\tau(z-z_0)^2} f_{z_0} \right) = e^{i\tau(z-z_0)^2} \db f_{z_0} = -\tfrac{1}{4} e^{i\tau(z-z_0)^2} e^{-i\tau R} \chi_\O \Cab( e^{i\tau R} \chi_\O q f_{z_0}) \\
= -\tfrac{1}{4} e^{-i\tau(\wbar{z}-\wbar{z_0})^2} \Cab( e^{i\tau R} \chi_\O q f_{z_0} ).
\end{multline}
because $\db \Ca = \d \Cab = \operatorname{Id}$ in $\mathscr{E}'(\C)$ by lemma \ref{cauchyIntPart}. Now
\begin{multline}
\Delta \left( e^{i\tau(z-z_0)^2} f_{z_0} \right) = 4\d\db \left(e^{i\tau(z-z_0)^2} f_{z_0} \right) = -e^{-i\tau (\wbar{z}-\wbar{z_0})^2} e^{i\tau R} \chi_\O q f_{z_0} \\
= - q e^{i\tau(z-z_0)^2} f.
\end{multline}
\end{proof}

\vfill
\section{The problem setting}
\subsection{Hadamard's criteria and the DN operator}
We will define what the well-posedness of the direct problem means, what is the Dirichlet-Neumann operator and then make an extension allowing us to solve the inverse problem for potentials that do not necessarily give well posed direct problems.

The Dirichlet-Neumann operator is inherently related to the trace mapping\index{trace mapping}\index[notation]{Tr@$\Tr$; trace mapping} $\Tr$ from a function space on a domain to its boundary. It is well known that $\Tr : H^{s}(\O) \to H^{s-\frac{1}{2}}(\d\O)$ for suitable values of the smoothness index $s$ when $\O$ is a bounded Lipschitz domain. We want to avoid the trouble of defining $H^s(\d\O)$ when $\O$ is not smooth, so instead we take the ``boundary values'' as equivalence classes in $W^{1,2}(\O)$. More specifically\footnote{To be more exact: given $u\in W^{1,2}(\O)$ there is a unique equivalence class $[u] \in W/W_0$ such that $u \in [u]$. Moreover $u + \phi \in [u]$ for all $\phi \in W^{1,2}_0(\O)$.}, they are in $W^{1,2}(\O) / W^{1,2}_0(\O)$\index[notation]{WWzero@$W/W_0, W^{1,2}/W^{1,2}_0$; quotient space} with norm
\begin{equation}
\norm{u}_{W/W_0} = \inf_{\phi \in W^{1,2}_0(\O)} \norm{u + \phi}_{W^{1,2}(\O)}.
\end{equation}
Using this space instead of the spaces $H^s(\d\O)$ is allowed. See for example Gagliardo \cite{gagliardo} and Ding \cite{ding}. They imply
\begin{equation}
\begin{tikzcd}
H^{1/2}(\d\O) \arrow{dr}[swap]{\mathscr{E}} \arrow{rr}{\operatorname{Id}} &  & H^{1/2}(\d\O) \\
	& W^{1,2}(\O) / W^{1,2}_0(\O) \arrow{ur}[swap]{\Tr}&
\end{tikzcd}
\end{equation}
where $\Tr, \mathscr{E}$ are bounded linear mappings. Equation (2.9) of \cite{gagliardo} gives linearity for $\mathscr{E}$. Note that here and from now on we write $W/W_0 = W^{1,2}(\O) / W^{1,2}_0(\O)$.

\medskip
The Dirichlet-Neumann\index{Dirichlet-Neumann map} map is a way to model boundary measurements. It makes sense to require the differential equation to behave nicely enough to be a physical model. One such class of problems is those satisfying Hadamard's three criteria\index{Hadamard's criteria} \cite{hadamard}: 1) the problem must have a solution, 2) the solution must be unique, and 3) the solution should depend continuously on the data. In such case, the problem is said to be well-posed. The reasons for these criteria is also a mathematical one. We can prove mapping properties for the Dirichlet-Neumann operator if the potential gives a well-posed problem.

\begin{definition}
\label{DPWP}\index{well-posed problem}
Let $\O \subset \C$ be open and $q$ measurable. Then \emph{the direct problem is well-posed} if there is $C <\infty$ such that for any $u \in W^{1,2}(\O) / W^{1,2}_0(\O)$ we have
\begin{enumerate}
\item there is $U\in W^{1,2}(\O)$ such that $\Delta U + q U = 0$, $U - u \in W^{1,2}_0(\O)$,
\item this $U$ is unique
\item $u \mapsto U$ is linear and bounded $\norm{U}_{W^{1,2}(\O)} \leq C \norm{u}_{W/W_0}$
\end{enumerate}
\end{definition}

\begin{definition}
\index{Dirichlet-Neumann map|textbf}
\index[notation]{Lambdaq@$\Lambda_q$; Dirichlet-Neumann map}
\label{DNdef}
Let $\O \subset \C$ be open and $q$ measurable such that the direct problem is well-posed. Then we define the \emph{Dirichlet-Neumann operator $\Lambda_q$} as follows. For $u \in W/W_0$ we define $\Lambda_q u$ by
\begin{equation}
\label{DNmap}
(\Lambda_q u , v) = \int_\O (-\nabla U \cdot \nabla V + q U V )dm, \quad v \in W/W_0,
\end{equation}
for any $U,V \in W^{1,2}(\O)$ such that $U-u,V-v \in W^{1,2}_0(\O)$ and $\Delta U + q U = 0$.
\end{definition}

\begin{lemma}
\label{DNok}
Let $\O \subset \C$ be a bounded Lipschitz domain and $q \in L^a(\O)$ with $a>1$. Then the Dirichlet-Neumann map is a well defined bounded linear operator mapping $W/W_0$ to its dual. Moreover it satisfies
\begin{equation}
(\Lambda_q u, v) = (\Lambda_q v, u), \quad u,v \in W/W_0.
\end{equation}
\end{lemma}
\begin{proof}
We start by showing that the choice of $U,V$ in definition \ref{DNdef} doesn't matter. First of all, $U$ exists and is unique on the right-hand side of \eqref{DNmap} by the well-posedness of the direct problem \ref{DPWP}. Assume that $V,V' \in W^{1,2}(\O)$ satisfy $V-v,V'-v \in W^{1,2}_0(\O)$. Then
\begin{multline}
\int_\O -\nabla U \cdot \nabla V + q UV dm - \int_\O -\nabla U \cdot \nabla V' + q UV' dm \\
= \int_\O - \nabla U \cdot \nabla (V-V') + q U (V-V') dm \\
= \lim_{n\to\infty} \int_\O -\nabla U \cdot \nabla \phi_n + q U \phi_n dm = \lim_{n\to\infty} \langle \Delta U + q U, \phi \rangle = 0
\end{multline}
where $\phi_n \in C^\infty_0(\O)$ such that $\norm{ \phi_n - (V-V')}_{1,2} \to 0$. This sequence exists because $V-V' = V- v -(V'-v) \in W^{1,2}_0(\O)$. Hence \eqref{DNmap} is well defined.

The mapping properties follow next. We have $W^{1,2}(\O) \hookrightarrow L^{\frac{2a}{a-1}}(\O)$ by Sobolev embedding, e.g. theorem 4.12 in \cite{adams}. Now
\begin{equation}
1 = \frac{1}{a} + \frac{1}{\frac{2a}{a-1}} + \frac{1}{\frac{2a}{a-1}},
\end{equation}
so
\begin{multline}
\abs{(\Lambda_q u,v)} \leq \norm{\nabla U}_2 \norm{\nabla V}_2 + \norm{q}_a \norm{U}_{\frac{2a}{a-1}} \norm{V}_{\frac{2a}{a-1}} \\
\leq C_{a,\O} (1+ \norm{q}_a) \norm{U}_{1,2} \norm{V}_{1,2}.
\end{multline}
Then, take the infimum over $V \in W^{1,2}(\O)$, $V-v \in W^{1,2}_0(\O)$ and use condition number three of the well-posedness \ref{DPWP} on $U$ to get
\begin{equation}
\abs{(\Lambda_q u,v)} \leq C_{a,\O,q} \norm{u}_{W/W_0} \norm{v}_{W/W_0}.
\end{equation}

To prove the last formula let $u,v \in W/W_0$ and take $U,V \in W^{1,2}(\O)$ such that $U-u,V-v \in W^{1,2}_0(\O)$, $\Delta U + q U = \Delta V + q V = 0$. These exist by the well-posedness \ref{DPWP}. Hence
\begin{equation}
(\Lambda_q u, v) = \int_\O -\nabla U \cdot \nabla V + qUV dm = \int_\O -\nabla V \cdot \nabla U + q VU dm = (\Lambda_q v,u).
\end{equation}
\end{proof}

\begin{remark}
\index[notation]{normLambdaq@$\norm{\Lambda_{q_1}-\Lambda_{q_2}}$}
$\Lambda_q$ is a well defined linear operator mapping $W/W_0$ to its dual. Hence the linear combinations of such operators are also well defined. In particular
\begin{equation}
\norm{\Lambda_{q_1} - \Lambda_{q_2}} = \sup \big\{ \abs{((\Lambda_{q_1} - \Lambda_{q_2})u,v)} \,\big|\, u,v \in W/W_0, \norm{u} = \norm{v} = 1 \big\}.
\end{equation}
\end{remark}

\subsection{Cauchy data}
We would like to extend the notion of the Dirichlet-Neumann map to cases where the direct problem is not well-posed. One such way is to make use of the \emph{Cauchy data}
\begin{equation}
\index{boundary data|textbf}
\index[notation]{Cq@$C_q$; boundary data}
C_q = \{ (\Tr u, \d_\nu u) \mid u \in W^{2,2}(\O), \Delta u + q u = 0\},
\end{equation}
but this would further require three more definitions and related properties: the trace-operator $\Tr$, the normal derivative $\d_\nu$, and a way to measure the distance of two Cauchy data $C_{q_1}$ and $C_{q_2}$.

We are not too interested in the direct problem, so instead we will just extend the notion of $\Lambda_{q_1} - \Lambda_{q_2}$ to situations where the Dirichlet problem $\Delta U + q U = 0$, $U-u \in W^{1,2}_0(\O)$ does not have a unique solution. It is based on the well known Alessandrini's identity $\int U_1(q_1-q_2)U_2 dm = ((\Lambda_{q_1} - \Lambda_{q_2}) u_1, u_2)$ for solutions $\Delta U_j + q_j U_j = 0$, $U_j - u_j \in W^{1,2}_0(\O)$. We shall prove it here too. But first, the generalization.

\begin{definition}
\index{boundary data}
\index[notation]{dC1C2@$d(C_{q_1},C_{q_2})$ distance between boundary data}
\label{dC1C2def}
Let $\O \subset \C$ be open and $q_1,q_2$ measurable. Then the \emph{distance between the boundary data}, $d(C_{q_1},C_{q_2})$, is
\begin{multline}
d(C_{q_1},C_{q_2}) = \sup \Big\{ \abs{\int_\O U(q_1-q_2)V dm} \,\Big|\, U,V \in W^{1,2}(\O), \\
\norm{U} = \norm{V} = 1, \Delta U + q_1 U = \Delta V + q_2 V = 0 \Big\}
\end{multline}
\end{definition}
\begin{remark}
It is possible to acquire this data purely using knowledge from the boundary, at least part of it when $U,V \in W^{2,2}$. This \emph{Alessandrini's identity} follows from Green's formula\index{integration by parts}: \index{Alessandrini's identity|textbf}
\begin{equation}
\int_\O U(q_1-q_2)V dm = \int_\O - V\Delta U + U\Delta V dm = \int_{\d\O} \Tr U \d_\nu V - \Tr V \d_\nu U d\sigma
\end{equation}
If the reader is concerned about the lack of smoothness in the solutions of the definition of $d(C_{q_1},C_{q_2})$, then note the following: the only solution $U,V$ that matter for solving the inverse problem are Bukhgeim's oscillating ones. Namely those that were constructed in corollary \ref{solExDE}. It is not hard to see that they are in $W^{2,2}$ and that their $W^{2,2}$ norms also grow exponentially.
\end{remark}

\begin{lemma}
Let $\O\subset\C$ be a bounded Lipschitz domain and $q_1,q_2 \in L^a(\O)$, $a>1$. Then $d(C_{q_1},C_{q_2}) \leq C_{a,\O} \norm{q_1 - q_2}_a$.
\end{lemma}
\begin{proof}
Use the Sobolev embedding $W^{1,2}(\O) \hookrightarrow L^{\frac{2a}{a-1}}(\O)$, for example in theorem 4.12, \cite{adams}. Then
\begin{multline}
\abs{\int_\O U(q_1-q_2)V dm} \leq \norm{U}_{\frac{2a}{a-1}} \norm{q_1 - q_2}_a \norm{V}_{\frac{2a}{a-1}} \\
\leq C_{a,\O} \norm{q_1-q_2}_a \norm{U}_{1,2} \norm{V}_{1,2}.
\end{multline}
\end{proof}

\begin{theorem}
\label{DNint}
Let $\O \subset \C$ be bounded and Lipschitz, and $q_1,q_2 \in L^a(\O)$, $a>1$, be such that the direct problem is well-posed. Then
\begin{equation}
d(C_{q_1},C_{q_2}) \leq \norm{ \Lambda_{q_1}-\Lambda_{q_2}}.
\end{equation}
\end{theorem}
\begin{proof}
Let $u, v \in W/W_0$. Take $U,V \in W^{1,2}(\O)$ such that $\Delta U + q_1 U = 0$, $U-u\in W^{1,2}_0(\O)$ and similarly with $V,v$. These exist by the well-posedness of the direct problems \ref{DPWP}. Because the Dirichlet-Neumann map in \eqref{DNmap} is well-defined, we have
\begin{multline}\index{Alessandrini's identity}
\left((\Lambda_{q_1} - \Lambda_{q_2}) u,v\right) = (\Lambda_{q_1} u,v) - (\Lambda_{q_2} u,v) =(\Lambda_{q_1} u,v) - (\Lambda_{q_2} v,u) \\
= \int_\O -\nabla U \cdot \nabla V + q_1 U V + \nabla V \cdot \nabla U - q_2 VU dm \\
= \int_\O U(q_1-q_2)V dm
\end{multline}
by lemma \ref{DNok}.

Let $\mathscr{P}$ be the canonical projection map $W^{1,2}(\O) \to W/W_0$. It has operator norm at most one by definition. Now choose $u = \mathscr{P}U$ and $v = \mathscr{P}V$ to get
\begin{multline}
\abs{ \int_\O U(q_1-q_2)V dm} = \abs{ \left( (\Lambda_{q_1} - \Lambda_{q_2}) u,v\right) } \\
\leq \norm{\Lambda_{q_1} - \Lambda_{q_2}} \norm{u}_{W/W_0} \norm{v}_{W/W_0} \leq \norm{\Lambda_{q_1} - \Lambda_{q_2}} \norm{U}_{1,2} \norm{V}_{1,2}.
\end{multline}
Taking the supremum over $U,V \in W^{1,2}(\O)$, $\Delta U + q_1 U = \Delta V + q_2 V = 0$, $\norm{U}_{1,2} = \norm{V}_{1,2} = 1$ gives the result.
\end{proof}

\subsection{Uniqueness and stability for the inverse problem}
A technical lemma first.
\begin{lemma}
\label{uNorm}
Let $\O \subset \C$ be bounded and Lipschitz, $f_1,f_2 \in BC(\wbar\O, W^{1,2}(\O))$ and $\tau > 0$. Let $u_1 (z) = e^{i\tau(z-z_0)^2}f_1(z)$ and $u_2(z) = e^{i\tau(\wbar{z}-\wbar{z_0})^2}f_2(z)$ for $z_0 \in \wbar\O$. Then
\begin{equation}
\norm{u_j}_{BC(W^{1,2})} \leq e^{C_\O \tau} \norm{f_j}_{BC(W^{1,2})}
\end{equation}
for some positive real $C_\O$ depending only on $\O$.
\end{lemma}
\begin{proof}
This follows from the elementary facts $\abs{z-z_0}, \abs{\wbar{z}-\wbar{z_0}} \leq \operatorname{diam}(\O)$, $\tau \leq e^\tau$ and $\abs{\d e^{i\tau(z-z_0)^2}}, \abs{\db e^{i\tau(\wbar{z}-\wbar{z_0})^2}} \leq 2\tau\operatorname{diam}(\O) e^{\operatorname{diam}(\O)^2 \tau}$.
\end{proof}

We are now ready to solve the inverse problem. The big goal of inverse problems for partial differential equations is to deduce the values of the coefficients inside the domain using only data from the boundary. We do not get that far, that is, we do not have a reconstruction formula. Instead, we show uniqueness and logarithmic stability. It means that if there are two potentials $q_1$ and $q_2$ which are distance $\epsilon$ apart, then their corresponding boundary data must be roughly at least $e^{-\epsilon^{-1}}$ apart. This is not much, but it is not possible to get a better modulus of continuity. Mandache showed that the inverse problem is inherently \emph{ill-posed}\index{counterexamples} \cite{mandache}.

We remind the general flow of the proof. For a more detailed reminder see section \ref{sketchSection}. We start by using stationary phase
\begin{equation}
\int_\O \frac{2\tau}{\pi} e^{i\tau R} (q_1-q_2) dm \longrightarrow q_1 - q_2
\end{equation}
as $\tau \to \infty$. The integral can be approximated by a term $\int_\O u_1 (q_1-q_2)u_2$, like in the definition of $d(C_{q_1},C_{q_2})$, because of the special form of our solutions. All in all
\begin{multline}
q_1 - q_2 = \Big(q_1 - q_2 - \int_\O \frac{2\tau}{\pi} e^{i\tau R} (q_1-q_2) dm \Big) + \int_\O \frac{2\tau}{\pi} e^{i\tau R} (q_1-q_2) dm \\
= \Big(q_1 - q_2 - \int_\O \frac{2\tau}{\pi} e^{i\tau R} (q_1-q_2) dm \Big) + \frac{2\tau}{\pi} \int_\O u^{(1)}(q_1 - q_2) u^{(2)} dm \\
+ \frac{2\tau}{\pi} \int_\O e^{i\tau R}(q_1 - q_2) (1 - f_1 f_2) dm
\end{multline}
where $u^{(1)} = e^{i\tau(z-z_0)^2}f_1$ and $u^{(2)} = e^{i\tau (\wbar{z}-\wbar{z_0})^2} f_2$ are the solutions given by theorem \ref{solEx} and corollary \ref{solExDE}. The first term will be estimated by lemma \ref{stationaryPhase}, the second one by theorem \ref{DNint}, and the last one by theorem \ref{errorTermIntegral}.

\begin{theorem}\index{inverse problem!solution}
Let $\O \subset \C$ be a bounded Lipschitz domain, $M>0$ and $0<s<\frac{1}{2}$. Then there is a positive real number $C$ such that if $q_j \in H^{s,(2,1)}(\O)$ and $\norm{q_j}_{s,(2,1)} \leq M$ then
\begin{equation}
\norm{q_1-q_2}_{L^{(2,\infty)}(\O)} \leq C \left( \ln d(C_{q_1},C_{q_2})^{-1} \right)^{-s/4}.
\end{equation}
In particular, we have uniqueness and stability for potentials in $H^{s,p}(\O)$ with $s>0$, $p>2$.
\end{theorem}

\begin{proof}
Denote $Q = q_1 - q_2$\index[notation]{Q@$Q$; $q_1-q_2$} and $R = (z-z_0)^2 + (\wbar{z}-\wbar{z_0})^2$ for $z,z_0 \in \C$. Remember that $z$ is the variable being operated on first. Let $\tau > 1$. By the triangle inequality
\begin{multline}
\norm{q_1 - q_2}_{L^{(2,\infty)}(\O)} \\
\leq \norm{ Q - \int_\O \tfrac{2\tau}{\pi} e^{i\tau R}Q dm(z)}_{L^{(2,\infty)}(\O,z_0)} \!\!\!\!\!\!+ \norm{\tfrac{2\tau}{\pi} \int_\O e^{i\tau R}Q dm(z) }_{L^{(2,\infty)}(\O,z_0)},
\end{multline}
where the norms are taken with respect to $z_0$. We will use stationary phase\index{stationary phase} next. Extend $Q$ by zero outside of $\O$ to get $Q_0$. It is in $H^{s,2}(\C)$ by 3.5 in \cite{clopFaracoRuiz} with norm estimate $\norm{Q_0}_{H^{s,2}(\C)} \leq C_{\O,s} \norm{Q}_{H^{s,2}(\O)}$. We get
\begin{multline}
\label{statPhaseIneq}
\norm{Q - \frac{2\tau}{\pi} \int_\O e^{i\tau R} Q dm(z)}_{L^2(\O,z_0)} \!\!\!\!\!\! = \norm{ Q_0 - \frac{2\tau}{\pi} \int_\C e^{i\tau R} Q_0 dm(z)}_{L^2(\C,z_0)}\\
\leq C_s \tau^{-s/2} \norm{Q_0}_{H^{s,2}(\C)} \leq C_{\O,s} \tau^{-s/2} \norm{Q}_{H^{s,2}(\O)} \\
\leq C_{\O,s} \tau^{-s/2} \norm{Q}_{s,(2,1)} \leq C_{\O,M,s} \tau^{-s/4}
\end{multline}
by theorem \ref{stationaryPhase}. We estimated $\tau^{-s/2} \leq \tau^{-s/4}$ for later purposes. We also used the embedding $L^2 \hookrightarrow L^{(2,\infty)}$.

Consider the second term next. Take $\tau_0 = \max\{1, \left( C_\O M \right)^3\}$\index[notation]{tauzero@$\tau_0$; minimum value of $\tau$} as in theorem \ref{solEx} and let $\tau = \frac{1}{2 B_\O} \ln d(C_{q_1},C_{q_2})^{-1}$, where $B_\O = 1 + 2 C_\O$ with the $C_\O$ of lemma \ref{uNorm}. We have $\tau > \tau_0$ if
\begin{equation}
d(C_{q_1},C_{q_2}) < e^{-2B_\O \tau_0}.
\end{equation}
Assume that $d(C_{q_1},C_{q_2}) < \min \{e^{-1}, e^{-2B_\O\tau_0}\}$ for now. The other case will be taken care of in the end of the proof. Note that $\tau_0$ grows with $M$.

Theorem \ref{solEx} (the sign of $i$ does not matter) gives the existence of $f^{(1)}, f^{(2)} \in BC(\wbar\O, M^s(\O) \cap W^{1,2}(\O))$ such that we have
\begin{equation}\index{solutions!Bukhgeim's}
\begin{cases}
f^{(1)} = 1 - \tfrac{1}{4} \Ca\big( e^{-i\tau R} \chi_\O \Cab (e^{i\tau R} \chi_\O q_1 f^{(1)})\big),\\
f^{(2)} = 1 - \tfrac{1}{4} \Cab\big( e^{-i\tau R} \chi_\O \Ca (e^{i\tau R} \chi_\O q_2 f^{(2)})\big),
\end{cases}
\end{equation}
for all $z_0 \in \wbar\O$, and
\begin{equation}
\begin{cases}
\sup_{z_0} \norm{f^{(j)} - 1}_{H^{s,(2,\infty)}(\O)} \leq C_{\O,M,s} \tau^{-1}(1+\ln \tau),\\
\sup_{z_0} \norm{f^{(j)}}_{W^{1,2}(\O)} \leq C_{\O,M,s} < \infty.
\end{cases}
\end{equation}
Denote
\begin{equation}
\begin{cases}
u^{(1)}_{z_0}(z) = e^{i\tau(z-z_0)^2}f^{(1)}(z_0,z),\\
u^{(2)}_{z_0}(z) = e^{i\tau(\wbar{z}-\wbar{z_0})^2}f^{(2)}(z_0,z).
\end{cases}
\end{equation}
Now they satisfy $u^{(j)}_{z_0} \in BC(\wbar{\O}, M^s(\O) \cap W^{1,2}(\O))$ and $\Delta u^{(j)}_{z_0} + q_j u^{(j)}_{z_0} = 0$ for all $z_0$ by corollary \ref{solExDE}. Moreover, we have
\begin{equation}
\label{solNorms}
\sup_{z_0} \norm{u^{(j)}_{z_0}}_{W^{1,2}(\O)} \leq C_{\O,M,s} e^{C_\O \tau}
\end{equation}
by lemma \ref{uNorm}. Now, by the triangle inequality,
\begin{multline}
\label{lastSplitEq}
\norm{\frac{2\tau}{\pi} \int_\O e^{i\tau R} Q dm(z)}_{L^{(2,\infty)}(\O,z_0)} \leq \norm{\frac{2\tau}{\pi} \int_\O u^{(1)}_{z_0}(q_1 - q_2)u^{(2)}_{z_0} dm(z)}_{L^{(2,\infty)}(\O,z_0)} \\
+ \norm{\frac{2\tau}{\pi} \int_\O e^{i\tau R}Q(f^{(1)} f^{(2)} - 1) dm(z)}_{L^{(2,\infty)}(\O,z_0)}.
\end{multline}
Use definition \ref{dC1C2def} and the equation \eqref{solNorms} to get
\begin{multline}
\label{boundaryTermIneq}
\norm{ \frac{2\tau}{\pi} \int_\O u^{(1)}_{z_0}(q_1 - q_2)u^{(2)}_{z_0} dm(z) }_{(2,\infty)} \\
\leq 2\pi^{-1}\,\tau\, d(C_{q_1},C_{q_2}) \sup_{z_0} \norm{u^{(1)}_{z_0}}_{W^{1,2}(\O)} \sup_{z_0} \norm{u^{(2)}_{z_0}}_{W^{1,2}(\O)} \\
\leq C_{\O,M,s} d(C_{q_1},C_{q_2}) e^{B_\O \tau}.
\end{multline}
For the second term, we need to show that $(f^{(1)} f^{(2)} - 1) \in H^{s,(2,\infty)}(\O)$ and that there is no problems with measurability for the $L^{(2,\infty)}(\O)$ norm. Notice that $f^{(1)} f^{(2)} - 1 = (f^{(1)} - 1)(f^{(2)} - 1) + f^{(1)} - 1 + f^{(2)} - 1$.

\smallskip
We will show that $H^{s,(2,\infty)}(\O)$ is stable with respect to multiplication by $f^{(2)}-1$, as the latter can be extended to a function $G \in BC(\C)$ with $\nabla G \in L^4(\C)$. Let $\phi \in C^\infty_0(\C)$ be such that $\phi \equiv 1$ on $\wbar\O$ and it vanishes outside $B_r=B(0,r) \supset \wbar\O$. Then set
\begin{equation}
G = -\tfrac{1}{4}\phi\Cab\big( e^{-i\tau R} \chi_\O \Ca (e^{i\tau R} \chi_\O q_2 f^{(2)})\big).
\end{equation}
Now $G_{|\O} = f^{(2)}-1$, and according to corollary \ref{corollary1} we have $G \in BC(\C)$. Next
\begin{multline}
\norm{\nabla G}_{L^4(\C)} \leq \norm{\nabla \phi}_{L^4(\C)} \tfrac{1}{4} \norm{ \Cab\big( e^{-i\tau R} \chi_\O \Ca (e^{i\tau R} \chi_\O q_2 f^{(2)})\big)}_{L^\infty(\C)} \\
+C \norm{\phi}_{L^\infty(\C)} \norm{ \Ca(e^{i\tau R} \chi_\O q_2 f^{(2)} ) }_{L^4(\O)} \\
\leq C_{\phi,\O} \norm{ \Ca(e^{i\tau R} \chi_\O q_2 f^{(2)} )}_{L^\infty(\O)} \leq C_\O \norm{q_2 f^{(2)}}_{L^{(2,1)}(\O)} \\
\leq C_\O \norm{q_2}_{H^{s,(2,1)}(\O)} \norm{f^{(2)}}_{M^s(\O)} \leq C_{\O,M,s}
\end{multline}
because $\nabla \Ca \chi_\O : L^4(\C) \to L^4(\C)$ by lemma \ref{cauchyIntPart}.

It is easy to see that $\norm{GF}_{(2,\infty)} \leq \norm{G}_\infty \norm{F}_{(2,\infty)} \leq C_{\O,M,s} \norm{F}_{(2,\infty)}$ for $F \in L^{(2,\infty)}(\C)$. Take $F\in H^{1,(2,\infty)}(\C)$. Now
\begin{multline}
\norm{GF}_{1,(2,\infty)} \leq C\left( \norm{FG}_{(2,\infty)} + \norm{ G\nabla F}_{(2,\infty)} + \norm{F \nabla G}_{(2,\infty)} \right) \\
\leq C \left( \norm{G}_\infty \norm{F}_{(2,\infty)} + \norm{G}_\infty \norm{\nabla F}_{(2,\infty)} + \norm{\nabla G}_{L^4(B_r)} \norm{F}_{L^4(B_r)}\right).
\end{multline}
Note that $F_{|B_r} \in W^{1,(2,\infty)}(B_r) \hookrightarrow W^{1,\frac{4}{3}}(B_r) \hookrightarrow L^4(B_r)$ by Sobolev embedding (e.g. 4.12 in \cite{adams}). Moreover $\norm{F_{|B_r}}_{H^{1,(2,\infty)}(B_r)} \leq C_r \norm{F}_{H^{1,(2,\infty)}(\C)}$. Hence
\begin{equation}
\norm{GF}_{1,(2,\infty)} \leq C_{r,\O} (\norm{G}_\infty + \norm{\nabla G}_4 ) \norm{F}_{1,(2,\infty)}.
\end{equation}
Interpolation implies that multiplying by $G$ is stable in $H^{s,(2,\infty)}(\C)$, with norm increasing by at most $C_{\O,M,s}$ since $r$ can be chosen based on $\O$ only.

Now take an extension $F$ of $f^{(1)}-1$ such that $\norm{F} _\C \leq 2 \norm{f^{(1)}-1}_\O$. This is possible by the definition of $H^{s,(2,\infty)}(\O)$. Then
\begin{multline}
\norm{ (f^{(1)} - 1)(f^{(2)} - 1) }_{H^{s,(2,\infty)}(\O)} \leq \norm{ F G}_{H^{s,(2,\infty)}(\C)}\\
\leq C_{\O,M,s} \norm{F}_{H^{s,(2,\infty)}(\C)} \leq C_{\O,M,s} \norm{f^{(1)}-1}_{H^{s,(2,\infty)}(\O)}.
\end{multline}
This shows that $f^{(1)} f^{(2)} - 1 \in H^{s,(2,\infty)}(\O)$. It has the norm bound
\begin{equation}
\norm{f^{(1)} f^{(2)} - 1}_{s,(2,\infty)} \leq C_{\O,M,s} \tau^{-1} (1+\ln \tau)
\end{equation}
by the previous deductions and theorem \ref{solEx}. Measurability with respect to $z_0$ is no problem since $f^{(1)},f^{(2)} \in BC(W^{1,2})$.

\smallskip
Next, use theorem \ref{errorTermIntegral} to continue from \eqref{lastSplitEq}. We have
\begin{multline}\index{error term estimate}
\label{errorTermIneq}
\norm{\frac{2\tau}{\pi} \int_\O e^{i\tau R}Q(f^{(1)} f^{(2)} - 1) dm(z)}_{L^{(2,\infty)}(\O,z_0)} \\
\leq C_\O \norm{\frac{2\tau}{\pi} \int_\O e^{i\tau R}Q(f^{(1)} f^{(2)} - 1) dm(z)}_{L^\infty(\O,z_0)} \\
\leq C_{\O,s} \tau^{1-s/3} \norm{Q}_{s,(2,1)} \sup_{z_0\in\O} \norm{f^{(1)} f^{(2)} - 1}_{s,(2,\infty)} \\
\leq C_{\O,M,s} \tau^{-s/3}(1+\ln \tau) \leq C_{\O,M,s} \tau^{-s/4},
\end{multline}
since $\ln \tau \leq C_s \tau^{s/3 - s/4}$.

\medskip
We can combine all the terms now, namely those from equations \eqref{statPhaseIneq}, \eqref{boundaryTermIneq} and \eqref{errorTermIneq}. Remember the choice of $\tau = \frac{1}{2B_\O} \ln d(C_{q_1},C_{q_2})^{-1}$ and that $d(C_{q_1},C_{q_2}) < e^{-1}$. Note that $\ln a \leq \frac{1}{b} a^b$ for $a,b > 0$, so $x^{1/2} \leq C_s (\ln \frac{1}{x})^{-s/4}$ when $0< x < e^{-1}$. Now
\begin{multline}
\norm{q_1 - q_2}_{(2,\infty)} \leq C_{\O,M,s} \left( \tau^{-s/4} + d(C_{q_1},C_{q_2}) e^{B_\O \tau} \right) \\
= C_{\O,M,s} \left( (\ln d(C_{q_1},C_{q_2})^{-1})^{-s/4} + d(C_{q_1},C_{q_2})^{1/2} \right) \\
\leq C_{\O,M,s} \left( \ln d(C_{q_1},C_{q_2})^{-1} \right)^{-s/4}.
\end{multline}

What if $d(C_{q_1},C_{q_2}) \geq \min\{e^{-1}, e^{-2B_\O\tau_0}\}$? Then we would get directly $\left( \ln d(C_{q_1},C_{q_2})^{-1} \right)^{-s/4} \geq \left( \ln \max\{e,e^{2B_\O\tau_0}\}\right)^{-s/4}$, so
\begin{multline}
\norm{q_1 - q_2}\leq 2M \leq \frac{2M}{\left( \ln \max \{e,e^{2B_\O\tau_0}\}\right)^{-s/4}} \left( \ln d(C_{q_1},C_{q_2})^{-1} \right)^{-s/4} \\
= C_{\O,M,s} \left( \ln d(C_{q_1},C_{q_2})^{-1} \right)^{-s/4}
\end{multline}
The boundedness of $\O$ implies that $L^p(\O) \hookrightarrow L^{(2,1)}(\O)$. Hence the claim is true for potentials in $H^{s,p}(\O)$.
\end{proof}

\begin{remark}
If $q_1$ and $q_2$ would give well-posed direct problems as in definition \ref{DPWP}, then theorem \ref{DNint} shows that
\begin{equation}
\norm{q_1 - q_2} \leq  C_{\O,M,s} \left( \ln d(C_{q_1},C_{q_2})^{-1} \right)^{-s/4} \leq C_{\O,M,s} \left( \ln \norm{\Lambda_{q_1}-\Lambda_{q_2}}^{-1} \right)^{-s/4}.
\end{equation}
\end{remark}

\section{Future work}
\subsection{Function space properties of $H^{s,(p,q)}(\O)$ and $W^{s,(p,q)}(\O)$}
In section \ref{fSpacesSection} we showed that $H^{k,(p,q)}(\R^n) = W^{k,(p,q)}(\R^n)$ and commented on the embedding $H^{k,(p,q)}(\O) \hookrightarrow W^{k,(p,q)}(\O)$. There is still these two very likely equalities left to prove:
\begin{enumerate}
\item $H^{k,(p,q)}(\O) = W^{k,(p,q)}(\O)$
\item $(L^{(p,q)}(\O), H^{1,(p,q)}(\O))_{[s]} = (L^{(p,q)}(\O), W^{1,(p,q)}(\O))_{[s]} = H^{s,(p,q)}(\O)$
\end{enumerate}
These would follow directly from the existence of a strong extension operator\index{extension operator} mapping both $E:L^{(p,q)}(\O) \to L^{(p,q)}(\R^n)$ and $E:W^{k,(p,q)}(\O) \to W^{k,(p,q)}(\R^n)$. Assume that such a map exists. Consider the restriction $Rf = f_{|\O}$. Now $R \circ E = \operatorname{Id}$ and
\begin{equation}
\begin{tikzcd}
W^{k,(p,q)}(\O) \arrow{dr}[swap]{E} \arrow{rr}{\operatorname{Id}} &  & W^{k,(p,q)}(\O) \\
	& \begin{array}{c}W^{k,(p,q)}(\R^n) \\ = \\ H^{k,(p,q)}(\R^n) \end{array} \arrow{ur}[swap]{R} \arrow{dr}{R} & \\
	& & H^{k,(p,q)}(\O)
\end{tikzcd}
\end{equation}
Hence $\operatorname{Id} = R \circ E : W^{k,(p,q)}(\O) \hookrightarrow H^{k,(p,q)}(\O)$, so they are the same space. Moreover, since the operator $E$ is a strong extension operator, the commutative diagram is preserved in interpolation. Hence $(L^{(p,q)}(\O), W^{1,(p,q)}(\O))_{[s]}$ is a retract of $H^{s,(p,q)}(\R^n)$, and so the first space is a subspace of $H^{s,(p,q)}(\O)$. The other direction follows by the definition of $H^{s,(p,q)}(\O)$, because it is the smallest space $X$ for which $R:H^{s,(p,q)}(\R^n) \to X$.

After that these are some other things to consider: Sobolev embedding theorems, trace theorems, and any other subject in most standard books on Sobolev spaces.

\subsection{Doing it in $W^{s,p}$}
\label{DoingInWsp}
An earlier version of this manuscript \cite{Blasten2011thesisManuscipt} focused solely on potentials in $W^{s,p}(\O)$ with $p>2$, $s>0$. If we compare that text to this one, there is a trade-off. In the old one we had simpler function spaces, but much more parameters in the estimates. This is because of the boundary integral operator
\begin{equation}
T: f \mapsto \frac{1}{2\pi} \int_{\d\O} \frac{f(z')}{z-z'} d\sigma(z').
\end{equation}
We couldn't quite prove that it would map $L^p(\d\O) \to L^p(\O)$ for $2 < p < \infty$. We now know that it maps $L^1(\d\O) \to L^{(2,\infty)}(\O) \hookrightarrow L^1(\O)$. Using \emph{Bloch} spaces and their relation to $BMO(\O)$, as in \cite{muramoto} or \cite{pavlovic}, we get
\begin{equation}
T: L^\infty(\d\O) \to BMO(\O).
\end{equation}
Hence the result would follow by an interpolation result like 
\begin{equation}
(L^1(\O), BMO(\O))_{[1/p]} = L^p(\O) \quad \text{or} \quad (L^1(\O), BMO(\O))_{1/p,p} = L^p(\O).
\end{equation}
The latter is true when $\O$ has a \emph{regular Vitali family}. See \cite{hanks}, \cite{riviere}.

There seems to be an easier way however. Note that in \ref{BIGTHM}, \ref{corollary1} and \ref{corollary2} we always studied operators of the form
\begin{equation}
a \mapsto \Ca(e^{-i\tau R}\chi_\O a) \quad \text{and} \quad f \mapsto \Ca\big( e^{-i\tau R}\chi_\O \Cab(e^{i\tau R}\chi_\O qf)\big).
\end{equation}
If we work in the conventional Sobolev spaces $W^{1,p}$, we have extension operators $E:W^{1,p}(\O) \to W^{1,p}_c(\C)$. Hence we may take a smooth test-function $\phi$, which is constant on $\O$, and consider the operators
\begin{equation}
a \mapsto \Ca(e^{-i\tau R}\phi a) \quad \text{and} \quad f \mapsto \Ca\big( e^{-i\tau R} \phi \Cab(e^{i\tau R} \phi Eq f)\big)
\end{equation}
instead. We construct Bukhgeim's solutions $u=e^{i\tau (z-z_0)^2}f$ as before. They won't be solutions to the Schr\"odinger equation in the whole plane, but they are so in every open set where $\phi \equiv 1$. Calculating with those operators is easier since there is no boundary terms when integrating by parts. See also the next section.

\subsection{Non-compactly supported potentials}
An obvious question is whether we can do all the steps in the proof for potentials supported on the whole domain $\R^2$. The first thing to do is to show the existence and norm estimates for the oscillating solutions. One way to do that is to replace $\chi_\O$ by a test function, integrate by parts as in the proof of theorem \ref{BIGTHM}, and then let it tend to the constant one pointwise.

If $\phi$ is a test function, we are able to prove
\begin{equation}
\norm{\Ca(e^{i\tau R} \phi a}_{L^p(\C)} \leq C (\norm{\phi}_\infty + \norm{\db \phi}_{2}) \tau^{-1/2} (\norm{a}_p + \norm{\db a}_{p^*}),
\end{equation}
where $\frac{1}{p^*} = \frac{1}{2} + \frac{1}{p}$ and $2 < p <\infty$. This is proven by taking $\chi \in C^\infty_0 (\C)$ which is constant one near the origin and letting
\begin{equation}
h(z) = \frac{1-\chi(\tau^{1/2}(z-z_0))}{\wbar{z} - \wbar{z_0}} \phi(z).
\end{equation}
Then proceed as in the proof of \ref{BIGTHM} and note that
\begin{equation}
\norm{\phi - (\wbar{z}-\wbar{z_0})h}_2 + \tau^{-1} \norm{h}_\infty + \tau^{-1} \norm{\db h}_2 \leq \tau^{-1/2}.
\end{equation}

Letting $\phi \to 1$ gives $\norm{\Ca(e^{i\tau R}a)}_p \leq C\tau^{-1/2} (\norm{a}_p + \norm{\db a}_{p^*})$. This in turn implies
\begin{multline}
\norm{ \Ca( e^{i\tau R} \Cab( e^{-i\tau R} q f) ) }_{L^p(\C)} \leq C \tau^{-1/2} (\norm{\Cab(e^{-i\tau R} qf)}_p + \norm{\wbar{\Pi}(e^{-i\tau R} qf)}_{p^*}) \\
\leq C \tau^{-1/2} \norm{qf}_{p^*} \leq C  \tau^{-1/2} \norm{q}_{L^{2}(\C)} \norm{f}_{L^p(\C)},
\end{multline}
since $\Cab : L^{p^*} \to L^p$ and the Beurling operator $\db\Cab$ \index{Beurling operator} is bounded on $L^{p^*}$. For the derivatives,
\begin{multline}
\norm{\nabla \Ca ( e^{i\tau R} \Cab ( e^{-i\tau R} qf))}_{L^p(\C)} \leq C \norm{\Cab(e^{-i\tau R}qf)}_{L^p(\C)} \\
\leq C \tau^{-1/2} ( \norm{qf}_p + \norm{f \db q}_{p^*} + \norm{q \db f}_{p^*}) \leq C \tau^{-1/2} \\
\leq C \tau^{-1/2} (\norm{q}_p + \norm{q}_2 + \norm{\db q}_2) (\norm{f}_\infty + \norm{f}_p + \norm{\db f}_p)\\
\leq C \tau^{-1/2} \norm{q}_{W^{1,2}(\C)} \norm{f}_{W^{1,p}(\C)}
\end{multline}
because $\nabla \Ca \cong (\operatorname{Id}, \Pi):L^p \to L^p$ and by Sobolev embedding. Thus we have shown the existence of Bukhgeim's oscillating solutions in the whole domain\index{solutions!Bukhgeim's!whole domain} if the potential is in $L^2$ or $W^{1,2}$, with error term vanishing at a rate of $\tau^{-1/2}$ in $L^p$ or $W^{1,p}$, respectively. This is not completely new, see for example \cite[ch. 3, 4]{guillarmouSaloTzou2011}.

We must handle the error term of the stationary phase integral next. The space $W^{1,p}$ is a Banach algebra, so our error term is
\begin{equation}
\int \tau e^{i\tau R} (q_1 - q_2) r_{\tau,z_0} dm(z)
\end{equation}
where $\sup_{z_0} \norm{r}_{1,p} \leq C \tau^{-1/2}$. Take a suitable smooth function $h$ cutting $z_0$ off. Split the integral by writing $1 = (1 - (\wbar{z} - \wbar{z_0})h) + (\wbar{z} - \wbar{z_0})h$. After integrating the second term by parts and using H\"older's inequality, we arrive at
\begin{multline}
\abs{ \int \tau e^{i\tau R} (q_1 - q_2) r dm} \leq \ldots\\
\leq ( \tau \norm{ 1 - (\wbar{z}-\wbar{z_0})h}_2 + \norm{\db h}_2 + \norm{h}_{{p^*}'}) \norm{q_1 - q_2}_{W^{1,2}} \norm{r}_{W^{1,p}} \\
\leq \tau^{\frac{2}{p+2} - \frac{1}{2}} \norm{q_1 - q_2}_{W^{1,2}} \longrightarrow 0
\end{multline}
as $\tau \to \infty$. Here ${p^*}'$ is the H\"older conjugate of $p^*$.

Assume that the potentials have a bit more smoothness and integrability, so that $\widehat{q_j} \in L^1$. Write $Q = q_1-q_2$. We get
\begin{multline}
\norm{q_1-q_2}_\infty \leq \norm{Q - \int_\O \frac{2\tau}{\pi} e^{i\tau R} Q dm}_\infty + \norm{\frac{2\tau}{\pi} \int_\O u^{(1)}(q_1-q_2) u^{(2)} dm}_\infty \\
+ \quad \norm{\frac{2\tau}{\pi} \int_\O e^{i\tau R}Q (1 - f_1 f_2) dm}_\infty \\
\leq \norm{ (1 - e^{i\frac{\xi^2 + \wbar{\xi}^2}{16\tau}}) \widehat{Q}}_1 +\norm{\frac{2\tau}{\pi} \int_\O u^{(1)}(q_1-q_2) u^{(2)} dm}_\infty + \tau^{\frac{2}{p+2} - \frac{1}{2}} \norm{Q}_{W^{1,2}} \\
\longrightarrow \lim_{\tau \to \infty} \norm{\frac{2\tau}{\pi} \int_\O u^{(1)}(q_1 - q_2) u^{(2)} dm}_\infty
\end{multline}
as $\tau \to \infty$ by dominated convergence and the previous deductions on the error term. Hence the inverse problem can be solved whenever $q_j \in W^{1,2}(\C)$, $\widehat{q_j} \in L^1(\C)$ and the measurements imply the orthogonality relation
\begin{equation}
\int u^{(1)}(q_1 - q_2)u^{(2)} dm = 0.
\end{equation}
Note that this does not follow from the equality of the corresponding scattering amplitudes. That would contradict the existence of certain counterexamples \index{counterexamples}in \cite{grinevichNovikov1995}.

\subsection{No smoothness}
Since we have stability for $q\in H^{s,(2,1)}(\O)$ with any $s>0$, it is tempting to try to prove uniqueness for $q\in L^{(2,1)}(\O)$. We basically have to estimate three terms, just like in the sketch of section \ref{sketchSection}:
\begin{multline}
\norm{q_1 - q_2 - \int_\O \frac{2\tau}{\pi} e^{i\tau R} (q_1-q_2) dm}, \\
\norm{\frac{2\tau}{\pi} \int_\O u^{(1)}(q_1 - q_2) u^{(2)} dm}, \\
\text{and} \quad \norm{\frac{2\tau}{\pi} \int_\O e^{i\tau R}(q_1 - q_2) (1 - f_1 f_2) dm}.
\end{multline}
The first term is not a problem. We can use dominated convergence on the Fourier-side to see that that term vanishes as $\tau$ grows. The second term vanishes if we assume that $C_{q_1} = C_{q_2}$. 

There's only the third term left. There are two choices. One is to try to show that $\norm{1-f_1f_2} = o(\tau^{-1})$, which seems unlikely. Maybe some numerical simulations could shed a better light on this? The other option is to try to study how $f_1f_2$ behaves with respect to $z_0$, and see what kind of operator we have here. If $f_j$ didn't depend on $z_0$, then it would be a convolution operator, and we could show that the term vanishes using stationary phase. But there is dependence, so something nontrivial has to be done.

Actually, the problem seems to have been solved after the writing of this thesis. Imanuvilov and Yamamoto published the paper \cite{imanuvilovYamamotoLp} in arXiv, and it claims to show uniqueness for $q_j \in L^p(\O)$, $p>2$. They approximate the potentials by smooth functions and use results for oscillatory integral operators.\index{non-smooth potential|textbf}

\subsection{A reconstruction formula}\index{reconstruction|textbf}
There is a reconstruction formula by Bukhgeim in the last few lines of his article \cite{bukhgeim}. The idea is as follows. We have
\begin{equation}
q(z_0) \longleftarrow \int_\O \frac{2\tau}{\pi} e^{i\tau R} q(z) dm(z)
\end{equation}
by stationary phase\index{stationary phase}. Let $u = e^{i\tau(z-z_0)^2}f$ be Bukhgeim's oscillating solution to $\Delta u + q u = 0$. Then $e^{i\tau(z-z_0)^2} = u + e^{i\tau(z-z_0)^2}(1-f)$, so
\begin{equation}
q(z_0) \longleftarrow \int_\O \frac{2\tau}{\pi} e^{i\tau(\wbar{z}-\wbar{z_0})^2} q u (z) dm(z) + \int_\O \frac{2\tau}{\pi} e^{i\tau R} q(1-f)(z) dm(z)
\end{equation}
The second term tends to zero according to theorem \ref{errorTermIntegral}. Use the fact that $qu = - \Delta u = -4\d \db u$ and integrate by parts. Note that $\d e^{i\tau(\wbar{z}-\wbar{z_0})^2} = 0$. Hence
\begin{equation}
q(z_0) \longleftarrow \int_{\d\O} \frac{4\tau}{\pi} \eta(z) e^{i\tau (\wbar{z}-\wbar{z_0})^2} \db u(z) d\sigma(z)
\end{equation}
After that, the idea is to reconstruct $\db u$ for Bukhgeim's solutions when we only know the boundary data. We haven't used all degrees of freedom when constructing $u$. This is because $\Ca$ is only a right inverse of $\db$. We can always add an analytic function, and it is completely determined by its boundary values. Hence, we should look for ways to choose the analytic functions so that we may set some boundary values of $u$ independently of $q$, and observe $\db u$. 

According to \cite{bukhgeim}, we may set $\Re u$ and $\Re \db u$, and observe $\Im u$ and $\Im \db u$. This is because the real part of a complex analytic function determines its imaginary part apart from a constant. In any case, this is just a theoretical tool for now. Any noise in the measurement of $\db u$ would get amplified and oscillated exponentially. For a recent result with an explicit boundary integral equation\index{boundary integral equation} for $u_{\d\O}$, see \cite{novikovSantacesariaRec}.

\vfill
\section{Calculations}

\begin{lemmaSect}
\label{gaussianFourier}
If $c > 0$ then the Fourier transform of $t \mapsto e^{-c t^2/2}$ is the mapping $\xi \mapsto \frac{1}{\sqrt{c}} e^{-\xi^2/(2c)}$.
\end{lemmaSect}
\begin{proof}
This is a direct calculation using Cauchy's integral theorem. Let $c>0$ and $\xi \in \R$. Then
\begin{equation}
\begin{split}
\big( e^{-c \frac{t^2}{2}} \big)^{\wedge} (\xi) &= \frac{1}{\sqrt{2\pi}} \int_{-\infty}^\infty \!\!\! e^{-c \frac{t^2}{2} - i \xi t} dt = \frac{1}{\sqrt{2\pi}} e^{-\frac{\xi^2}{2c}} \int_{-\infty}^\infty \!\!\! e^{-\big( \sqrt{\frac{c}{2}} t + \frac{i \xi}{2\sqrt{c/2}} \big)^2} \! dt \\
&= \frac{1}{\pi c} e^{-\frac{\xi^2}{2c}} \int_{-\infty}^\infty \!\!\! e^{-\big( s + \frac{i \xi}{2\sqrt{c/2}} \big)^2} \! ds = \frac{1}{\sqrt{\pi c}} e^{-\frac{\xi^2}{2c}} \int_{-\infty}^\infty e^{-u^2} du \\
&= \frac{1}{\sqrt{c}} e^{-\frac{\xi^2}{2c}}
\end{split}
\end{equation}
by Cauchy's integral theorem. This is justified because the function given by $z \mapsto e^{-z^2}$ is analytic and for any $A \in \R$ we have
\begin{equation}
\Big\lvert \int_s^{s+i A} \!\!\! e^{-z^2} dz \Big\rvert \leq \int_s^{s+iA} \!\!\! \lvert e^{-z^2} \rvert d\sigma(z) = \int_s^{s+iA} \!\!\! e^{A^2 - s^2} d\sigma(z) = \lvert A \rvert e^{A^2 - s^2},
\end{equation}
which tends to zero when $s \longrightarrow \infty$ or $s \longrightarrow -\infty$ along the real line.
\end{proof}

\begin{lemmaSect}
\label{convKernTransf}
\index{complex Gaussian}
Let $\tau > 0$ and define $\kappa_\tau:\C \to \C$ by $\kappa_\tau (z) = \frac{2\tau}{\pi} e^{i\tau (z^2 + \wbar{z}^2)}$. It is a tempered distribution and
\begin{equation}
	\widehat{\kappa_\tau}(\xi) = \frac{1}{2\pi}e^{-i\frac{\xi^2 + \wbar{\xi}^2}{16\tau}}.
\end{equation}
\end{lemmaSect}
\begin{proof}
The function $\kappa_\tau$ is bounded and measurable so it is a tempered distribution. Let $\varphi \in \mathscr{S} (\R)$ be a Schwartz test function. Note that by lemma \ref{gaussianFourier}
\begin{equation}
\int_\R e^{-c \frac{t^2}{2}} \what{\varphi}(t) dt = \int_\R \frac{1}{\sqrt{c}} e^{-\frac{\xi^2}{2c}} \varphi(\xi) d\xi.
\end{equation}
Let us choose a branch of the square root in the complex plane such that $\operatorname{arg} \sqrt{z} \in \left] -\frac{\pi}{2}, \frac{\pi}{2} \right]$. Now both sides of the previous equation are analytic functions of $c$ in the right half-plane $\Re (c) > 0$. Hence they are equal also in the whole right half-plane. 

Let $\phi, \psi \in \mathscr{S}(\R)$ be two Schwartz test functions. By Fubini's theorem and dominated convergence
\begin{multline}
\int_{\C} e^{i\tau(z^2 + \wbar{z}^2)} \big(\phi(\xi_1)\psi(\xi_2)\big)^\wedge(z) dm(z) = \int_{\R^2} e^{i2\tau(x^2 - y^2)} \what{\phi}(x)\what{\psi}(y) dm(x,y) \\
= \int_{-\infty}^\infty e^{i 2\tau x^2} \what{\phi}(x) dx \int_{-\infty}^\infty e^{-i 2\tau y^2} \what{\psi}(y) dy \\
= \lim_{\epsilon \to 0+} \int_{-\infty}^\infty e^{-(2\epsilon - i4\tau) \frac{x^2}{2}} \what{\phi}(x) dx \int_{-\infty}^\infty e^{-(2\epsilon + i4\tau) \frac{y^2}{2}} \what{\psi}(y) dy \\
= \lim_{\epsilon\to 0+} \frac{1}{\sqrt{2\epsilon - i 4\tau}}\int_{-\infty}^\infty \!\!\!\!\!\! e^{-\frac{\xi_1^2}{4\epsilon - i8\tau}} \phi(\xi_1) d\xi_1 \frac{1}{\sqrt{2\epsilon + i 4\tau}} \int_{-\infty}^\infty \!\!\!\!\!\! e^{-\frac{\xi_2^2}{4\epsilon + i8\tau}} \psi(\xi_2) d\xi_2 \\
= \frac{1}{\sqrt{4\tau}}e^{i\frac{\pi}{4} \tau} \frac{1}{\sqrt{4\tau}} e^{-i\frac{\pi}{4} \tau} \int_{-\infty}^\infty e^{\frac{\xi_1^2 - \xi_2^2}{i8\tau}}\phi(\xi_1)\psi(\xi_2) dm(\xi_1,\xi_2) \\
= \frac{1}{4\tau} \int_\C e^{-i \frac{\xi^2 + \wbar{\xi}^2}{16\tau}} \phi(\xi_1) \psi(\xi_2) dm(\xi),
\end{multline}
so $\mathscr{F}\big\{ \frac{2\tau}{\pi} e^{i\tau(z^2 + \wbar{z}^2)} \big\}(\xi) = \frac{1}{2\pi} e^{-i\frac{\xi^2 + \wbar{\xi}^2}{16\tau}}$.
\end{proof}

\begin{lemmaSect}
\label{shapeToAnnulus}
Let $\O \subset \C$ be open, $f:\R_+ \to \R_+$ non-increasing and $\epsilon \geq 0$. Then
\begin{equation}
\int_{\O\setminus B(0,\epsilon)} f(\abs{z}) dm(z) \leq \int_{\{\epsilon \leq \abs{z} \leq \sqrt{\frac{m(\O)}{\pi} + \epsilon^2}\}} f(\abs{z})dm(z).
\end{equation}
\end{lemmaSect}
\begin{proof}
The idea is that the integral on the left-hand side is maximized when $\O$ is an annulus around the ball.
\smallskip

\noindent
\begin{minipage}[b]{0.595\textwidth}
Let $r = \sqrt{m(\O)\pi^{-1} + \epsilon^2}$ denote the outer radius of this annulus. Write
{\begin{equation*}
\begin{split}
\O_r &= (\O \setminus B(0,\epsilon)) \setminus B(0,r),\\
\O_\cap &= (\O \setminus B(0,\epsilon)) \cap B(0,r), \\
A_r &= B(0,r)\setminus B(0,\epsilon) \setminus \O.
\end{split}
\end{equation*}} Note that
\end{minipage}
\begin{minipage}[b]{0.405\textwidth}
\begin{tikzpicture}
\def\Vc{1,1} %center of the circles
\def\Vr{0.65} %radius of small circle
\def\VR{1.4} %radius of big circle

\def\borderPath{(0,0) .. controls (-0.5,0) and (-1.6,0) .. (-1,1) %the border of Omega
		-- (-1,1) .. node[anchor=east, fill opacity=1] {$\Omega$} controls (-0.4,2) and (-0.5,2.5) .. (0,3)
		-- (0,3) .. controls (0.5,3.5) and (2,2) .. (2,1)
		-- (2,1) .. controls (2,0.8) and (1.5,1.9) .. (1,1.7)
		-- (1,1.7) .. controls (0.5,1.7) and (1,0) .. (0,0)
		-- cycle}
\def\isoCirc{(\Vc) circle (\VR)} %the border of the big circle
\def\pieniCirc{(\Vc) circle (\Vr)} %the border of the small circle

\tikzstyle{borderStyle} = [very thick, dashed]
\tikzstyle{OrStyle} = [pattern=grid]%[fill=red]%
\tikzstyle{OcapStyle} = [pattern = crosshatch]%[fill=yellow]%
\tikzstyle{AnnulusStyle} = [pattern=dots]%[fill=green]%

	\fill [OrStyle] \borderPath;
	\fill [AnnulusStyle] \isoCirc;
	\begin{scope}
		\clip \borderPath;
		\fill [OcapStyle] \isoCirc;
	\end{scope}
	\fill [fill=white] \pieniCirc;

	\draw \isoCirc;
	\draw \pieniCirc;
	\draw [borderStyle] \borderPath;	

	\draw (\Vc) node[below=-2pt]{$0$};
	\draw (\Vc) -- node[above=-3pt]{$\epsilon$} ++(0:\Vr);
	\draw (\Vc) -- node[anchor=north]{$r$} ++(-45:\VR);

	\begin{scope}
	\newcommand{\Vnd}{0.6,0}
	\draw (3,3) node [OrStyle] {$\phantom\Omega$}
		++(\Vnd) node {$\Omega_r$};
	\draw (3,2.3) node [OcapStyle] {$\phantom\Omega$}
		node [AnnulusStyle] {$\phantom\Omega$}
		node [OrStyle] {$\phantom\Omega$}
		node [OcapStyle] {$\phantom\Omega$}
		++(\Vnd) node {$\Omega_\cap$};
	\draw (3,1.6) node [AnnulusStyle] {$\phantom\Omega$}
		++(\Vnd) node {$A_r$};
	\draw (3,0.9) node [borderStyle] {\tikz \draw [borderStyle] (0,0) -> (0.5,0.5);}
		++(\Vnd) node {$\partial\O$};
	\end{scope}
\end{tikzpicture}
\end{minipage}

\begin{multline}
m(\O_r) = m(\O\setminus B(0,\epsilon)) - m(\O_\cap) \leq m(\O) - m(\O_\cap) \\
= m(B(0,r)\setminus B(0,\epsilon)) - m(\O_\cap) \\
= m(B(0,r)\setminus B(0,\epsilon) \setminus \O_\cap) = m(A_r),
\end{multline}
and $\sup_{\O_r} f(\abs{z}) \leq \inf_{A_r} f(\abs{z})$ because if $z\in\O_r$ then $\abs{z} \geq r$, but it is smaller than $r$ in $A_r$. Hence
\begin{multline}
\int_{\O\setminus B(0,\epsilon)} f(\abs{z}) dm(z) = \int_{\O_\cap}f(\abs{z}) dm(z) + \int_{\O_r} f(\abs{z}) dm(z) \\
\leq \int_{\O_\cap} f(\abs{z}) dm(z) + \int_{A_r} f(\abs{z}) dm(z) = \int_{B(0,r)\setminus B(0,\epsilon)} f(\abs{z})dm(z).
\end{multline}
\end{proof}

\begin{lemmaSect}
\label{h1}\index{cut-off function}
Let $\O\subset \C$ be a bounded Lipschitz domain, $z_0\in\C$ and $\tau > 0$. Then there exists $h\in W^{1,1}(\O)$ such that
\begin{equation}
\tau\norm{1-(\wbar{z}-\wbar{z_0})h}_{L^1(\O)} + \norm{h}_{W^{1,1}(\O)} \leq 2\pi \left( \sqrt{\frac{m(\O)}{\pi}} + 2 + \ln_+ \sqrt{\frac{m(\O)\tau}{\pi}} \right).
\end{equation}
\end{lemmaSect}
\begin{proof}
Write $\epsilon = \tau^{-1/2}$. Define
\begin{equation}
H(z) = \begin{cases}
\frac{1}{\,\wbar{z}\,} &, \abs{z} > \epsilon\\
\frac{\abs{z}}{\epsilon \wbar{z}} &, \abs{z}<\epsilon
\end{cases}.
\end{equation}
Note that $H \in L^1_{loc} \subset \mathscr{D}'$. A straightforward integration by parts against a test function gives us
\begin{equation}
\d H(z) = \begin{cases}
0 &, \abs{z} > \epsilon\\
\frac{1}{2\epsilon\abs{z}} &, \abs{z}<\epsilon
\end{cases} ,
\qquad
\db H(z) = \begin{cases}
\frac{-1}{\wbar{z}^2} &, \abs{z} > \epsilon\\
\frac{-\abs{z}}{2\epsilon \wbar{z}^2} &, \abs{z}<\epsilon
\end{cases}.
\end{equation}
Let $h(z) = H(z-z_0)$. Now $h \in W^{1,1}(\O)$ because $\O$ is bounded. The rest is straightforward calculations using lemma \ref{shapeToAnnulus}. We will show the hardest case, namely $\db h$. By the lemma, we have
\begin{equation}
\begin{split}
\int_\O &\abs{\db H(z-z_0)} dm(z) \leq \int_{\{\abs{z} \leq \sqrt{\frac{m(\O)}{\pi}}\}} \abs{\db H(z)} dm(z) \\
&= 2\pi \int_0^{\sqrt{\frac{m(\O)}{\pi}}} \abs{\db H (r)} r dr = 2\pi \int_0^\epsilon \frac{dr}{2\epsilon} + 2\pi \int_{\epsilon}^{\sqrt{\frac{m(\O)}{\pi}}} \frac{dr}{r} \\
&= \pi + 2\pi \ln \sqrt{\frac{m(\O)}{\pi\epsilon^2}},
\end{split}
\end{equation}
when $\epsilon < \sqrt{\frac{m(\O)}{\pi}}$. When not, then the upper bound is just $\pi$. This justifies $\ln_+$ in the estimate.

To estimate the first term, it is enough to note that
\begin{equation}
1-\wbar{z}H(z) = \begin{cases}
0&, \abs{z}>\epsilon\\
1-\abs{z}\epsilon^{-1} &, \abs{z} < \epsilon
\end{cases}
\end{equation}
and integrate over the ball $B(z_0,\epsilon)$. This gives us
\begin{equation}
\norm{1-(\wbar{z}-\wbar{z_0})h}_{L^1(\O)} \leq \tfrac{\pi}{3} \epsilon^2.
\end{equation}
Summing all the terms and estimating upwards a bit gives the estimate.
\end{proof}
\begin{remarkSect}
This works also for unbounded domains with finite measure.
\end{remarkSect}

\begin{lemmaSect}
\label{kernelCalc}
Let $\O\subset \C$ be a bounded open set, $z_0\in \C$ and $\tau > 0$. Write $\O_\tau = \O \setminus B(z_0,\tau)$. Then
\begin{equation}
\begin{split}
&\norm{(\wbar{z}-\wbar{z_0})^{-1}}_{L^{(2,1)}(\O_\tau)} \leq 2\sqrt{\pi} \ln \Big(\frac{2m(\O)}{\pi\tau^2} + 1 + \sqrt{\big(\tfrac{2m(\O)}{\pi\tau^2} + 1\big)^2 - 1}\Big), \\
&\norm{(\wbar{z}-\wbar{z_0})^{-2}}_{L^{(2,1)}(\O_\tau)} \leq 4\sqrt{\pi} \tau^{-1} \arctan \sqrt{\frac{m(\O)}{\pi\tau^2}}.
\end{split}
\end{equation}
\end{lemmaSect}
\begin{proof}
We may assume that $z_0 = 0$ by the translation invariance of $L^{(2,1)}$ and $\O \mapsto m(\O)$. First, calculate
\begin{multline}
m(\abs{z}^{-a},\lambda) = m\{ z\in \O_\tau \mid \abs{z}^{-a} > \lambda \} = m\{\O\setminus B(0,\tau) \cap B(0,\lambda^{-1/a})\}\\
\leq
\begin{cases}
0, \qquad \lambda^{-1/a}\leq \tau \Leftrightarrow \tau^{-a} \leq \lambda \\
\pi(\lambda^{-2/a} - \tau^2), \quad \tau < \lambda^{-1/a} \leq \sqrt{\tfrac{m(\O)}{\pi} + \tau^2} \Leftrightarrow \sqrt{\tfrac{m(\O)}{\pi} + \tau^2}^{-a} \leq \lambda < \tau^{-a} \\
m(\O), \qquad \sqrt{\tfrac{m(\O)}{\pi} + \tau^2} < \lambda^{-1/a} \Leftrightarrow \lambda < \sqrt{\tfrac{m(\O)}{\pi} + \tau^2}^{-a}
\end{cases}
\end{multline}
then
\begin{equation}
\big(\abs{z}^{-a}\big)^*(s) = \inf \{ \lambda \mid m(\abs{z}^{-a},\lambda) \leq s \} = 
\begin{cases}
\tau^{-a}, \qquad s = 0\\
\sqrt{\tfrac{s}{\pi} + \tau^2}^{-a}, \quad 0<s<m(\O)\\
0, \qquad m(\O) \leq s
\end{cases}
\end{equation}
By H\"older's inequality $\norm{f}_{(2,1)} \leq 2 \norm{f}_{2,1} = 2\int_0^\infty s^{-1/2}f^*(s) ds$, so
\begin{equation}
\norm{\abs{z}^{-a}}_{L^{(2,1)}(\O_\tau)} \leq 2 \int_0^\infty s^{-1/2}(\abs{z}^{-a})^*(s)ds = 2 \int_0^{m(\O)} s^{-1/2}\big(\tfrac{s}{\pi} + \tau^2\big)^{-a/2} ds.
\end{equation}

\underline{Case $a=1$:} We have $D \ln(x+\sqrt{x^2-1}) = (x^2-1)^{-1/2}$ for $x > 1$. Use the change of variables $u = \tfrac{2s}{\pi\tau^2} + 1$ to get
\begin{multline}
\int_0^{m(\O)} s^{-1/2}\big(\tfrac{s}{\pi} + \tau^2\big)^{-1/2} ds = \sqrt{\pi} \frac{2}{\pi\tau^2} \int_0^{m(\O)} \frac{ds}{\sqrt{(\frac{2s}{\pi\tau^2}+1)^2 - 1}} \\
= \sqrt{\pi} \int_1^{\frac{2m(\O)}{\pi\tau^2} +1} \frac{du}{\sqrt{u^2-1}} = \sqrt{\pi} \ln \Big(\frac{2m(\O)}{\pi\tau^2} + 1 + \sqrt{\big(\tfrac{2m(\O)}{\pi\tau^2} + 1\big)^2 - 1}\Big).
\end{multline}

\underline{Case $a=2$:} Use $D \arctan v = (v^2+1)^{-1}$ with the change of variables $v = \frac{u}{\sqrt{\pi}\tau}$ and $u = s^{1/2}$. Then
\begin{multline}
\int_0^{m(\O)} s^{-1/2} \big(\tfrac{s}{\pi} + \tau^2\big)^{-1} ds = 2\pi \int_0^{\sqrt{m(\O)}} \frac{du}{u^2 + \pi\tau^2} = 2\pi \int_0^{\sqrt{\frac{m(\O)}{\pi\tau^2}}} \frac{\sqrt{\pi}\tau dv}{\pi\tau^2(v^2+1)} \\
= 2\sqrt{\pi} \tau^{-1} \int_0^{\sqrt{\frac{m(\O)}{\pi\tau^2}}} \frac{dv}{v^2+1} = 2\sqrt{\pi} \tau^{-1} \arctan\sqrt{\frac{m(\O)}{\pi\tau^2}}.
\end{multline}
\end{proof}

\begin{lemmaSect}
\label{beltMeasure}\index{e@$\varepsilon$-neighborhood}
Let $\O\subset \R^n$ be a bounded Lipschitz domain. Write
\begin{equation}
\index[notation]{dOepsilon@$\d\O_\varepsilon$; $\varepsilon$-neighborhood of $\d\O$}
\d \O_\varepsilon = \{x \in \O \mid d(x,\d \O) < \varepsilon\}.
\end{equation}
Then there is $C_\O < \infty$ such that
\begin{equation}
m(\d \O_\varepsilon) \leq C_\O \varepsilon
\end{equation}
for any $\varepsilon \geq 0$.
\end{lemmaSect}
\begin{proof}
Let $\{U_j\}$ be a finite open cover of $\d \O$ such that each $U_j$ is a cube and there exists Lipschitz functions $f_j$ and orthonormal coordinate systems $(\zeta_{j,1}, \ldots, \zeta_{j,n})$ in $U_j$ such that
\begin{equation}
\begin{split}
U_j &= \{ \zeta_j \mid -s_j < \zeta_{j,k} < s_j \text{ for all }k \} \\
\O \cap U_j &= \{ \zeta_j \mid \zeta_{j,n} < f(\zeta_{j,1}, \ldots, \zeta_{j,n-1}) \}.
\end{split}
\end{equation}
The existence of such an atlas follows by first taking an arbitrary open cover, the functions $f_j$ and the coordinate axes, which are all given by the definition of a Lipschitz domain. Then cover each point in $\wbar{\O}\cap U_j$ by a properly oriented cube $K \subset U_j$. The compactness of $\wbar{\O}$ implies the rest.

From now on we will write $\check{x} = (x_1,\ldots,x_{n-1})$ for $x\in\R^n$. Using the given coordinate systems we will see that
\begin{multline}
\d \O_\varepsilon \cap U_j = \{ \zeta_j \mid \inf_{\check{\xi_j}} \abs{(\check{\xi_j},f_j(\check{\xi_j})) - \zeta_j} < \varepsilon \} \\
\subset \{ \zeta_j \mid \abs{\big(\check{\zeta_j}, f_j(\check{\zeta_j})\big) - \zeta_j} < (M+1)\varepsilon \},
\end{multline}
where $M = \max \norm{f_j}_{C^{0,1}}$. This is true because of the following reasoning: if $\zeta_j \in \d \O_\varepsilon \cap U_j$ then there is $\xi_j \in \d \O \cap U_j$ such that $\abs{\zeta_j - \xi_j} < \varepsilon$, and
\begin{multline}
\abs{\big(\check{\zeta_j}, f_j(\check{\zeta_j})\big) - \zeta_j} \leq \abs{\zeta_j - (\check{\zeta_j}, \xi_{j,n})} + \abs{(\check{\zeta_j}, \xi_{j,n}) - \big(\check{\zeta_j}, f_j(\check{\zeta_j})\big)} \\
= \abs{\zeta_{j,n} - \xi_{j,n}} + \abs{\xi_{j,n} - f_j(\check{\zeta_j})} \leq \abs{\zeta_j - \xi_j} + \abs{f(\check{\xi_j}) - f_j(\check{\zeta_j})} \\
\leq (1+M) \abs{\zeta_j - \xi_j} \leq (1+M) \varepsilon
\end{multline}
because $\xi_j \in \d \O \cap U_j$ implies $\xi_{j,n} = f(\check{\xi_j})$.

We are almost done. Let $\varepsilon_0 > 0$ be such that $\d \O_\varepsilon \subset \cup U_j$ if $0< \varepsilon < \varepsilon_0$. If $\varepsilon \geq \varepsilon_0$, then $m(\d \O_\varepsilon) \leq m(\O) \leq \frac{m(\O)}{\varepsilon_0} \varepsilon$. The constant in front depends only on $\O$ since the choice of the atlas doesn't depend on $\varepsilon$. Hence we may assume that $\d \O_\varepsilon \subset \cup U_j$.

Now 
\begin{multline}
m(\d \O_\varepsilon \cap U_j) \leq m\{ \zeta_j \mid \abs{f_j (\check{\zeta_j}) - \zeta_{j,n}} < (M+1)\varepsilon \} \\
\leq m\{ \zeta_j \mid -s_j < \zeta_{j,k} < s_j \text{ for all } k \neq n, \text{ and } \abs{f_j(\check{\zeta_j}) - \zeta_{j,n}} < (M+1)\varepsilon \} \\
= (2s_j)^{n-1}(M+1)\varepsilon.
\end{multline}
The cover is finite, so summing all the pieces gives $m(\d \O_\varepsilon) \leq C_\O \varepsilon$.
\end{proof}

\begin{lemmaSect}
\label{h2first}\index{cut-off function}
Let $\O \subset \C$ be bounded and Lipshitz. Then there is $C_\O < \infty$ such that for all $z_0 \in \C$ and $\epsilon, \delta > 0$ there exists $h \in C^\infty_0(\O)$ with
\begin{equation}
\begin{split}
&\norm{1-(\wbar{z}-\wbar{z_0})h}_{L^{(2,1)}(\O)} \leq C_\O \sqrt{\delta^2 + \epsilon},\\
&\norm{h}_{L^\infty(\O)} \leq \delta^{-1},\\
&\norm{\db h}_{L^{(2,1)}(\O)} \leq C_\O \left( \epsilon^{-1} \ln \left( 1 + C_\O \epsilon\delta^{-2} + \sqrt{(1+C_\O \epsilon\delta^{-2})^2-1}\right) + \delta^{-1}\right),
\end{split}
\end{equation}
\end{lemmaSect}
\begin{remarkSect}
$\epsilon$ describes how close the support of $h$ is to $\d \O$, while $\delta$ tells how close it is to $z_0$.
\end{remarkSect}
\begin{proof}
Let $\phi \in C^\infty_0(\C)$ be such that $\supp \phi \subset B(0,1)$, $0 \leq \phi$, $\int \phi = 1$. Denote
\begin{equation}
\begin{split}
\chi_\epsilon &= \chi_{\{z \in \O \mid d(z,\d \O) > 2\epsilon \}} \ast \phi_\epsilon \\
\chi^\delta &= \chi_{\C \setminus B(z_0,2\delta)} \ast \phi_\delta,
\end{split}
\end{equation}
where $\phi_a(z) = a^{-2}\phi(z/a)$. It is clear that $\chi_\epsilon \in C^\infty_0(\O)$, $\chi^\delta \in C^\infty(\C)$, $0 \leq \chi_\epsilon, \chi^\delta \leq 1$,
\begin{equation}
\begin{split}
\chi_\epsilon(z) = 0 \text{ if } d(z,\d \O) < \epsilon, &\qquad \chi^\delta(z) = 0 \text{ if } \abs{z-z_0} < \delta \\
\chi_\epsilon(z) = 1 \text{ if } d(z,\d \O) > 3\epsilon, &\qquad \chi^\delta(z) = 1 \text{ if } \abs{z-z_0} > 3\delta
\end{split}
\end{equation}
and
\begin{multline}
\abs{\d_j \chi_\epsilon (z) } = \abs{\d_j \int_{\{z \in \O \mid d(z,\d \O) > 2\epsilon \}} \!\!\!\!\!\!\!\!\!\!\!\!\!\!\!\!\!\!\!\!\!\!\!\!\!\!\!\!\!\!\!\!\! \epsilon^{-2}\phi(\tfrac{z-z'}{\epsilon}) dm(z')} = \abs{\epsilon^{-1} \int_{\{z \in \O \mid d(z,\d \O) > 2\epsilon \}} \!\!\!\!\!\!\!\!\!\!\!\!\!\!\!\!\!\!\!\!\!\!\!\!\!\!\!\!\!\!\!\!\! \epsilon^{-2}(\d_j\phi)(\tfrac{z-z'}{\epsilon}) dm(z')} \\
\leq \epsilon^{-1} \int_{\C} \abs{\d_j\phi}(w) dm(w) = \epsilon^{-1} \norm{\d_j \phi}_1.
\end{multline}
Similarly, $\abs{\d_j \chi^\delta(z)} \leq \delta^{-1} \norm{\d_j \phi}_1$. Finally, set
\begin{equation}
h(z) = \frac{\chi_\epsilon (z) \chi^\delta(z)}{\wbar{z}-\wbar{z_0}}.
\end{equation}
This is a compactly supported test function in $\O$ because $z_0 \notin \supp \chi^\delta$ and $\chi_\epsilon \in C^\infty_0(\O)$. The estimate for $\norm{h}_\infty$ comes directly from the support of $\chi^\delta$. For the first estimate, note that if $A \subset \C$ is measurable, then $\norm{\chi_A}_{(p,q)} = \big(\tfrac{p}{q} + \tfrac{p}{q(p-1)}\big)^{1/q} \big(m(A)\big)^{1/p}$ for $p > 1$, $q < \infty$. Hence we have
\begin{multline}
\norm{1 - (\wbar{z}-\wbar{z_0})h}_{L^{(2,1)}(\O)} = \norm{1-\chi_\epsilon\chi^\delta}_{(2,1)} = \norm{\chi_{\{\abs{z-z_0}<3\delta \text{ or } d(z,\d \O) < 3\epsilon\}}}_{(2,1)} \\
\leq 4 \sqrt{m\{\abs{z-z_0}<3\delta \text{ or } d(z,\d \O) < 3\epsilon\}} \leq 4\sqrt{9\pi\delta^2 + m(\d \O_{3\epsilon})} \\
\leq C_\O \sqrt{\delta^2 + \epsilon}
\end{multline}
by lemma \ref{beltMeasure}.

Note that $\norm{fg}_{(2,1)} \leq \norm{f}_{(2,1)} \norm{g}_\infty$. Use the estimates for $\d_j \chi_\epsilon$ and $\d_j \chi^\delta$, and keep track of the sets where $\chi_\epsilon, \chi^\delta$ are constants to get
\begin{multline}
\norm{\db h}_{L^{(2,1)}(\O)} \leq \norm{ \frac{\chi^\delta}{\wbar{z}-\wbar{z_0}} \db \chi_\epsilon}_{(2,1)} + \norm{\frac{\chi_\epsilon}{\wbar{z}-\wbar{z_0}} \db \chi^\delta}_{(2,1)} + \norm{\frac{\chi_\epsilon \chi^\delta}{(\wbar{z}-\wbar{z_0})^2}}_{(2,1)} \\
\leq \norm{\frac{1}{\wbar{z}-\wbar{z_0}}}_{L^{(2,1)}(\d \O_{3\epsilon} \setminus B(z_0,\delta))} \!\!\!\!\!\!\!\!\!\!\!\!\!\!\!\!\!\!\!\!\!\!\!\! \epsilon^{-1} \norm{\db \phi}_1 + \norm{\frac{1}{\wbar{z}-\wbar{z_0}}}_{L^{(2,1)}(B(z_0,3\delta)\setminus B(z_0,\delta))}  \!\!\!\!\!\!\!\!\!\!\!\!\!\!\!\!\!\!\!\!\!\!\!\! \delta^{-1} \norm{\db\phi}_1 \\
+ \norm{\frac{1}{(\wbar{z}-\wbar{z_0})^2}}_{L^{(2,1)}(\O\setminus B(z_0,\delta))} \\
\end{multline}
Again, by lemma \ref{beltMeasure}, we have $m(\d \O_{3\epsilon}) \leq C_\O \epsilon$, and $m(B(z_0,3\epsilon)\setminus B(z_0,\epsilon)) = 8\pi \delta^2$. The rest follows directly from lemma \ref{kernelCalc}.
\end{proof}

\begin{corollarySect}
\label{h2}
Let $\O \subset \C$ be a bounded Lipschitz domain. Then, for every $\tau \geq 1$ and $z_0 \in \C$, there exists $h \in C^\infty_0(\O)$ such that
\begin{equation}
\tau \norm{1 - (\wbar{z}-\wbar{z_0})h(z)}_{L^{(2,1)}(\O)} + \norm{h}_{L^\infty(\O)} + \norm{\db h}_{L^{(2,1)}(\O)} \leq C_\O \tau^{2/3},
\end{equation}
where $C_\O$ does not depend on $\tau$ or $z_0$.
\end{corollarySect}
\begin{proof}
Take $h$ as in lemma \ref{h2first}, and choose $\epsilon = \delta^2 = \tau^{-2/3}$. This gives us
\begin{equation}
C_\O (\tau^{2/3} + \tau^{1/3}) \leq 2C_\O \tau^{2/3}
\end{equation}
for the right-hand side.
\end{proof}

\index{orthogonality relation|see{Alessandrini's identity}}
\index{Dirichlet-Neumann map|seealso{boundary data}}
\index{ill-posed|see{counterexamples}}
\index{Green's formula|see{integration by parts}}
\index{stability|see{inverse problem solution}}
\index{Cauchy data|see{boundary data}}

\newpage
\bibliographystyle{plain}
\bibliography{omaBib}

\newpage
\clearpage
\printindex[notation]

\newpage
\clearpage
\printindex

\end{document}